\ifdefined\useaosstyle
\def\compareprevious{1}
\fi

\ifdefined\useaosstyle
\documentclass[dvipdf,aos,preprint]{imsart}
\RequirePackage[OT1]{fontenc}
\RequirePackage{amsthm,amsmath}
\RequirePackage[numbers]{natbib}
\RequirePackage[colorlinks,citecolor=blue,urlcolor=blue]{hyperref}
\arxiv{arXiv:1302.3203}
\else
\documentclass[11pt]{article}
\usepackage{fullpage}
\usepackage[numbers]{natbib}
\usepackage{hyperref}
\fi

\usepackage[svgnames]{xcolor}
\usepackage{overpic}

\hypersetup{linkcolor=MediumSlateBlue,urlcolor=Olive}

\usepackage{rate-distortion}
\usepackage{color}
\usepackage{bbm}
\usepackage{psfrag}
\usepackage{xifthen} 
\usepackage{xr}

\ifdefined\useaosstyle  
\fi

\newcommand{\paragraphc}[1]{\paragraph{{#1}:}}

\ifdefined\useaosstyle
\renewcommand{\paragraphc}[1]{\paragraph{{#1}}}
\else
\makeatletter
\long\def\@makecaption#1#2{
  \vskip 0.8ex
  \setbox\@tempboxa\hbox{\small {\bf #1:} #2}
  \parindent 1.5em  
  \dimen0=\hsize
  \advance\dimen0 by -3em
  \ifdim \wd\@tempboxa >\dimen0
  \hbox to \hsize{
    \parindent 0em
    \hfil 
    \parbox{\dimen0}{\def\baselinestretch{0.96}\small
      {\bf #1.} #2
    } 
    \hfil}
  \else \hbox to \hsize{\hfil \box\@tempboxa \hfil}
  \fi
}
\makeatother
\fi

\long\def\comment#1{}

\newcommand{\simplefixed}[2]{#1}

\comment{

   }

\begin{document}

\ifdefined\useaosstyle
\begin{frontmatter}
\title{Local Privacy, Data Processing Inequalities,
  and Minimax Rates}
\runtitle{Local Privacy and Minimax Rates}

\begin{aug}
  \author{\fnms{John C.} \snm{Duchi}\thanksref{t1,t2},
    \ead[label=e1]{jduchi@stanford.edu}}
  \author{\fnms{Michael I.} \snm{Jordan}\thanksref{t2}
    \ead[label=e2]{jordan@stat.berkeley.edu}} \\
  \and
  \author{\fnms{Martin J.} \snm{Wainwright}\thanksref{t2}
    \ead[label=e3]{wainwrig@stat.berkeley.edu}}
  
  \runauthor{J.\ C.\ Duchi, M.\ I.\ Jordan, and M.\ J.\ Wainwright}
  
  \affiliation{Stanford University \and University of California, Berkeley}

  \address{Department of Statistics \\
    Department of Electrical Engineering \\
    Stanford University \\
    \printead{e1}}
  
  \address{Department of Statistics \\
    Department of EECS \\
    University of California, Berkeley \\
    \printead{e2} \\
    \phantom{E-mail:\ }\printead*{e3}}
  
  \thankstext{t1}{Partially supported by a Facebook Graduate Fellowship}
  \thankstext{t2}{Partially supported by the Office of Naval Research
    MURI grant N00014-11-1-0688 and the
    U.S.\ Army Research Laboratory and the U.S.\ Army Research Office under
    grant no.\ W911NF-11-1-0391.
  }
\end{aug}

\else
\begin{center}
  {\Large{\bf{Local Privacy, Data Processing Inequalities,
        and Minimax Rates}}}

  \vspace*{.3in}
  
  \begin{tabular}{ccc}
    John C.\ Duchi$^\dagger$ & Michael I.\ Jordan$^{\ast}$ &
    Martin J.\ Wainwright$^{\ast}$
    \\ \texttt{jduchi@stanford.edu} &
    \texttt{jordan@stat.berkeley.edu} &
    \texttt{wainwrig@stat.berkeley.edu}
  \end{tabular}
  
  \vspace*{.5cm}
  
  \begin{tabular}{ccc}
    Stanford University$^\dag$ &&
    University of California, Berkeley$^\ast$ \\
    Stanford, CA 94305 &&
    Berkeley, CA 94720
  \end{tabular}
  
  \vspace*{.2in}
\end{center}
\fi

\begin{abstract}  
  Working under a model of privacy in which data remains private even
  from the statistician, we study the tradeoff between privacy
  guarantees and the utility of the resulting statistical estimators.
  We prove bounds on information-theoretic quantities, including
  mutual information and Kullback-Leibler divergence, that depend on
  the privacy guarantees.  When combined with standard minimax
  techniques, including the Le Cam, Fano, and Assouad methods, these
  inequalities allow for a precise characterization of statistical
  rates under local privacy constraints. We provide a treatment of
  several canonical families of problems: mean estimation, parameter
  estimation in fixed-design regression, multinomial probability
  estimation, and nonparametric density estimation.  For all of these
  families, we provide lower and upper bounds that match up to
  constant factors, and exhibit new (optimal) privacy-preserving
  mechanisms and computationally efficient estimators that achieve the
  bounds.
\end{abstract}

\ifdefined\useaosstyle

\end{frontmatter}
\else
\vspace*{.1in}
\fi


\section{Introduction}

A major challenge in statistical inference is that of characterizing
and balancing statistical utility with the privacy of individuals from
whom data is obtained~\cite{DuncanLa86,DuncanLa89,FienbergMaSt98}.
Such a characterization requires a formal definition of privacy, and
\emph{differential privacy} has been put forth as one such
formalization~\cite[e.g.,][]{ DworkMcNiSm06, BlumLiRo08, DworkRoVa10,
  HardtRo10, HardtTa10}.  In the database and cryptography literatures
from which differential privacy arose, early research was mainly
algorithmic in focus, and researchers have used differential privacy
to evaluate privacy-retaining mechanisms for transporting, indexing,
and querying data.  More recent work aims to link differential privacy
to statistical concerns~\cite{DworkLe09, WassermanZh10, HallRiWa11,
Smith11, ChaudhuriMoSa11, RubinsteinBaHuTa12}; in particular, researchers 
have developed algorithms for private robust statistical estimators, 
point and histogram estimation, and principal components analysis. 
Guarantees of optimality in this line of work have often been 
non-inferential, aiming to approximate a class of statistics under 
privacy-respecting transformations of the data at hand and not with 
respect to an underlying population.  There has also been recent work 
within the context of classification problems and the ``probably 
approximately correct'' framework of statistical learning
theory~\cite[e.g.][]{KasiviswanathanLeNiRaSm11,BeimelKaNi10}
that treats the data as random and aims to recover aspects of the 
underlying population; we discuss this work in  Section~\ref{sec:related-work}.


In this paper, we take a fully inferential point of view on privacy,
bringing differential privacy into contact with statistical decision
theory.  Our focus is on the fundamental limits of
\mbox{differentially-private} estimation.
By treating differential privacy as an abstract constraint on estimators,
we obtain independence from specific estimation procedures and
privacy-preserving mechanisms.  Within
this framework, we derive both lower bounds and matching upper bounds
on minimax risk.  We obtain our lower bounds by integrating
differential privacy into the classical paradigms for bounding minimax
risk via the inequalities of Le Cam, Fano, and Assouad, while we obtain
matching upper bounds by proposing and analyzing specific
private procedures.

We study the setting of \emph{local privacy}, in which providers 
do not even trust the statistician collecting the data.
Although local privacy is a relatively stringent requirement, we view
this setting as a natural step in identifying minimax risk bounds
under privacy constraints. Indeed, local privacy is one of the oldest
forms of privacy: its essential form dates to~\citet{Warner65},
who proposed it as a remedy for what he termed ``evasive answer bias''
in survey sampling. We hope that we can leverage deeper understanding
of this classical setting to treat other privacy-preserving approaches
to data analysis.

More formally, let \mbox{$\statrv_1, \ldots, \statrv_n \in
  \statdomain$} be observations drawn according to a distribution
$\statprob$, and let $\optvar = \optvar(\statprob)$ be a parameter
of this unknown distribution.  We wish to estimate $\optvar$ based
on access to obscured views $\channelrv_1, \ldots, \channelrv_n \in
\channeldomain$ of the original data.  The original random variables
$\{\statrv_i\}_{i=1}^n$ and the privatized observations
$\{\channelrv_i\}_{i=1}^n$ are linked via a family of conditional
distributions $\channelprob_i(\channelrv_i \mid \statrv_i =
\statsample, \channelrv_{1:i-1} = \channelval_{1:i-1})$.  To simplify
notation, we typically omit the subscript in $\channelprob_i$.
We refer to $\channel$ as a 
\emph{channel distribution}, as it acts as a conduit from the
original to the privatized data, and we assume it is
\emph{sequentially interactive}, meaning the channel
has the conditional
independence structure
\begin{equation*}
  \{\statrv_i, \channelrv_1, \ldots, \channelrv_{i-1}\} \to
  \channelrv_i ~~~ \mbox{and} ~~~ \channelrv_i \perp \statrv_j \mid
  \{\statrv_i, \channelrv_1, \ldots, \channelrv_{i-1}\} ~ \mbox{for~}
  j \neq i,
\end{equation*}
illustrated on the left of
Figure~\ref{fig:interactive-channel}.  A special case of 
such a channel is the \emph{non-interactive} case,
in which each $\channelrv_i$ depends only on $\statrv_i$
(Fig.~\ref{fig:interactive-channel}, right).

\begin{figure}[t]
  \begin{center}
    \begin{tabular}{cc}
      \begin{overpic}[width=.4\columnwidth]
        {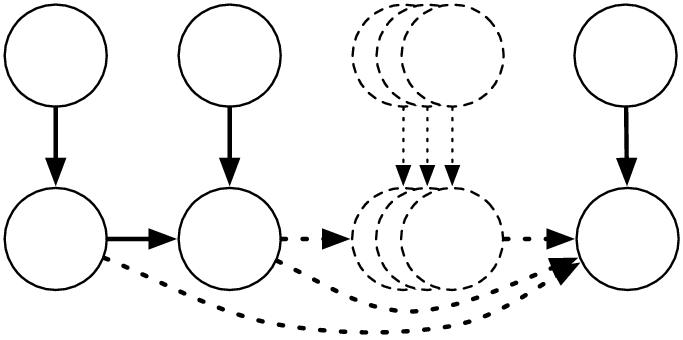} \put(5,13){$\channelrv_1$}
        \put(30,13){$\channelrv_2$} \put(88,13){$\channelrv_n$}
        \put(5,39.5){$\statrv_1$} \put(30,39.5){$\statrv_2$}
        \put(88,39.5){$\statrv_n$}
      \end{overpic} &
      \hspace{.5cm}
      \begin{overpic}[width=.4\columnwidth]
        {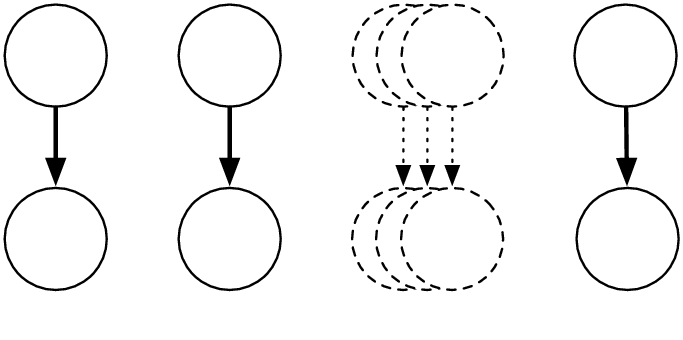} \put(5,13){$\channelrv_1$}
        \put(30,13){$\channelrv_2$} \put(88,13){$\channelrv_n$}
        \put(5,39.5){$\statrv_1$} \put(30,39.5){$\statrv_2$}
        \put(88,39.5){$\statrv_n$}
      \end{overpic}
    \end{tabular}
    \caption{\label{fig:interactive-channel} Left: graphical
      structure of private $\channelrv_i$ and non-private data
      $\statrv_i$ in interactive case. Right: graphical
      structure of channel in non-interactive case.}
  \end{center}
\end{figure}

Our work is based on the following definition of
privacy. For a given privacy parameter $\diffp \geq 0$, we say that
$\channelrv_i$ is an $\diffp$-\emph{differentially locally private}
view of $\statrv_i$ if for all $\channelval_1, \ldots,
\channelval_{i-1}$ and $\statsample, \statsample' \in \statdomain$ we
have
\begin{equation}
  \label{eqn:local-privacy}
  \sup_{S \in \sigma(\channeldomain)}
  \frac{\channelprob_i(\channelrv_i \in S \mid \statrv_i =
    \statsample, \channelrv_1 = \channelval_1, \ldots,
    \channelrv_{i-1} = \channelval_{i-1})}{
    \channelprob_i(\channelrv_i \in S \mid \statrv_i = \statsample',
    \channelrv_1 = \channelval_1, \ldots, \channelrv_{i-1} =
    \channelval_{i-1})}
  \le \exp(\diffp),
\end{equation}
where $\sigma(\channeldomain)$ denotes an appropriate $\sigma$-field
on $\channeldomain$. Definition~\eqref{eqn:local-privacy} does not
constrain $\channelrv_i$ to be a release of data based exclusively
on $\statrv_i$: the channel $\channel_i$ may be
\emph{interactive}~\cite{DworkMcNiSm06}, changing based on prior
private observations $\channelrv_j$. We also consider the
non-interactive case~\cite{Warner65,EvfimievskiGeSr03} where $\channelrv_i$
depends only on $\statrv_i$
(see the right side of Figure~\ref{fig:interactive-channel}); here
the bound~\eqref{eqn:local-privacy} reduces to
\begin{equation}
  \label{eqn:local-privacy-simple}
  \sup_{S \in \sigma(\channeldomain)} \sup_{\statsample, \statsample'
    \in \statdomain} \frac{\channelprob(\channelrv_i \in S \mid \statrv_i =
    \statsample)}{
    \channelprob(\channelrv_i \in S \mid \statrv_i = \statsample')}
  \leq \exp(\diffp).
\end{equation}

These definitions capture a type of plausible deniability: no matter
what data $\channelrv$ is released, it is nearly equally as likely to
have come from one point $\statsample \in \statdomain$ as any other.
It is also possible to interpret differential privacy within a
hypothesis testing framework, where $\alpha$ controls the error rate
in tests for the presence or absence of individual data points in a
dataset~\citep{WassermanZh10}.  Such guarantees against discovery,
together with the treatment of issues of side information or
adversarial strength that are problematic for other formalisms, have
been used to make the case for differential privacy within the
computer science literature; see, for example, the
papers~\cite{EvfimievskiGeSr03, DworkMcNiSm06, BarakChDwKaMcTa07,
  GantaKaSm08}.

Although differential privacy provides an elegant formalism for
limiting disclosure and protecting against many forms of privacy
breach, it is a stringent measure of privacy, and it is conceivably
overly stringent for statistical practice.  Indeed,
\citet{FienbergRiYa10} criticize the use of differential privacy in
releasing contingency tables, arguing that known mechanisms for
differentially private data release can give unacceptably poor
performance.  As a consequence, they advocate---in some
cases---recourse to weaker privacy guarantees to maintain the utility
and usability of released data.  There are results that are
more favorable for differential privacy; for example, \citet{Smith11}
shows that the non-local form of
differential privacy~\cite{DworkMcNiSm06} can be satisfied while
yielding asymptotically optimal parametric rates of convergence for
some point estimators.  Resolving
such differing perspectives requires investigation into whether
particular methods have optimality properties that would allow a
general criticism of the framework, and characterizing the trade-offs
between privacy and statistical efficiency.  Such are the goals of the
current paper.


\subsection{Our contributions}

The main contribution of this work is to provide general techniques
for deriving minimax bounds under local privacy constraints and to
illustrate these techniques by computing minimax rates
for several canonical problems: (a) mean estimation; (b) parameter
estimation in fixed design regression; (c) multinomial probability
estimation; and (d) density estimation. We now outline our main
contributions. (Because a deeper comparison of the current work with 
prior research requires a formal definition of our minimax framework 
and presentation of our main results, we defer a full discussion of 
related work to Section~\ref{sec:related-work}. We note here, however, 
that our minimax rates are for estimation of \emph{population} quantities, 
in accordance with our connections to statistical decision theory; by
way of comparison, most prior work in the privacy literature focuses 
on accurate approximation of statistics in a conditional analysis in 
which the data are treated as fixed.

Many methods for obtaining minimax bounds involve information-theoretic
quantities relating data-generating distributions~\citep{Yu97,YangBa99,Tsybakov09}.  
In particular, let $\statprob_1$ and $\statprob_2$ denote two distributions 
on the observations $\statrv_i$, and for $\packval \in \{1, 2 \}$, define the 
marginal distribution $\marginprob^n_\packval$ on $\channeldomain^n$ by
\begin{equation}
  \marginprob_\packval^n(S) \defeq \int \channelprob^n(S \mid
  \statsample_1, \ldots, \statsample_n) d
  \statprob_\packval(\statsample_1, \ldots, \statsample_n) ~~~
  \mbox{for}~ S \in \sigma(\channeldomain^n).
  \label{eqn:marginal-channel}
\end{equation}
Here $\channelprob^n(\cdot \mid \statsample_1, \ldots, \statsample_n)$
denotes the joint distribution on $\channeldomain^n$ of the private
sample $\channelrv_{1:n}$, conditioned on
$\statrv_{1:n} = \statsample_{1:n}$.  The
mutual information of samples drawn according to distributions of the
form~\eqref{eqn:marginal-channel} and the KL divergence between such
distributions are key objects in statistical discriminability and
minimax rates~\cite{Hasminskii78,Birge83,Yu97,YangBa99,Tsybakov09},
where they are often applied in one of three lower-bounding
techniques: Le Cam's, Fano's, and Assouad's methods.

Keeping in mind the centrality of these information-theoretic
quantities, we summarize our main results at a high-level as
follows.  Theorem~\ref{theorem:master}
bounds the KL divergence between distributions $\marginprob_1^n$ and
$\marginprob_2^n$, as defined by the
marginal~\eqref{eqn:marginal-channel}, by a quantity dependent on the
differential privacy parameter $\diffp$ and the total variation
distance between $\statprob_1$ and $\statprob_2$. The essence of
Theorem~\ref{theorem:master} is that
\begin{equation*}
  \dkl{\marginprob_1^n}{\marginprob_2^n}
  \lesssim \diffp^2 n
  \tvnorm{\statprob_1 - \statprob_2}^2,
\end{equation*}
where $\lesssim$ denotes inequality up to numerical constants.
When
$\diffp^2 < 1$, which is the usual region of interest, this result
shows that for statistical procedures whose minimax rate of
convergence can be determined by classical information-theoretic
methods, the additional requirement of $\diffp$-local differential
privacy causes the \emph{effective sample size} of \emph{any}
statistical procedure to be reduced from $n$ to at most $\diffp^2 n$.
Section~\ref{sec:pairwise-kl-bounds} contains the formal statement of
this theorem, while Section~\ref{sec:pairwise-corollaries} provides
corollaries showing its application to minimax risk bounds.
We follow this in Section~\ref{sec:initial-examples} with applications
of these results to estimation of one-dimensional means and
fixed-design regression problems, providing corresponding upper bounds
on the minimax risk.
In addition to our general analysis, we exhibit some
striking difficulties of locally private estimation
in non-compact spaces: if we wish to
estimate the mean of a random variable $\statrv$ satisfying
$\var(\statrv) \le 1$, the minimax rate of estimation of $\E[X]$
decreases from the parametric $1/n$ rate to $1 / \sqrt{n \diffp^2}$.


Theorem~\ref{theorem:master} is appropriate for many one-dimensional
problems, but it does not address difficulties inherent in 
higher-dimensional problems.  With this motivation, our next two 
main results (Theorems~\ref{theorem:super-master}
and~\ref{theorem:sequential-interactive}) generalize
Theorem~\ref{theorem:master} and incorporate dimensionality in an
essential way: each provides bounds on information-theoretic
quantities by dimension-dependent analogues of total variation.
More specifically, Theorem~\ref{theorem:super-master} provides bounds
on mutual information quantities essential in information-theoretic techniques
such as Fano's method~\cite{Yu97,YangBa99}, while
Theorem~\ref{theorem:sequential-interactive} provides analogous bounds on
summed pairs of KL-divergences useful in applications of Assouad's
method~\cite{Assouad83,Yu97,Arias-CastroCaDa13}.


As a consequence of Theorems~\ref{theorem:super-master}
and~\ref{theorem:sequential-interactive}, we obtain that for many
$d$-dimensional estimation problems the effective sample size is
reduced from $n$ to $n \diffp^2 / d$; as our examples illustrate,
this dimension-dependent reduction in sample size can have
dramatic consequences. We provide the main statement and consequences
of Theorem~\ref{theorem:super-master} in
Section~\ref{sec:super-master-lemmas}, showing its application to
obtaining minimax rates for mean estimation in both classical and
high-dimensional settings. In Section~\ref{sec:assouad}, we present
Theorem~\ref{theorem:sequential-interactive}, showing how it provides
(sharp) minimax lower bounds for multinomial and probability density
estimation.  Our results enable us to derive (often new) optimal
mechanisms for these problems.
\ifdefined\useaosstyle
\else
One interesting consequence of our
results is that Warner's randomized response procedure~\cite{Warner65}
from the 1960s is an optimal mechanism for multinomial estimation.
\fi

\ifdefined\compareprevious
We now take a moment to situate the current paper in the context of our own
previous work~\cite{DuchiJoWa14}. The earlier paper presents results on
stochastic optimization, and to provide guarantees it assumes a particular
noise model for preventing disclosure that restricts its applicability
specifically to such optimization problems.  In contrast, here we provide a
suite of new results via new techniques; as noted in the preceding
paragraphs, we give private versions of three classical minimax
techniques---those due to Le Cam, Fano and Assouad---and show how to apply
them to several estimation problems. We also note that preliminary abstracts
of this paper have appeared in conferences~\cite{DuchiJoWa13_focs,
  DuchiJoWa13_nips}, though this paper also differs from these
abstracts. Briefly, the current manuscript provides results for
\emph{interactive estimation}, meaning that the channel used to prevent
disclosure of sensitive information may adapt to data seen
previously. Additionally, most of the examples---including regression,
high-dimensional mean estimation, and multinomial and density estimation in
interactive settings---are new. Finally, the current manuscript provides a
new, general form of Assouad's inequality appropriate for privacy problems,
which we view as a major contribution of the current manuscript.
\fi

\paragraphc{Notation}

For distributions $P$ and $Q$
defined on a space $\statdomain$, each absolutely continuous with
respect to a distribution $\mu$ (with corresponding densities $p$ and
$q$), the KL divergence between $P$ and $Q$ is
\begin{equation*}
  \dkl{P}{Q} \defeq \int_\statdomain dP \log \frac{dP}{dQ} =
  \int_\statdomain p \log \frac{p}{q} d\mu.
\end{equation*}
Letting $\sigma(\statdomain)$ denote the (an appropriate)
$\sigma$-field on $\statdomain$, the total variation distance between
two distributions $P$ and $Q$ is
\begin{equation*}
  \tvnorm{P - Q} \defeq \sup_{S \in \sigma(\statdomain)} |P(S) - Q(S)|
  = \half \int_\statdomain \left|p(\statsample) -
  q(\statsample)\right| d\mu(\statsample).
\end{equation*}
Let $P$ and $P_Y$ denote marginal distributions of
random vectors $X$ and $Y$ and $P_Y(\cdot \mid X)$
denote the distribution of $Y$ conditional on $X$. The mutual
information between $X$ and $Y$ is
\begin{equation*}
  \information(X; Y) = \E_P\left[\dkl{P_Y(\cdot \mid
      X)}{P_Y(\cdot)}\right] = \! \int \! \dkl{P_Y(\cdot
    \mid X = x)}{P_Y(\cdot)} dP(x).
\end{equation*}
Random variable $Y$ has $\laplace(\alpha)$ distribution if its density
is $p_Y(y) = \frac{\alpha}{2} \exp\left(-\alpha |y|\right)$.
For matrices $A, B \in \R^{d \times d}$, the notation $A
\preceq B$ means that $B - A$ is positive semidefinite.
For real sequences
$\{a_n\}$ and $\{b_n\}$, we use $a_n \lesssim b_n$ to mean there
is a universal constant $C < \infty$ such that $a_n \leq C b_n$ for
all $n$, and $a_n \asymp b_n$ to denote that $a_n \lesssim b_n$ and
$b_n \lesssim a_n$.

\section{Background and problem formulation}
\label{SecEstimationToTest}

We first establish the minimax framework we use throughout this paper;
see references~\cite{YangBa99,Yu97,Tsybakov09} for further background.
Let $\mc{P}$ denote a class of distributions on the sample space
$\statdomain$, and let $\optvar(\statprob) \in \parameterspace$ denote
a function defined on $\mc{P}$.  The space $\Theta$ in which the
parameter $\optvar(\statprob)$ takes values depends on the underlying
statistical model (for univariate mean estimation, it is a
subset of the real line).  Let $\metric$ denote a semi-metric on the
space $\parameterspace$, which we use to measure the error of an
estimator for the parameter $\optvar$, and let $\Phi : \R_+
\rightarrow \R_+$ be a non-decreasing function with $\Phi(0) = 0$ (for
example, $\Phi(t) = t^2$).

In the classical setting, the statistician is given direct access to
i.i.d.\ observations $X_i$ drawn according to some $\statprob \in
\mc{P}$.  The local privacy setting involves an additional
ingredient, namely, a conditional distribution $\channelprob$ that
transforms the sample $\{X_i\}_{i=1}^n$ into the private sample
$\{\channelrv_i\}_{i=1}^n$ taking values in $\channeldomain$.  Based
on these $\channelrv_i$, our
goal is to estimate the unknown parameter $\theta(\statprob) \in
\Theta$.  An estimator $\thetahat$ is a measurable function
$\thetahat: \channeldomain^n \rightarrow \optdomain$, and we assess
the quality of the estimate $\thetahat(Z_1, \ldots, Z_\numobs)$ in
terms of the risk
\begin{align*}
  \Exs_{\statprob, \channelprob} \left[ \Phi\big(\metric(\thetahat(Z_1,
    \ldots, Z_\numobs), \theta(\statprob))\big) \right].
\end{align*}
For instance, for a univariate mean problem with $\metric(\theta,
\theta') = |\theta - \theta'|$ and $\Phi(t) = t^2$, this risk
is the mean-squared error.  For any fixed conditional distribution
$\channelprob$, the minimax rate is
\begin{equation}
  \label{EqnMinimaxRiskCond}
  \minimax_n(\theta(\mc{P}), \Phi \circ \metric, \channelprob) \defeq
  \inf_{\what{\theta}} \sup_{\statprob \in \mc{P}} \E_{\statprob,
    \channelprob}\left[\Phi\big(\metric(\what{\optvar}(\channelrv_1,
    \ldots, \channelrv_n), \optvar(\statprob))\big)\right],
\end{equation}
where we take the supremum over distributions $\statprob \in \mc{P}$, 
and the infimum is taken over all estimators $\what{\optvar}$.

For $\diffp > 0$, let $\QFAMA$ denote the set of all conditional
distributions guaranteeing $\diffp$-local
privacy~\eqref{eqn:local-privacy}. By minimizing the
minimax risk~\eqref{EqnMinimaxRiskCond} over all
$\channelprob \in \QFAMA$, we obtain the central object of study for
this paper, a functional which characterizes the optimal rate of estimation
in terms of the privacy parameter $\diffp$.

\begin{definition}
  Given a family of distributions $\theta(\mc{P})$ and a privacy
  parameter $\diffp > 0$, the \emph{$\diffp$-minimax rate} in the metric
  $\metric$ is
  \begin{align}
    \label{eqn:minimax-risk-optimal}
    \minimax_n(\theta(\mc{P}), \Phi \circ \metric, \diffp) & \defeq
      \inf_{\channelprob \in \QFAMA} \inf_{\what{\theta}}
      \sup_{\statprob \in \mc{P}} \E_{\statprob,
        \channelprob}\left[\Phi(\metric(\what{\optvar}(\channelrv_1,
        \ldots, \channelrv_n), \optvar(\statprob)))\right]
      .
  \end{align}
\end{definition}

\paragraphc{From estimation to testing}

A standard first step in proving minimax bounds is to reduce the estimation
problem to a testing problem~\cite{Yu97,YangBa99,Tsybakov09}.  We use
two types of testing problems: one a multiple hypothesis test,
the second based on multiple binary hypothesis tests.  We begin
with the first of the two. Given an index set $\packset$ of finite
cardinality, consider a family of distributions $\{\statprob_\packval,
\packval \in \packset \}$ contained within $\mc{P}$.  This family induces a
collection of parameters $\{\theta(\statprob_\packval), \packval \in \packset
\}$; it is a $2 \delta$-packing in the $\metric$-semimetric if
\begin{align}
  \metric(\theta(\statprob_\packval), \theta(\statprob_\altpackval)) &
  \ge 2 \delta \quad \mbox{for all $\packval \neq \altpackval$.}
\end{align}
We use this family to define the \emph{canonical hypothesis testing
  problem}:
\begin{itemize}
\item first, nature chooses $\packrv$ according to the uniform distribution
  over $\packset$;
\item second, conditioned on the choice $\packrv = \packval$, the
  random sample $X = (X_1, \ldots, X_\numobs)$ is drawn from the
  $\numobs$-fold product distribution $\statprob_\packval^\numobs$.
\end{itemize}
In the classical setting, the statistician
directly observes the sample $X$,
while the local privacy constraint means that
a new random sample $\channelrv = (\channelrv_1, \ldots, \channelrv_n)$
is generated by sampling $\channelrv_i$ from the distribution
$\channel(\cdot \mid \statrv_{1:n})$.
By construction,
conditioned on the choice $\packrv = \packval$, the private sample
$\channelrv$ is distributed according to the marginal measure
$\marginprob_\packval^\numobs$ defined in
equation~\eqref{eqn:marginal-channel}.

Given the observed vector $\channelrv$, the goal is to determine the
value of the underlying index $\packval$.  We refer to any measurable
mapping $\test: \Zspace^\numobs \rightarrow \packset$ as a test
function.  Its associated error probability is $\mprob(\test(Z_1,
\ldots, Z_\numobs) \neq \packrv)$, where $\mprob$ denotes the joint
distribution over the random index $\packrv$ and $Z$.  The classical
reduction from estimation to testing~\cite[e.g.,][Section
  2.2]{Tsybakov09} guarantees that the minimax
error~\eqref{EqnMinimaxRiskCond} has lower bound
\begin{equation}
  \label{eqn:estimation-to-testing}
  \minimax_n(\parameterspace, \Phi \circ \metric, \channelprob) \geq
  \Phi(\delta) \; \inf_{\test} \mprob(\test(Z_1, \ldots, Z_\numobs)
  \neq \packrv).
\end{equation}

The remaining challenge is to lower bound
the probability of error in the underlying multi-way hypothesis
testing problem.  There are a variety of techniques for this, and we
focus on bounds on the probability of
error~\eqref{eqn:estimation-to-testing} due to Le Cam and Fano.  The
simplest form of Le Cam's inequality~\cite[e.g.,][Lemma~1]{Yu97} is
applicable when there are two values $\packval, \altpackval$ in
$\packset$. In this case,
\begin{equation}
  \inf_{\test} \P\left(\test(\channelrv_1, \ldots, \channelrv_n) \neq
  \packrv\right) = \half - \half \tvnorm{\marginprob_\packval^n -
    \marginprob_{\altpackval}^n},
  \label{eqn:le-cam}
\end{equation}
where the marginal $\marginprob$ is defined as in
expression~\eqref{eqn:marginal-channel}. More generally, Fano's
inequality~\cite[Lemma 4.2.1]{YangBa99,Gray90}
holds when nature chooses uniformly at random from a
set $\packset$ of cardinality larger than two, and takes the form
\begin{equation}
  \inf_{\test} \P(\test(\channelrv_1, \ldots, \channelrv_n) \neq
  \packrv) \ge \left[1 - \frac{\information(\channelrv_1, \ldots,
      \channelrv_n; \packrv) + \log 2}{\log |\packset|}\right].
  \label{eqn:fano}
\end{equation}

The second reduction we consider---which transforms estimation problems into
multiple binary hypothesis testing problems---uses the structure of the
hypercube in an essential way. For some $d \in \N$, we set $\packset = \{-1,
1\}^d$.  We say that the the family
$\statprob_\packval$ induces a $2 \delta$-Hamming separation for $\Phi \circ
\metric$ if there exists a function $\maptocube : \optvar(\mc{\statprob}) \to
\{-1, 1\}^d$ satisfying
\begin{equation}
  \label{eqn:risk-separation}
  \Phi(\metric(\optvar, \optvar(\statprob_\packval)))
  \ge 2\delta \sum_{j=1}^d \indic{[\maptocube(\optvar)]_j
    \neq \packval_j}.
\end{equation}
Letting $\P_{\pm j}$ denote the joint distribution over the random index
$\packrv$ and $\channelrv$ conditional on the $j$th coordinate $\packrv_j =
\pm 1$, we are able to establish the following sharpening of
Assouad's lemma~\cite{Assouad83,Arias-CastroCaDa13} (see
Appendix~\ref{sec:proof-sharp-assouad} for a proof).
\begin{lemma}
  \label{lemma:sharp-assouad}
  Under the conditions of the previous paragraph, we have
  \begin{equation*}
    \minimax_n(\optvar(\mc{\statprob}), \Phi \circ \metric, \channel)
    \ge \delta \sum_{j=1}^d
    \inf_{\test} \left[\P_{+j}(\test(\channelrv_{1:n})
      \neq +1) + \P_{-j}(\test(\channelrv_{1:n})
      \neq -1)\right].
  \end{equation*}
\end{lemma}
\noindent
With the definition of the marginals
$\marginprob_{\pm j}^n = 2^{-d+1} \sum_{\packval :
  \packval_j = \pm 1} \marginprob_\packval^n$, expression~\eqref{eqn:le-cam}
shows that Lemma~\ref{lemma:sharp-assouad} is equivalent to the 
lower bound
\begin{equation}
  \minimax_n(\optvar(\mc{\statprob}), \Phi \circ \metric, \channel)
  \ge \delta \sum_{j=1}^d
  \left[1 - \tvnorm{\marginprob_{+j}^n - \marginprob_{-j}^n}\right].
  \label{eqn:sharp-assouad}
\end{equation}

As a consequence of the preceding reductions to testing and the error
bounds~\eqref{eqn:le-cam}, \eqref{eqn:fano}, and~\eqref{eqn:sharp-assouad},
we obtain bounds on the
private minimax rate~\eqref{eqn:minimax-risk-optimal} by controlling
variation distances of the form $\tvnorm{\marginprob_{1}^n -
  \marginprob_{2}^n}$ or the mutual information between the random parameter
index $\packrv$ and the sequence of random variables $\channelrv_1, \ldots,
\channelrv_n$.  We devote the following sections to these tasks.


\section{Pairwise bounds under privacy: Le Cam and
local Fano methods}
\label{sec:simple-master-lemmas}

We begin with results that upper bound the symmetrized Kullback-Leibler
divergence under a privacy constraint, developing
consequences of this result for both Le Cam's method and a local form
of Fano's method.  Using these methods, we derive sharp minimax rates
under local privacy for estimating $1$-dimensional means
and for $d$-dimensional fixed design regression.

\subsection{Pairwise upper bounds on Kullback-Leibler divergences}
\label{sec:pairwise-kl-bounds}

Many statistical problems depend on comparisons between a pair of
distributions $\statprob_1$ and $\statprob_2$ defined on a common space
$\statdomain$.  Any conditional distribution $\channelprob$ transforms such a
pair of distributions into a new pair $(\marginprob_1, \marginprob_2)$ via the
marginalization~\eqref{eqn:marginal-channel}; that is, 
$\marginprob_j(S)
= \int_{\statdomain} \channelprob(S \mid x) d \statprob_j(x)$
for $j = 1, 2$.
Our first main result bounds the symmetrized Kullback-Leibler (KL)
divergence between these induced marginals as a function of the
privacy parameter $\diffp > 0$ associated with the conditional
distribution $\channelprob$ and the total variation distance between
$\statprob_1$ and $\statprob_2$.
\begin{theorem}
  \label{theorem:master}
  For any $\diffp \geq 0$, let $\channelprob$ be a conditional
  distribution that guarantees $\diffp$-differential privacy.  Then for
  any pair of distributions $\statprob_1$ and $\statprob_2$, the induced
  marginals $\marginprob_1$ and $\marginprob_2$ satisfy the bound
  \begin{equation}
    \label{eqn:diffp-to-tv-bound}
    \dkl{\marginprob_1}{\marginprob_2} +
    \dkl{\marginprob_2}{\marginprob_1} \le \min\{4, e^{2\diffp}\}
    (e^\diffp - 1)^2 \tvnorm{\statprob_1 - \statprob_2}^2.
  \end{equation}
\end{theorem} 

\paragraphc{Remarks}


Theorem~\ref{theorem:master} is a type of \emph{strong data processing}
inequality~\cite{AnantharamGoKaNa13}, providing
a quantitative relationship from the
divergence $\tvnorm{\statprob_1 - \statprob_2}$ to the 
KL-divergence $\dkl{\marginprob_1}{\marginprob_2}$ that arises
after applying the channel $\channel$.
The result of Theorem~\ref{theorem:master} is similar to a result due
to \citet[Lemma III.2]{DworkRoVa10}, who show that
$\dkl{\channel(\cdot \mid \statsample)}{\channel(\cdot \mid
  \statsample')} \le \diffp(e^\diffp - 1)$ for any $\statsample,
\statsample' \in \statdomain$, which implies
$\dkl{\marginprob_1}{\marginprob_2} \le \diffp(e^\diffp - 1)$
by convexity.
This upper bound is weaker than Theorem~\ref{theorem:master} since it
lacks the term $ \tvnorm{\statprob_1 - \statprob_2}^2$.  This total
variation term is essential to our minimax lower bounds: more than
providing a bound on KL divergence, Theorem~\ref{theorem:master} shows
that differential privacy acts as a contraction on the space of
probability measures.  This contractivity holds in a strong
sense: indeed, the bound~\eqref{eqn:diffp-to-tv-bound} shows that even
if we start with a pair of distributions $\statprob_1$ and
$\statprob_2$ whose KL divergence is infinite, the induced marginals
$\marginprob_1$ and $\marginprob_2$ always have finite KL divergence.  \\

We provide the proof of Theorem~\ref{theorem:master} in
Section~\ref{SecProofTheoremOne}.  Here we develop a corollary that
has useful consequences for minimax theory under local privacy
constraints.  Suppose that conditionally on $\packrv = \packval$, we
draw a sample $X_1, \ldots, X_\numobs$ from the product
measure $\prod_{i=1}^\numobs \statprob_{\packval, i}$, and that we
draw the $\diffp$-locally private sample $\channelrv_1,
\ldots, \channelrv_n$ according to the channel $\channel(\cdot \mid
\statrv_{1:n})$.  Conditioned on $\packrv = \packval$, the private sample
is distributed according to the measure $\marginprob_\packval^\numobs$
defined previously~\eqref{eqn:marginal-channel}. Because
we allow interactive protocols, the distribution
$\marginprob_\packval^\numobs$ need not be a product distribution in
general.  Given this setup, we have the following:
\begin{corollary}
  \label{corollary:idiot-fano}
  For any $\diffp$-locally differentially private~\eqref{eqn:local-privacy}
  conditional distribution $\channelprob$
  and
  any paired sequences of distributions $\{\statprob_{\packval,i}\}$ and
  $\{\statprob_{\altpackval,i}\}$,
  \begin{equation}
    \label{EqnWeakBound}
    \dkl{\marginprob_\packval^n}{\marginprob_{\altpackval}^n} +
    \dkl{\marginprob_{\altpackval}^n}{\marginprob_{\packval}^n} \le 4
    (e^\diffp - 1)^2 \sum_{i = 1}^n \tvnorm{\statprob_{\packval, i} -
      \statprob_{\altpackval, i}}^2.
  \end{equation}
\end{corollary}
\noindent
See Section~\ref{sec:proof-idiot-fano} for the proof, which requires a
few intermediate steps to obtain the additive inequality.
Inequality~\eqref{EqnWeakBound} also immediately implies a mutual
information bound, which may be useful in applications of Fano's
inequality.  In particular, if we define the mean distribution
$\meanmarginprob^n = \frac{1}{|\packset|} \sum_{\packval \in \packset}
\marginprob^n_\packval$, then by the definition of mutual information,
we have
\begin{align}
  \information(\channelrv_1, \ldots, \channelrv_n; \packrv)
  & =
  \frac{1}{|\packset|} \sum_{\packval \in \packset}
  \dkl{\marginprob_\packval^n}{ \meanmarginprob^n} \le
  \frac{1}{|\packset|^2} \sum_{\packval, \altpackval}
  \dkl{\marginprob_\packval^n}{ \marginprob_{\altpackval}^n} \nonumber
  \\
  \label{eqn:idiot-fano}
  & \le 4(e^\diffp - 1)^2 \sum_{i=1}^n \frac{1}{|\packset|^2}
  \sum_{\packval, \altpackval \in \packset}
  \tvnorm{\statprob_{\packval, i} - \statprob_{\altpackval, i}}^2,
\end{align}
the first inequality following from the joint convexity of the KL
divergence and the final inequality from
Corollary~\ref{corollary:idiot-fano}.

\paragraphc{Remarks} Mutual information bounds under local privacy
have appeared previously. \citet{McGregorMiPiReTaVa10} study
relationships between communication complexity and differential
privacy, showing that differentially private schemes allow low
communication.  They provide a
result~\cite[Prop.~7]{McGregorMiPiReTaVa10} guaranteeing
$\information(\statrv_{1:n}; \channelrv_{1:n}) \le 3 \diffp n$;
they strengthen this bound to $\information(\statrv_{1:n};
\channelrv_{1:n}) \le (3/2)\diffp^2 n$ when the $\statrv_i$ are
i.i.d.\ uniform Bernoulli variables.  Since the total
variation distance is at most $1$, our result also implies this
scaling (for arbitrary $\statrv_i$),
but it is stronger since it involves the total variation
terms $\tvnorms{\statprob_{\packval,i} - \statprob_{\altpackval,i}}$,
which are essential in our minimax results. In addition,
Corollary~\ref{corollary:idiot-fano} allows for \emph{any}
(sequentially) interactive channel $\channelprob$; each $\channelrv_i$ may
depend on the private answers $\channelrv_{1:i-1}$ of other data providers.


\subsection{Consequences for minimax theory under local privacy constraints}
\label{sec:pairwise-corollaries}

We now turn to some consequences of Theorem~\ref{theorem:master} for
minimax theory under local privacy constraints.  For ease of
presentation, we analyze the case of independent and identically
distributed (i.i.d.) samples, meaning that $\statprob_{\packval, i}
\equiv \statprob_\packval$ for $i = 1, \ldots, \numobs$.  We show that
in both Le Cam's inequality and the local version of Fano's method,
the constraint of $\diffp$-local differential privacy reduces
the effective sample size  (at least) from $\numobs$ to $4 \diffp^2 \numobs$.

\paragraphc{Consequence for Le Cam's method}
The classical non-private
version of Le Cam's method bounds the usual minimax risk
\begin{equation*}
  \minimax_n(\theta(\mc{P}), \Phi \circ \metric) \defeq
  \inf_{\thetahat} \sup_{P \in \mc{P}} \Exs_P \Big[ \Phi \big(
    \metric(\thetahat(X_1, \ldots, X_n), \theta(P)) \big) \Big],
\end{equation*}
for estimators $\what{\optvar} : \statdomain^n \to \optdomain$ by
a binary hypothesis test.  One
version of Le Cam's lemma~\eqref{eqn:le-cam} asserts that, for any
pair of distributions $\{\statprob_1, \statprob_2\}$ such that
$\metric(\theta(\statprob_1), \theta(\statprob_2)) \ge 2 \delta$, we
have
\begin{align}
  \label{EqnLeCam}
  \minimax_n(\theta(\mc{P}), \Phi \circ \metric) & \geq \Phi(\delta)
  \; \Big \{ \half - \frac{1}{2\sqrt{2}} \, \sqrt{\numobs
    \dkl{\statprob_1}{\statprob_2}} \Big \}.
\end{align}

Returning to the $\diffp$-locally private setting, in which
the estimator $\thetahat$ depends only on the private variables
$(Z_1, \ldots, Z_\numobs)$, we measure the $\diffp$-private
minimax risk~\eqref{eqn:minimax-risk-optimal}.  By applying Le~Cam's
method to the pair $(\marginprob_1, \marginprob_2)$ along with
Corollary~\ref{corollary:idiot-fano} in the form of
inequality~\eqref{EqnWeakBound}, we find:
\begin{corollary}[Private form of Le Cam bound]
\label{CorPrivateLeCam}
Given observations from an $\diffp$-locally differential private
channel for some $\diffp \in [0, \frac{22}{35}]$, the $\diffp$-private
minimax risk is lower bounded as
\begin{align}
  \label{eqn:le-cam-private}
  \minimax_n(\theta(\mc{P}), \Phi \circ \metric, \diffp) & \geq
  \Phi(\delta) \; \Big \{ \half - \frac{1}{2\sqrt{2}} \, \sqrt{ 8
    \numobs \diffp^2 \tvnorm{\statprob_1 - \statprob_2}^2} \Big \}.
\end{align}
\end{corollary}
\noindent
Using the fact that $\tvnorm{\statprob_1 - \statprob_2}^2 \le \half
\dkl{\statprob_1}{\statprob_2}$, comparison with the original Le Cam
bound~\eqref{EqnLeCam} shows that for $\diffp \in [0, \frac{22}{35}]$, the
effect of $\diffp$-local differential privacy is to reduce the \emph{effective
  sample size} from $\numobs$ to $4 \diffp^2 \numobs$.  We illustrate use of
this private version of Le Cam's bound in our analysis of the one-dimensional
mean problem to follow.


\paragraphc{Consequences for local Fano's method}

We now turn to consequences for the so-called local form of Fano's
method.  This method is based on constructing a family of distributions 
$\{ \statprob_\packval, \packval \in \packset \}$ that defines a
$2\delta$-packing, meaning $\metric(\theta(\statprob_\packval),
\theta(\statprob_{\altpackval})) \ge 2 \delta$ for all $\packval \neq
\altpackval$, satisfying
\begin{align}
  \label{EqnFanoPairwiseKL}
  \dkl{\statprob_\packval}{\statprob_{\altpackval}} \leq \localpack^2
  \delta^2
  ~~~ \mbox{for~some~fixed~} \localpack > 0.
\end{align}
We refer to any such construction as a
\emph{$(\delta, \kappa)$ local packing}.  Recalling Fano's
inequality~\eqref{eqn:fano}, the pairwise upper
bounds~\eqref{EqnFanoPairwiseKL} imply $\information(\statrv_1, \ldots,
\statrv_\numobs; \packrv) \leq \numobs \localpack^2 \delta^2$ by a convexity
argument.  We thus obtain the local Fano lower
bound~\cite{Hasminskii78,Birge83} on the classical minimax risk:
\begin{align}
  \label{EqnLocalFano}
  \minimax_n(\theta(\mc{P}), \Phi \circ \metric) & \geq \Phi(\delta)
  \; \Big \{ 1 - \frac{\numobs \localpack^2 \delta^2 + \log 2}{\log
    |\packset|} \Big \}.
\end{align}
We now state the extension of this bound to the
$\diffp$-locally private setting.
\begin{corollary}[Private form of local Fano inequality]
  \label{CorPrivateFano}
  Consider observations from an $\diffp$-locally differential private
  channel for some \mbox{$\diffp \in [0, \frac{22}{35}]$.}  Given any
  $(\delta, \kappa)$ local packing, the $\diffp$-private minimax risk has
  lower bound
  \begin{equation}
    \label{eqn:local-fano-private}
    \minimax_n(\optdomain, \Phi \circ \metric, \diffp) \ge \Phi(\delta)
    \Big\{1 - \frac{4 \numobs \diffp^2 \localpack^2 \delta^2 + \log
      2}{\log |\packset|} \Big\}.
  \end{equation}
\end{corollary}
\noindent Once again, by comparison to the classical
version~\eqref{EqnLocalFano}, we see that, for all $\diffp \in [0,
  \frac{22}{35}]$, the price for privacy is a reduction in the
effective sample size from $\numobs$ to $4 \diffp^2 \numobs$.  The
proof is again straightfoward using Theorem~\ref{theorem:master}.  By
Pinsker's inequality, the pairwise bound~\eqref{EqnFanoPairwiseKL}
implies that
\begin{align*}
\tvnorm{\statprob_\packval - \statprob_{\altpackval}}^2 & \leq
\frac{1}{2} \localpack^2 \delta^2 \quad \mbox{for all $\packval \neq
  \altpackval$.}
\end{align*}
We find that $\information(\channelrv_1, \ldots, \channelrv_n; \packrv) \le 4
n \diffp^2 \localpack^2 \delta^2$ for all $\diffp \in [0, \frac{22}{35}]$ by
combining this inequality with the upper bound~\eqref{eqn:idiot-fano} from
Corollary~\ref{corollary:idiot-fano}.
The claim~\eqref{eqn:local-fano-private} follows by combining this
upper bound with the usual local Fano bound~\eqref{EqnLocalFano}.


\subsection{Some applications of Theorem~\ref{theorem:master}}
\label{sec:initial-examples}

In this section, we illustrate the use of the $\diffp$-private
versions of Le Cam's and Fano's inequalities, established in the
previous section as Corollaries~\ref{CorPrivateLeCam}
and~\ref{CorPrivateFano} of Theorem~\ref{theorem:master}.  First, we
study the problem of one-dimensional mean estimation.  In addition to
demonstrating how the minimax rate changes as a function of $\diffp$,
we also reveal some interesting (and perhaps disturbing) effects of
enforcing $\diffp$-local differential privacy: the effective sample
size may be even polynomially smaller than $\diffp^2 \numobs$. Our
second example studies fixed design linear regression, where we again
see the reduction in effective sample size from $\numobs$ to $\diffp^2
\numobs$. We state each of our bounds assuming $\diffp \in [0, 1]$;
the bounds hold (with different numerical constants) whenever
$\diffp \in [0, C]$ for some universal constant $C$.

\subsubsection{One-dimensional mean estimation}
\label{sec:location-family}

For some $k > 1$, consider the family
\begin{equation*}
  \mc{P}_k \defeq \big \{ \mbox{distributions $P$ such that} ~
  \Exs_P[X] \in [-1, 1] ~ \mbox{and} ~ \Exs_P[|X|^k] \leq 1 \big \},
\end{equation*}
and suppose that our goal is to estimate the mean $\theta(P) =
\Exs_P[X]$. The next proposition characterizes the $\diffp$-private
minimax risk in squared $\ell_2$-error:
\begin{align*}
  \minimax_n(\theta(\mc{P}_k), (\cdot)^2, \diffp) \defeq
  \inf_{\channelprob \in \QFAMA} \inf_{\what{\optvar}} \sup_{\statprob
    \in \mc{P}_k} \Exs \left[ \big(\what{\optvar}(\channelrv_1,
    \ldots, \channelrv_n) - \theta(P) \big)^2\right].
\end{align*}

\begin{proposition}
  \label{proposition:location-family-bound}
  There exist universal constants $0 < c_\ell \leq c_u < \infty$ such
  that for all $k > 1$ and $\diffp \in [0,1]$, the minimax
  error $\minimax_n(\theta(\mc{\statprob}_k, (\cdot)^2,
  \diffp)$ is bounded as
  \begin{equation}
    \label{eqn:location-family-bound}
    c_\ell \min\bigg\{1, 
    \left(n\diffp^2\right)^{-\frac{k - 1}{k}}\bigg\} \le \minimax_n(
    \theta(\mc{P}_k), (\cdot)^2, \diffp) \le c_u \min\left\{1,
    u_k \left(n
    \diffp^2\right)^{-\frac{k - 1}{k}}\right\},
  \end{equation}
  where $u_k = \max\{1, (k-1)^{-2}\}$.
\end{proposition}
\noindent We prove this result using
the $\diffp$-private
version~\eqref{eqn:le-cam-private} of Le Cam's inequality, as stated
in Corollary~\ref{CorPrivateLeCam}. See
Section~\ref{sec:proof-location-family} for the details. \\

To understand the bounds~\eqref{eqn:location-family-bound}, it is
worthwhile considering some special cases, beginning with the usual
setting of random variables with finite variance ($k = 2$).  In the
non-private setting in which the original sample $(X_1, \ldots,
X_\numobs)$ is observed, the sample mean $\thetahat =
\frac{1}{\numobs} \sum_{i=1}^\numobs X_i$ has mean-squared error at
most $1/n$. When we require $\diffp$-local differential
privacy, Proposition~\ref{proposition:location-family-bound} shows
that the minimax rate worsens to $1 / \sqrt{n \diffp^2}$. More
generally, for any $k > 1$, the minimax rate scales as
$\minimax_n(\optvar(\mc{P}_k), (\cdot)^2, \diffp) \asymp (n
\diffp^2)^{-\frac{k-1}{k}}$, ignoring $k$-dependent pre-factors.
As $k \uparrow \infty$, the moment condition $\Exs[|X|^k] \leq 1$
becomes equivalent to the boundedness constraint $|\statrv| \leq 1$
a.s., and we obtain the more standard parametric rate $(n
\diffp^2)^{-1}$, where there is no reduction in the exponent.

More generally, the behavior of the $\diffp$-private minimax
rates in~\eqref{eqn:location-family-bound} helps demarcate
situations in which local differential privacy may or may not be
acceptable.  In particular, for bounded domains---where we may take $k
\uparrow \infty$---local differential privacy may be quite
reasonable. However, in situations in which the sample takes values in
an unbounded space, local differential privacy provides much
stricter constraints.
Indeed, in Appendix~\ref{AppPathology}, we discuss an example
that illustrates the pathological consequences of providing (local)
differential privacy for non-compact spaces.


\subsubsection{Linear regression with fixed design}
\label{sec:linear-regression}

We turn now to the problem of linear regression. Concretely, for a
given design matrix $X \in \R^{n \times d}$, consider the standard
linear model
\begin{equation}
  \label{eqn:regression-model}
  Y = X \optvar^* + \varepsilon,
\end{equation}
where $\varepsilon \in \R^n$ is a vector of independent, zero-mean
random variables.  By rescaling as needed, we may assume that
$\optvar^* \in \optdomain = \Ball_2(1)$, the Euclidean ball of radius
one.  Moreover, we assume that a scaling constant $\stddev < \infty$
such that the noise sequence $|\varepsilon_i| \le \stddev$ for all
$i$. Given the challenges of non-compactness exhibited by the location
family estimation problems (cf.\
Proposition~\ref{proposition:location-family-bound}), this type of
assumption is required for non-trivial results.  We also assume that
$X$ has rank $d$; otherwise, the design matrix $X$ has a non-trivial
nullspace and $\optvar^*$ cannot be estimated even when $\stddev = 0$.

With the model~\eqref{eqn:regression-model} in place, let us consider
estimation of $\optvar^*$ in the squared $\ell_2$-error, where
we provide $\diffp$-locally differentially private views of the
response $Y = \{Y_i\}_{i=1}^n$.
By following the outline established in
Section~\ref{sec:pairwise-corollaries}, we provide a sharp
characterization of the $\diffp$-private minimax rate. 
\simplefixed{
  In stating the result, we let $\singval_j(A)$ denote the $j$th
  singular value of a matrix $A$.
  (See Section~\ref{sec:proof-fixed-design} for the proof.)
  \begin{proposition}
    \label{proposition:fixed-design}
    In the fixed design regression
    model where the variables $Y_i$ and are
    $\diffp$-locally differentially private for some $\diffp \in [0, 1]$,
    \begin{equation}
      \label{EqnRegressionBounds}
      \min\bigg\{1, \frac{\stddev^2 d}{n \diffp^2
        \singval_{\max}^2(X / \sqrt{n})}\bigg\} \lesssim
      \minimax_n\left(\optdomain, \ltwo{\cdot}^2, \diffp\right) \lesssim
      \min\Bigg\{1, \frac{\stddev^2 d}{ \diffp^2 n \singval_{\min}^2(X /
        \sqrt{n})} \Bigg\}.
    \end{equation}
  \end{proposition}
  To interpret the bounds~\eqref{EqnRegressionBounds}, it is helpful to
  consider some special cases.  First consider the case of an
  orthonormal design, meaning that $\frac{1}{n} X^\top X = I_{d \times
    d}$. The bounds~\eqref{EqnRegressionBounds} imply that
  $\minimax_n(\optdomain, \ltwo{\cdot}^2, \diffp) \asymp \stddev^2 d /
  (n \diffp^2)$, so that the $\diffp$-private minimax rate is fully
  determined (up to constant pre-factors).  Standard minimax rates
  for linear regression problems scale as $\stddev^2 d / n$; thus, by
  comparison, we see that requiring differential privacy indeed causes
  an effective sample size reduction from $n$ to $n \diffp^2$. %
}{
}
%
More generally, up to the difference between the maximum and minimum
singular values of the design $X$,
Proposition~\ref{proposition:fixed-design} provides a sharp
characterization of the $\diffp$-private rate for fixed-design linear
regression. As the proof makes clear, the upper bounds are
attained by adding Laplacian noise to the response variables $Y_i$
and solving the resulting normal equations as in standard linear
regression. In this case, the standard Laplacian
mechanism~\cite{DworkMcNiSm06} is optimal.

\section{Mutual information under local privacy: Fano's method}
\label{sec:super-master-lemmas}

As we have previously noted, Theorem~\ref{theorem:master} provides
indirect upper bounds on the mutual information.  However, since the
resulting bounds involve pairwise distances only, as in
Corollary~\ref{corollary:idiot-fano}, they must be used with local
packings.  Exploiting Fano's inequality in its full generality
requires a more sophisticated upper bound on the mutual information
under local privacy, which is the main topic of this section.  We
illustrate this more powerful technique by deriving lower bounds for
mean estimation problems in both classical as well as high-dimensional
settings under the non-interactive privacy
model~\eqref{eqn:local-privacy-simple}.


\subsection{Variational bounds on mutual information}

We begin by introducing some definitions needed to state the result.
Let $\packrv$ be a discrete random variable uniformly distributed over
some finite set $\packset$.  Given a family of distributions $\{
\statprob_\packval, \packval \in \packset \}$, we define the
mixture distribution
\begin{align*}
  \meanstatprob & \defeq \frac{1}{|\packset|} \sum_{\packval \in
    \packset} \statprob_\packval.
\end{align*}
A sample $\statrv \sim \meanstatprob$ can be obtained by first drawing
$\packrv$ from the uniform distribution over $\packset$, and then
conditionally on $\packrv = \packval$, drawing $X$ from the
distribution $\statprob_\packval$.  By definition, the mutual
information between the random index $\packrv$ and the sample
$\statrv$ is
\begin{align*}
  I(\statrv; \packrv) & = \frac{1}{|\packset|} \sum_{\packval \in
    \packset} \dkl{\statprob_\packval}{\meanstatprob},
\end{align*}
a representation that plays an important role in our theory.
As in the definition~\eqref{eqn:marginal-channel}, any conditional
distribution $\channelprob$ induces the family of marginal
distributions $\{\marginprob_\packval, \packval \in \packset\}$
and the associated mixture $\meanmarginprob \defeq
\frac{1}{|\packset|} \sum_{\packval \in \packset}
\marginprob_\packval$.  Our goal is to upper bound the mutual
information $I(\channelrv_1, \ldots, \channelrv_\numobs; \packvar)$,
where conditioned on $\packrv = \packval$, the random variables
$\channelrv_i$ are drawn according to $\marginprob_\packval$.

Our upper bound is variational in nature: it involves optimization
over a subset of the space $L^\infty(\statdomain) \defn \big \{ f :
\statdomain \rightarrow \R \, \mid \, \|f\|_\infty < \infty \big \}$
of uniformly bounded functions, equipped with the usual norm
\mbox{$\|f\|_\infty = \sup\limits_{x \in \statdomain} |f(x)|$.}
We define the $1$-ball of the supremum norm
\begin{equation}
  \linfset(\statdomain) \defeq \left\{ \optdens \in
  L^\infty(\statdomain) \mid \|\optdens\|_\infty \leq 1 \right\}.
  \label{eqn:L-infty-set}
\end{equation}
We show that
this set describes the maximal amount of perturbation allowed in the
conditional $\channelprob$.
Since the set $\statdomain$ is generally clear from
context, we typically omit this dependence. For each $\packval
\in \packset$, we define the linear functional $\varphi_\packval :
L^\infty(\statdomain) \rightarrow \R$ by
\begin{equation*}
  \varphi_\packval(\optdens) = \int_{\statdomain} \optdens(x)
  (d\statprob_\packval(x) - d \meanstatprob(x)).
\end{equation*}
With these definitions, we have the following result:
\begin{theorem}
  \label{theorem:super-master}
  Let $\{\statprob_\packval\}_{\packval \in \packset}$ be an arbitrary
  collection of probability measures on $\statdomain$, and let
  $\{\marginprob_\packval\}_{\packval \in \packset}$ be the set of
  marginal distributions induced by an
  $\diffp$-differentially private distribution $\channelprob$.
  Then
  \begin{align}
    \frac{1}{|\packset|} \sum_{\packval \in \packset}
    \left[\dkl{\marginprob_\packval}{\meanmarginprob} +
      \dkl{\meanmarginprob}{\marginprob_\packval}\right] & \leq
    \frac{(e^\diffp - 1)^2}{|\packset|}
    \sup_{\optdens \in \linfset(\statdomain)} \sum_{\packval \in
      \packset} \left(\varphi_\packval(\optdens)\right)^2.
  \end{align}
\end{theorem}

It is important to note that, at least up to constant factors,
Theorem~\ref{theorem:super-master} is never weaker than the results
provided by Theorem~\ref{theorem:master}, including the bounds of
Corollary~\ref{corollary:idiot-fano}.  By definition of the linear
functional $\varphi_\packval$, we have
\begin{align*}
  \sup_{\optdens \in \linfset(\statdomain)} \sum_{\packval \in
    \packset} \left(\varphi_\packval(\optdens)\right)^2
  & \stackrel{(i)}{\le}
  \sum_{\packval \in \packset} \sup_{\optdens \in \linfset(\statdomain)}
  \left(\varphi_\packval(\optdens)\right)^2
  = 4 \sum_{\packval \in \packset}
  \tvnorm{\statprob_\packval - \meanstatprob}^2,
\end{align*}
where inequality $(i)$ follows by interchanging the summation and
supremum.  Overall, we have
\begin{align*}
  I(\channelrv; \packrv)
  \le 4 (e^\diffp - 1)^2 \frac{1}{|\packset|^2} \sum_{\packval,
    \altpackval \in \packset} \tvnorm{\statprob_\packval -
    \statprob_{\altpackval}}^2.
\end{align*}
The
strength of Theorem~\ref{theorem:super-master} arises from the fact
that inequality $(i)$---the interchange of the order of
supremum and summation---may be quite loose.\\

We now present a corollary that extends
Theorem~\ref{theorem:super-master} to the setting of repeated
sampling, providing a tensorization inequality analogous to
Corollary~\ref{corollary:idiot-fano}.  Let $\packrv$ be distributed
uniformly at random in $\packset$, and assume that given $\packrv =
\packval$, the observations $\statrv_i$ are sampled independently
according to the distribution $\statprob_\packval$ for $i = 1,
\ldots, n$.
For this corollary, we require the non-interactive
setting~\eqref{eqn:local-privacy-simple} of local privacy, where each
private variable $\channelrv_i$ depends only on $\statrv_i$.

\begin{corollary}
  \label{corollary:super-fano}
  Suppose that the distributions $\{\channelprob_i\}_{i=1}^n$ are
  $\diffp$-locally differentially private in the non-interactive
  setting~\eqref{eqn:local-privacy-simple}.  Then
  \begin{equation}
    \information(\channelrv_1, \ldots, \channelrv_n; \packrv) \le
    n (e^\diffp - 1)^2
    \frac{1}{|\packset|} \sup_{\optdens \in \linfset} \sum_{\packval
      \in \packset} \left(\varphi_\packval(\optdens)\right)^2.
    \label{eqn:super-fano}
  \end{equation}
\end{corollary}
\noindent
We provide the proof of Corollary~\ref{corollary:super-fano} in
Section~\ref{sec:proof-super-fano}.  We conjecture that the
bound~\eqref{eqn:super-fano} also holds in the fully interactive
setting, but given well-known difficulties of characterizing multiple
channel capacities with feedback~\cite[Chapter 15]{CoverTh06}, it may
be challenging to verify this conjecture.

Theorem~\ref{theorem:super-master} and
Corollary~\ref{corollary:super-fano} relate the amount of mutual
information between the random perturbed views $\channelrv$ of the
data to geometric or variational properties of the underlying packing
$\packset$ of the parameter space $\optdomain$. In particular,
Theorem~\ref{theorem:super-master} and
Corollary~\ref{corollary:super-fano} show that if we can find a
packing set $\packset$ that yields linear functionals
$\varphi_\packval$ whose sum has good ``spectral''
properties---meaning a small operator norm when taking suprema over
$L^\infty$-type spaces---we can provide sharper results.
%


\subsection{Applications of Theorem~\ref{theorem:super-master} to
  mean estimation}
\label{sec:mean-estimation}

In this section, we show how Theorem~\ref{theorem:super-master},
coupled with Corollary~\ref{corollary:super-fano}, leads to sharp
characterizations of the $\diffp$-private minimax rates for classical
and high-dimensional mean estimation problems.  Our results show that
for in $d$-dimensional mean-estimation problems, the requirement of
$\diffp$-local differential privacy causes a reduction in effective
sample size from $n$ to $n \diffp^2 / d$.  Throughout this section, we
assume that the channel $\channel$ is \emph{non-interactive}, meaning
that the random variable $\channelrv_i$ depends only on $\statrv_i$,
and so that local privacy takes the simpler
form~\eqref{eqn:local-privacy-simple}. We also state each of our
results for privacy parameter $\diffp \in [0, 1]$, but note that all of
our bounds hold for any constant $\diffp$, with appropriate changes
in the numerical pre-factors.

Before proceeding, we describe two sampling mechanisms for enforcing
$\diffp$-local differential privacy.  Our methods for achieving the
upper bounds in minimax rates are based on unbiased estimators.  Let 
us assume we wish to construct an $\diffp$-private unbiased estimate 
$\channelrv$ for the vector $v \in \R^d$. The following sampling strategies 
are based on a radius $r > 0$ and a bound $\sbound > 0$ specified for 
each problem, and they require the Bernoulli random variable
\begin{align*}
  T \sim \bernoulli(\pi_\diffp), \quad \mbox{where} \quad \pi_\diffp
  \defeq e^\diffp / (e^\diffp + 1).
\end{align*}
\begin{subequations}
  \textbf{Strategy A:} Given a vector $v$
  with $\ltwo{v} \le \radius$, set $\wt{v} = \radius v / \ltwo{v}$
  with probability $\half + \ltwo{v} / 2\radius$ and $\wt{v} = -
  \radius v / \ltwo{v}$ with probability $\half - \ltwo{v} / 2
  \radius$.  Then sample \mbox{$T \sim \bernoulli(\pi_\diffp)$} and
  set
  \begin{equation}
    \label{eqn:ltwo-sampling}
    \channelrv \sim \begin{cases} \uniform(\channelval \in \R^d :
      \<\channelval, \wt{v}\> > 0, \ltwo{\channelval} = \sbound) &
      \mbox{if~} T = 1 \\ \uniform(\channelval \in \R^d :
      \<\channelval, \wt{v}\> \le 0, \ltwo{\channelval} = \sbound) &
      \mbox{if~} T = 0.
    \end{cases}
  \end{equation}
  \textbf{Strategy B:} Given a vector $v$
  with $\linf{v} \le \radius$, construct $\wt{v} \in \R^d$ with
  coordinates $\wt{v}_j$ sampled independently from $\{-\radius,
  \radius\}$ with probabilities $1/2 - v_j / (2\radius)$ and $1/2 +
  v_j / (2 \radius)$.  Then sample $T \sim \bernoulli(\pi_\diffp)$
  and set
  \begin{equation}
    \label{eqn:linf-sampling}
    \channelrv \sim \begin{cases} \uniform(\channelval \in
      \{-\sbound, \sbound\}^d : \<\channelval, \wt{v}\> > 0) &
      \mbox{if~} T = 1 \\ \uniform(\channelval \in \{-\sbound,
      \sbound\}^d : \<\channelval, \wt{v}\> \le 0) & \mbox{if~} T =
      0.
    \end{cases}
  \end{equation}
\end{subequations}

\noindent 
See Figure~\ref{fig:sampling} for visualizations of these sampling
strategies.  By inspection, each is $\diffp$-differentially private
for any vector satisfying $\ltwo{v} \le \radius$ or $\linf{v} \le
\radius$ for Strategy A or B, respectively. Moreover, each strategy
is efficiently implementable: Strategy A by normalizing a sample
from the $\normal(0, I_{d \times d})$ distribution, and Strategy B by
rejection sampling over the scaled hypercube $\{-\sbound,
\sbound\}^d$.

\begin{figure}[t]
  \begin{center}
    \begin{tabular}{ccc}
      \psfrag{g}[][][1]{\large$v$} \psfrag{1}[][][1.4]{$\frac{1}{1 +
          e^\diffp}$} \psfrag{ea}[][][1.4]{$\frac{e^\diffp}{1 +
          e^{\diffp}}$}
      \includegraphics[width=.35\columnwidth]{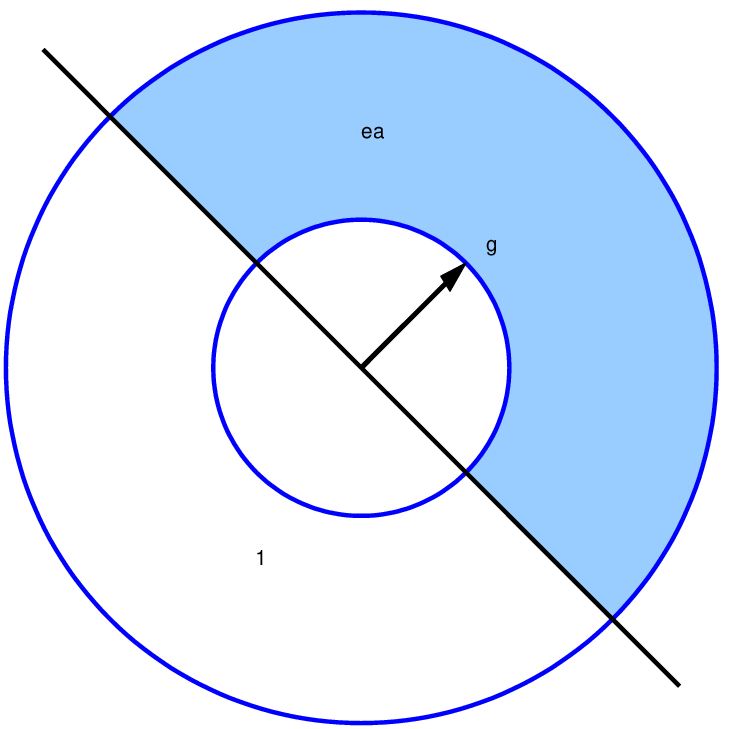} &
      \hspace{.5cm} & \psfrag{g}[][][1]{\large$v$}
      \psfrag{e0}[][][1.4]{$\frac{1}{1 + e^\diffp}$}
      \psfrag{ea}[][][1.4]{$\frac{e^\diffp}{1 + e^{\diffp}}$}
      \includegraphics[width=.4\columnwidth]{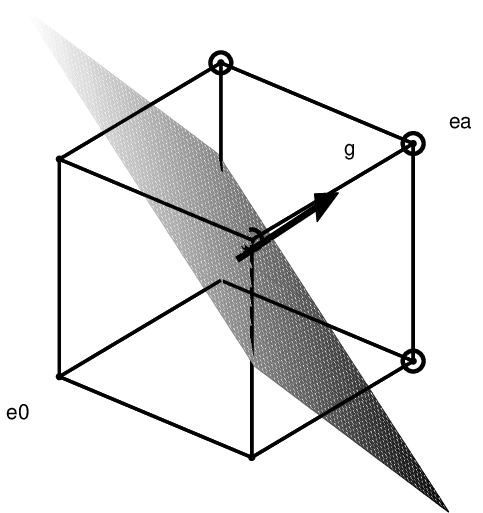} \\ (a)
      & & (b)
    \end{tabular}
  \end{center}
  \caption{\label{fig:sampling} Private sampling strategies. (a)
    Strategy~\eqref{eqn:ltwo-sampling} for the $\ell_2$-ball. Outer
    boundary of highlighted region sampled uniformly with probability
    $e^\diffp / (e^\diffp + 1)$. (b)
    Strategy~\eqref{eqn:linf-sampling} for the
    \mbox{$\ell_\infty$-ball.} Circled point set sampled uniformly
    with probability $e^\diffp / (e^\diffp + 1)$.}
\end{figure}


Given these sampling strategies, we study the $d$-dimensional problem of
estimating the mean $\theta(\statprob) \defeq \E_\statprob[\statrv]$ of a
random vector. We consider a few different metrics for the error of an
estimator of the mean to flesh out the testing reduction in
Section~\ref{SecEstimationToTest}.  Due to the difficulties associated with
differential privacy on non-compact spaces (recall
Section~\ref{sec:location-family}), we focus on distributions with compact
support.  We defer all proofs to
Appendix~\ref{sec:proofs-big-mean-estimation}; they use a combination
of Theorem~\ref{theorem:super-master} with Fano's method.

\subsubsection{Minimax rates}

\newcommand{\PRFAM}{\ensuremath{\mc{\statprob}_{p, \radius}}}
\newcommand{\PRFAMINF}{\ensuremath{\mc{\statprob}_{\infty, \radius}}}

We begin by bounding the minimax rate in the squared $\ell_2$-metric.
For a parameter $p \in [1,2]$ and radius $\radius < \infty$, consider
the family
\begin{align}
  \PRFAM & \defeq \big \{ \mbox{distributions~} \statprob
  ~ \mbox{supported on}~ \ball_p(\radius) \subset \R^d \big\}.
\end{align}
where $\Ball_p(\radius) = \{ x \in \real^d \mid \norm{x}_p \le \radius \}$ is
the $\ell_p$-ball of radius $\radius$.

\begin{proposition}
  \label{proposition:d-dimensional-mean}
  For the mean estimation problem, for all $p \in [1, 2]$ and
  privacy levels $\diffp \in [0, 1]$,
  \begin{equation*}
    \radius^2 
    \min\left\{1, \frac{1}{\sqrt{n \diffp^2}},
    \frac{d}{n \diffp^2} \right\} \lesssim
    \minimax_n(\optvar(\PRFAM), \ltwo{\cdot}^2, \diffp) \lesssim
    \radius^2 \min\left\{\frac{d}{n \diffp^2}, 1\right\}.
  \end{equation*}
\end{proposition}

\noindent This bound does not depend on the norm for $\statrv$ so
long as $p \in [1, 2]$, which is consistent with the classical mean estimation
problem. Proposition~\ref{proposition:d-dimensional-mean} demonstrates
the substantial difference between $d$-dimensional mean estimation in
private and non-private settings: more precisely, the privacy constraint leads
to a multiplicative penalty of $d/\diffp^2$ in terms of mean-squared error.
Indeed, in the non-private setting, the standard mean estimator
$\what{\optvar} = \frac{1}{n} \sum_{i=1}^n \statrv_i$ has mean-squared error
at most $\radius^2 / n$,
since $\norm{\statrv}_2 \le \norm{\statrv}_p \le \radius$ by assumption.  Thus,
Proposition~\ref{proposition:d-dimensional-mean} exhibits an effective sample
size reduction of $n \mapsto n \diffp^2 / d$.

To show the applicability of the general metric construction in
Section~\ref{SecEstimationToTest}, we now consider estimation in
$\ell_\infty$-norm; estimation in this metric is natural in scenarios where
one wishes only to guarantee that the maximum error of any particular
component in the vector $\optvar$ is small. We focus in this scenario on the
family $\PRFAMINF$ of distributions $\statprob$ supported on
$\ball_\infty(\radius) \subset \R^d$.
\begin{proposition}
  \label{proposition:minimax-mean-linf}
  For the mean estimation problem, for all $\diffp \in [0,1]$,
  \begin{equation*}
    \min\left\{\radius, \frac{\radius \sqrt{d \log(2d)}}{ \sqrt{n
        \diffp^2}}\right\} \lesssim \minimax_n(\optvar(\PRFAMINF),
    \linf{\cdot}, \diffp) \lesssim \min\left\{\radius, \frac{\radius
      \sqrt{d \log(2d)}}{ \sqrt{n \diffp^2}}\right\}.
  \end{equation*}
\end{proposition}
\noindent
Proposition~\ref{proposition:minimax-mean-linf} provides a similar message to
Proposition~\ref{proposition:d-dimensional-mean} on the loss of statistical
efficiency. This is clearest from an example: let $\statrv_i$ be random
vectors bounded by one in $\ell_\infty$-norm.  Then classical results on
sub-Gaussian random variables~\cite[e.g.,][]{BuldyginKo00}) immediately imply that the standard
non-private mean $\what{\optvar} = \frac{1}{n} \sum_{i=1}^n \statrv_i$
satisfies $\E[\linfs{\what{\optvar} - \E[\statrv]}] \le \sqrt{\log(2d) /
  n}$. Comparing this result to the rate $\sqrt{d \log(2d) / n}$ of
Proposition~\ref{proposition:minimax-mean-linf}, we again see the effective
sample size reduction $n \mapsto n \diffp^2 / d$.

Recently, there has been substantial interest in high-dimensional
problems, in which the dimension $d$ is larger than the sample size
$n$, but there is a low-dimensional latent structure that makes
inference possible.  (See the paper by \citet{NegahbanRaWaYu12} for a
general overview.)  Accordingly, let us consider an idealized version
of the high-dimensional mean estimation problem, in which we assume
that $\optvar(\statprob) = \E[\statrv] \in \R^d$ has (at most) one
non-zero entry, so $\norm{\E[\statrv]}_0 \le 1$.
In the non-private case, estimation of such an $s$-sparse predictor in
the squared $\ell_2$-norm is possible at rate
$\E[\ltwos{\what{\optvar} - \optvar}^2] \le s \log (d/s) / n$, so that
the dimension $d$ can be exponentially larger than the sample size
$n$.  With this context, the next result shows that local privacy can
have a dramatic impact in the high-dimensional setting.  Consider the
family
\begin{equation*}
  \mc{\statprob}_{\infty, \radius}^s
  \defeq \left\{\mbox{distributions}~ \statprob ~ \mbox{supported~on~}
  \ball_\infty(\radius) \subset \R^d ~ \mbox{with~}
  \norm{\E_\statprob[\statrv]}_0 \le s \right\}.
\end{equation*}
\begin{proposition}
  \label{proposition:minimax-mean-high-dim}
  For the $1$-sparse means problem, for all $\diffp \in [0, 1]$,
  \begin{equation*}
    \min\left\{\radius^2, \frac{\radius^2 d \log(2d)}{ n
      \diffp^2}\right\} \lesssim
    \minimax_n\left(\optvar(\mc{\statprob}_{\infty, \radius}^1),
    \ltwo{\cdot}^2, \diffp\right)
    \lesssim \min\left\{\radius^2,
    \frac{\radius^2 d \log(2d)}{ n \diffp^2}\right\}.
  \end{equation*}
\end{proposition}
\noindent
See Section~\ref{sec:proof-mean-high-dim} for a proof.  From
Proposition~\ref{proposition:minimax-mean-high-dim}, it becomes clear
that in locally private but
non-interactive~\eqref{eqn:local-privacy-simple} settings, 
high-dimensional estimation is effectively impossible.


\subsubsection{Optimal mechanisms: attainability for mean estimation}
\label{sec:attainability-means}



In this section, we describe how to achieve matching upper bounds in
Propositions~\ref{proposition:d-dimensional-mean}
and~\ref{proposition:minimax-mean-linf} using simple and practical
algorithms---namely, the ``right'' type of stochastic perturbation of the
observations $\statrv_i$ coupled with a standard mean estimator. We show the
optimality of privatizing via the sampling
strategies~\eqref{eqn:ltwo-sampling} and~\eqref{eqn:linf-sampling};
interestingly, we also show that privatizing via Laplace perturbation is
strictly sub-optimal.
To give a private mechanism, we must specify the conditional distribution
$\channelprob$ satisfying $\diffp$-local differential privacy used to
construct $\channelrv$. In this case, given an observation $\statrv_i$, we
construct $\channelrv_i$ by perturbing $\statrv_i$ in such a way that
$\E[\channelrv_i \mid \statrv_i = \statsample] = \statsample$. Each of the
strategies~\eqref{eqn:ltwo-sampling} and~\eqref{eqn:linf-sampling} also
requires a constant $\sbound$, and we show how to choose $\sbound$ for each
strategy to satisfy the unbiasedness condition $\E[\channelrv \mid \statrv =
  \statsample] = \statsample$.

We begin with the mean estimation problem for distributions $\PRFAM$ in
Proposition~\ref{proposition:d-dimensional-mean}, for which we use the
sampling scheme~\eqref{eqn:ltwo-sampling}. That is, let $\statrv = \statsample
\in \R^d$ satisfy $\ltwo{\statsample} \le \norm{\statsample}_p \le \radius$.
Then we construct the random vector $\channelrv$ according
to strategy~\eqref{eqn:ltwo-sampling}, where we set the initial vector
$v = \statsample$ in the sampling scheme. To achieve the unbiasedness
condition $\E[\channelrv \mid \statsample] = \statsample$,
we define the bound
\begin{equation}
  \label{eqn:ltwo-bound-size}
  \sbound = \radius \frac{e^\diffp + 1}{e^\diffp - 1} \frac{d
    \sqrt{\pi} \Gamma(\frac{d - 1}{2} + 1)}{\Gamma(\frac{d}{2} + 1)}.
\end{equation}
(See Appendix~\ref{appendix:ltwo-sampling} for a proof that
$\E[\channelrv \mid \statsample] = \statsample$ with this
choice of $\sbound$).  Notably, the choice~\eqref{eqn:ltwo-bound-size} implies
$\sbound \le c \radius \sqrt{d} / \diffp$ for a
universal constant $c < \infty$, since
$d \Gamma(\frac{d-1}{2} + 1) / \Gamma(\frac{d}{2} + 1) \lesssim \sqrt{d}$
and $e^\diffp - 1 = \diffp + \order(\diffp^2)$.
As a consequence, generating each $\channelrv_i$ by this perturbation
strategy and using the mean estimator
$\what{\optvar} = \frac{1}{n} \sum_{i = 1}^n \channelrv_i$,
the estimator $\what{\optvar}$ is unbiased for $\E[\statrv]$ and 
satisfies
\begin{equation*}
  \E\left[\ltwos{\what{\optvar} - \E[\statrv]}^2\right] =
  \frac{1}{n^2} \sum_{i = 1}^n \var(\channelrv_i) \le \frac{B^2}{n}
  \le c \frac{\radius^2 d}{n \diffp^2}
\end{equation*}
for a universal constant $c$.

In Proposition~\ref{proposition:minimax-mean-linf}, we consider the family
$\PRFAMINF$ of distributions supported on the $\ell_\infty$-ball of radius
$\radius$. In our mechanism for attaining the
upper bound, we use the sampling scheme~\eqref{eqn:linf-sampling}
to generate the private $\channelrv_i$, so that for an observation $\statrv =
\statsample \in \R^d$ with $\linf{\statsample} \le \radius$, we resample
$\channelrv$ (from the initial vector $v = \statsample$) according to
strategy~\eqref{eqn:linf-sampling}.  Again, we would like to guarantee the
unbiasedness condition $\E[\channelrv \mid \statrv = \statsample] =
\statsample$, for which we use a result of \citet{DuchiJoWa12}. 
That paper shows that taking
\begin{equation}
  \label{eqn:linf-bound-size}
  \sbound = c \frac{\radius \sqrt{d}}{\diffp}
\end{equation}
for a (particular) universal constant $c$, yields the desired
unbiasedness~\cite[Corollary 3]{DuchiJoWa12}.  Since the random variable
$\channelrv$ satisfies $\channelrv \in \ball_\infty(\radius)$ with probability
1, each coordinate $[\channelrv]_j$ of $\channelrv$ is sub-Gaussian. As a
consequence, we obtain via standard bounds~\cite{BuldyginKo00} that
\begin{equation*}
  \E[\linfs{\what{\optvar} - \optvar}^2] \le \frac{\sbound^2 \log(2d)}{n}
  = c^2 \frac{\radius^2 d \log(2d)}{n \diffp^2}
\end{equation*}
for a universal constant $c$, proving the upper bound in
Proposition~\ref{proposition:minimax-mean-linf}.

To conclude this section, we note that the strategy of adding Laplacian noise
to the vectors $\statrv$ is sub-optimal. Indeed, consider the the family
$\mc{\statprob}_{2,1}$ of distributions supported on $\ball_2(1) \subset \R^d$
as in Proposition~\ref{proposition:d-dimensional-mean}.  To guarantee
$\diffp$-differential privacy using independent Laplace noise vectors for
$\statsample \in \ball_2(1)$, we take $\channelrv = \statsample + W$
where $W \in \R^d$ has components $W_j$ that are independent and distributed
as $\laplace(\diffp / \sqrt{d})$. We have the following information-theoretic
result: if the $\channelrv_i$ are constructed via the Laplace noise mechanism,
\begin{equation}
  \label{eqn:mean-laplace-sucks}
  \inf_{\what{\optvar}}
  \sup_{\statprob \in \mc{\statprob}} \E_\statprob\left[
    \ltwos{\what{\optvar}(\channelrv_1, \ldots, \channelrv_n) -
      \E_\statprob[\statrv]}^2\right] \gtrsim \min\left\{\frac{d^2}{n
    \diffp^2}, 1\right\}.
\end{equation}
See Appendix~\ref{sec:mean-laplace-sucks} for the proof of this claim.
The poorer dimension dependence exhibted by the Laplace
mechanism~\eqref{eqn:mean-laplace-sucks} in comparison to
Proposition~\ref{proposition:d-dimensional-mean} demonstrates that
sampling mechanisms must be chosen carefully, as in the
strategies~\eqref{eqn:ltwo-sampling}--\eqref{eqn:linf-sampling}, in
order to obtain statistically optimal rates.




\section{Bounds on multiple pairwise divergences: Assouad's method}
\label{sec:assouad}

Thus far, we have seen how Le Cam's method and Fano's method, in the
form of Theorem~\ref{theorem:super-master} and
Corollary~\ref{corollary:super-fano}, can give sharp minimax rates for
various problems. However, their application appears to be limited to
problems whose minimax rates can be controlled via reductions to binary 
hypothesis tests (Le Cam's method) or for non-interactive channels satisfying 
the simpler definition~\eqref{eqn:local-privacy-simple} of local privacy 
(Fano's method). In this
section, we show that a privatized form of Assouad's method (in the
form of Lemma~\ref{lemma:sharp-assouad}) can be used to obtain sharp minimax
rates in interactive settings.  In particular, it can be applied when
the loss is sufficiently ``decomposable,'' so that the coordinate-wise
nature of the Assouad construction can be brought to bear.  Concretely, 
we show that an upper bound on a sum of paired KL-divergences, when
combined with Assouad's method, provides sharp lower bounds for
several problems, including multinomial probability estimation and
nonparametric density estimation. Each of these problems can be
characterized in terms of an effective dimension $d$, and our results
(paralleling those of Section~\ref{sec:super-master-lemmas}) show that
the requirement of $\diffp$-local differential privacy causes a
reduction in effective sample size from $n$ to $n \diffp^2 / d$.

\subsection{Variational bounds on paired divergences}

For a fixed $d \in \N$, we consider collections of distributions
indexed using the Boolean hypercube $\packset = \{-1, 1\}^d$. For
each $i \in [n]$ and $\packval \in \packset$, we let the distribution
$\statprob_{\packval, i}$ be supported on the fixed set
$\statdomain$, and we define the product
distribution $\statprob_\packval^n = \prod_{i=1}^n
\statprob_{\packval,i}$. Then for $j \in [d]$ we define the
paired mixtures
\begin{equation}
  \statprob_{+j}^n = \frac{1}{2^{d-1}}\sum_{\packval : \packval_j = 1}
  \statprob_\packval^n,
  ~~~
  \statprob_{-j}^n = \frac{1}{2^{d-1}} \sum_{\packval : \packval_j = -1}
  \statprob_\packval^n,
  ~~~
  \statprob_{\pm j,i} = \frac{1}{2^{d-1}} \sum_{\packval : \packval_j = \pm 1}
  \statprob_{\packval,i}.
  \label{eqn:paired-mixtures}
\end{equation}
(Note that $\statprob_{+j}^n$ is not necessarily a product distribution.)
Recalling the marginal
channel~\eqref{eqn:marginal-channel}, we may then
define the marginal mixtures
\begin{equation*}
  \marginprob_{+j}^n(S) \defeq 
  \frac{1}{2^{d - 1}} \sum_{\packval : \packval_j = 1}
  \marginprob_\packval^n(S)
  = \int \channelprob^n(S \mid \statsample_{1:n})
  d \statprob_{+j}^n(\statsample_{1:n})
  ~~~ \mbox{for~} j = 1, \ldots, d,
\end{equation*}
with the distributions $\marginprob_{-j}^n$ defined analogously.
For a given pair of distributions $(M, M')$, we let
\mbox{$\dklsym{M}{M'} = \dkl{M}{M'} + \dkl{M'}{M}$} denote the
symmetrized KL-divergence.
Recalling the $1$-ball of the supremum
norm~\eqref{eqn:L-infty-set}, with these definitions
we have the following theorem:
\begin{theorem}
  \label{theorem:sequential-interactive}
  Under the conditions of the previous paragraph, for any
  $\diffp$-locally differentially private~\eqref{eqn:local-privacy}
  channel $\channel$, we have
  \begin{align*}
    \sum_{j=1}^d \dklsym{\marginprob_{+j}^n}{\marginprob_{-j}^n} & \le
    2 (e^\diffp - 1)^2 \sum_{i=1}^n
    \ifdefined\useaosstyle
    \sup_{\optdens \in
      \linfset(\statdomain)} \sum_{j=1}^d \left(\int_{\statdomain}
    \optdens (d\statprob_{+j,i} -
    d\statprob_{-j,i}) \right)^2.
    \else
    \sup_{\optdens \in
      \linfset(\statdomain)} \sum_{j=1}^d \left(\int_{\statdomain}
    \optdens(\statsample) d\statprob_{+j,i}(\statsample) -
    d\statprob_{-j,i}(\statsample) \right)^2.
    \fi
  \end{align*}
\end{theorem}
\noindent
Theorem~\ref{theorem:sequential-interactive} generalizes
Theorem~\ref{theorem:master}, which corresponds to the special case $d
= 1$, though it also has parallels with Theorem~\ref{theorem:super-master}, 
as taking the supremum outside the summation is essential to obtain
sharp results. We provide the proof of
Theorem~\ref{theorem:super-master} in
Section~\ref{sec:proof-sequential-interactive}.\\

Theorem~\ref{theorem:sequential-interactive} allows us to prove
sharper lower bounds on the minimax risk. A combination
of Pinsker's inequality and Cauchy-Schwarz implies
\begin{equation*}
  \sum_{j=1}^d \tvnorm{\marginprob_{+j}^n - \marginprob_{-j}^n}
  \le \half \sqrt{d} \bigg(\sum_{j=1}^d
  \dkl{\marginprob_{+j}^n}{\marginprob_{-j}^n}
  + \dkl{\marginprob_{-j}^n}{\marginprob_{+j}^n}\bigg)^\half.
\end{equation*}
Thus, in combination with the sharper Assouad
inequality~\eqref{eqn:sharp-assouad}, whenever $\statprob_\packval$
induces a $2\delta$-Hamming separation for $\Phi \circ \metric$ we
have
\begin{equation}
  \minimax_n(\optvar(\mc{\statprob}), \Phi \circ \metric) \ge d \delta
  \Bigg[1 - \bigg(\frac{1}{4d}
    \sum_{j=1}^d\dklsym{\marginprob_{+j}^n}{\marginprob_{-j}^n}
    \bigg)^\half
    \Bigg].
  \label{eqn:sharp-assouad-kled}
\end{equation}
The combination of inequality~\eqref{eqn:sharp-assouad-kled} with
Theorem~\ref{theorem:sequential-interactive} is the foundation for the
remainder of this section.


\subsection{Multinomial estimation under local privacy}
\label{sec:multinomial-estimation}

For our first application of
Theorem~\ref{theorem:sequential-interactive}, we return to the
original motivation for local
privacy~\cite{Warner65}: avoiding survey answer bias.
Consider the probability simplex
\begin{align*}
  \simplex_d \defeq \Big \{\optvar \in \R^d \, \mid \, \optvar \ge 0
  \mbox{ and } \sum_{j=1}^d \optvar_j = 1 \Big\}.
\end{align*}
Any vector $\optvar \in \simplex_d$ specifies a multinomial random
variable taking $d$ states, in particular with probabilities
\mbox{$\statprob_\optvar(\statrv = j) = \optvar_j$} \mbox{for $j \in
  \{1, \ldots, d\}$.}  Given a sample from this distribution, our goal
is to estimate the probability vector $\optvar$.  \citet{Warner65}
studied the Bernoulli variant of this problem (corresponding to $d = 2$),
proposing a
mechanism known as \emph{randomized response}: for a given survey
question, respondents answer truthfully with probability $p >
1/2$ and lie with probability $1 - p$.  Here we show that an
extension of this mechanism is optimal for $\diffp$-locally differentially
private multinomial estimation.


\subsubsection{Minimax rates of convergence for multinomial estimation}


Our first result provides bounds on the minimax error measured in either the
squared $\ell_2$-norm or the $\ell_1$-norm for (sequentially) interactive
channels. The
$\ell_1$-norm is sometimes more appropriate for probability estimation
due to its connections with total variation distance and testing.
\begin{proposition}
  \label{proposition:multinomial-rate}
  For the multinomial estimation problem,
  for any $\diffp$-locally differentially private
  channel~\eqref{eqn:local-privacy},
  there exist universal constants $0 < c_\ell \le c_u < 5$ such
  that for all $\diffp \in [0,1]$,
  \begin{equation}
    \label{eqn:multinomial-limits}
    c_\ell \, \min\left\{1, \frac{1}{\sqrt{n \diffp^2}}, \frac{d}{n
      \diffp^2}\right\} \leq \minimax_n\left(\simplex_d, \ltwo{\cdot}^2,
    \diffp\right) \leq c_u \, \min\left\{1, \frac{d}{n \diffp^2}\right\},
  \end{equation}
  and
  \begin{equation}
    \label{eqn:multinomial-rate-l1}
    c_\ell \min
    \left\{1, \frac{d}{\sqrt{n \diffp^2}}\right\}
    \le \minimax_n\left(\simplex_d, \lone{\cdot}, \diffp\right)
    \le c_u \min\left\{1, \frac{d}{\sqrt{n \diffp^2}}\right\}.
  \end{equation}
\end{proposition}
\noindent
See Appendix~\ref{sec:proof-multinomial-rate} for the proofs of the
lower bounds. We provide simple estimation strategies
achieving the upper bounds in the next section.

As in the previous section, let us compare the private rates to
the classical rate in which there is no privacy.  The maximum
likelihood estimate sets $\what{\optvar}_j$ as the proportion of 
samples taking value $j$; it has mean-squared error
\begin{equation*}
  \E\left[\ltwos{\what{\optvar} - \optvar}^2\right]
  = \sum_{j=1}^d \E\left[(\what{\optvar}_j - \optvar_j)^2\right]
  = \frac{1}{n} \sum_{j=1}^d \optvar_j(1 - \optvar_j)
  \le \frac{1}{n}\left(1 - \frac{1}{d}\right)
  < \frac{1}{n}.
\end{equation*}
An analogous calculation for the $\ell_1$-norm yields
\begin{equation*}
  \E[\lones{\what{\optvar} - \optvar}] \le \sum_{j=1}^d
  \E[|\what{\optvar}_j - \optvar_j|] \le \sum_{j = 1}^d
  \sqrt{\var(\what{\optvar}_j)} \le \frac{1}{\sqrt{n}} \sum_{j=1}^d
  \sqrt{\optvar_j(1 - \optvar_j)} < \frac{\sqrt{d}}{\sqrt{n}}.
\end{equation*}
Consequently, for estimation in $\ell_1$ or $\ell_2$-norm, the
effect of providing $\diffp$-differential privacy causes the
effective sample size to decrease as $n \mapsto n \diffp^2 / d$.

\subsubsection{Optimal mechanisms: attainability for multinomial estimation}
\label{sec:private-multinomial-estimation}

An interesting consequence of the lower
bound~\eqref{eqn:multinomial-limits} is the following: a minor variant
of Warner's randomized response strategy is an optimal mechanism.
There are also other relatively simple estimation strategies
that achieve convergence rate $d / n \diffp^2$; the
Laplace perturbation approach~\cite{DworkMcNiSm06} is another.
Nonetheless, its ease of use, coupled with our optimality
results, provide support for randomized response as a desirable probability
estimation method.

Let us demonstrate that these strategies attain the optimal rate
of convergence.  Since there is a bijection between multinomial
observations $\statsample \in \{1, \ldots, d\}$ and the $d$ standard basis
vectors $e_1, \ldots, e_d \in \R^d$, we abuse notation and represent
observations as either when designing estimation strategies.
%
%
In randomized response, we construct the private vector $\channelrv
\in \{0, 1\}^d$ from a multinomial observation $\statsample \in \{e_1,
\ldots, e_d\}$ by sampling $d$ coordinates independently via the
procedure
\begin{equation}
  [\channelrv]_j = \begin{cases} \statsample_j & \mbox{with~probability~}
    \frac{\exp(\diffp / 2)}{1 + \exp(\diffp / 2)} \\
    1 - \statsample_j & \mbox{with~probability~}
    \frac{1}{1 + \exp(\diffp/2)}. \end{cases}
  \label{eqn:randomized-multinomial-response}
\end{equation}
The distribution~\eqref{eqn:randomized-multinomial-response} is
$\diffp$-differentially private: indeed, for $\statsample,
\statsample' \in \simplex_d$ and any $\channelval \in \{0, 1\}^d$, we
have
\begin{align*}
  \frac{\channelprob( \channelrv = \channelval \mid \statsample)}{
    \channelprob( \channelrv = \channelval \mid \statsample')}
  & = \exp\left(\frac{\diffp}{2} \left(\lone{\channelval - \statsample}
  - \lone{\channelval - \statsample'}\right)\right) \in
  \left[\exp(-\diffp), \exp(\diffp)\right],
\end{align*}
where the triangle inequality guarantees $|\lone{\channelval - \statsample} -
\lone{\channelval - \statsample'}| \le 2$.  We now compute the expected value
and variance of the random variables $\channelrv$.  Using the
definition~\eqref{eqn:randomized-multinomial-response}, we have
\begin{equation*}
  \E[\channelrv \mid \statsample] = \frac{e^{\diffp / 2}}{1 +
    e^{\diffp/2}} \statsample + \frac{1}{1 + e^{\diffp/2}} (\onevec -
  \statsample) = \frac{e^{\diffp/2} - 1}{e^{\diffp/2} + 1} \statsample
  + \frac{1}{1 + e^{\diffp/2}} \onevec.
\end{equation*}
Since the random variables $\channelrv$ are Bernoulli, we have the
variance bound $\E[\ltwo{\channelrv}^2] \le d$.
Letting $\project_{\simplex_d}$ denote the projection operator onto the
simplex, we arrive at the natural estimator
\begin{equation}
  \what{\optvar}_{\rm part}
  \defeq \frac{1}{n}
  \sum_{i=1}^n \left(\channelrv_i - \onevec / (1 + e^{\diffp/2}) \right)
  \frac{e^{\diffp/2} + 1}{e^{\diffp/2} - 1}
  ~~~ \mbox{and} ~~~
  \what{\optvar} \defeq \project_{\simplex_d}\left(\what{\optvar}_{\rm part}
  \right).
  \label{eqn:randomized-response-estimator}
\end{equation}
The projection of $\what{\optvar}_{\rm part}$ onto the probability
simplex can be done in time linear in the dimension $d$ of the
problem~\cite{Brucker84}, so the
estimator~\eqref{eqn:randomized-response-estimator} is efficiently
computable.  Since projections onto convex sets are non-expansive, any
pair of vectors in the simplex are at most $\ell_2$-distance $\sqrt{2}$ apart,
and $\E_\optvar[\what{\optvar}_{\rm part}] = \optvar$ by construction,
we have
\begin{align*}
  \E \left [\ltwos{\what{\optvar} - \optvar}^2\right] 
  & \leq \min
  \left\{2, \E \left[ \ltwos{\what{\optvar}_{\rm part} - \optvar}^2
    \right] \right \} \\
  &  \le \min \bigg\{2, \frac{d}{n} \left (
  \frac{e^{\diffp/2} + 1}{ e^{\diffp/2} - 1} \right)^2\bigg\} \lesssim
  \min\left\{1, \frac{d}{n \diffp^2} \right \}.  \nonumber
\end{align*}
Similar results hold for the $\ell_1$-norm: using the same estimator,
since Euclidean projections to the simplex are non-expansive
for the $\ell_1$ distance,
\begin{equation*}
  \E\left[\lones{\what{\optvar} - \optvar}\right]
  \le \min\bigg\{1,
  \sum_{j = 1}^d \E\left[|\what{\optvar}_{{\rm part},j} - \optvar_j|
    \right]\bigg\}
  \lesssim \min\left\{1, \frac{d}{\sqrt{n \diffp^2}}\right\}.
\end{equation*}

\subsection{Density estimation under local privacy}
\label{sec:density-estimation}

In this section, we show that the effects of local differential
privacy are more severe for nonparametric density estimation: instead
of just a multiplicative loss in the effective sample size as in
previous sections, imposing local differential privacy leads to a
different convergence rate.  This result holds even though we solve a
problem in which the function estimated and the observations themselves
belong to compact spaces.


\begin{definition}[Elliptical Sobolev space]
  \label{definition:sobolev-densities}
  For a given orthonormal basis $\{\basisfunc_j\}$ of $L^2([0, 1])$,
  smoothness parameter $\beta > 1/2$ and radius $\lipconst$, the
  Sobolev class of order $\beta$ is given by
  \begin{equation*}
    \densclass[\lipconst] \defeq \bigg\{f \in L^2([0, 1]) \mid f =
    \sum_{j=1}^\infty \optvar_j \basisfunc_j ~ \mbox{such~that} ~
    \sum_{j=1}^\infty j^{2\numderiv} \optvar_j^2 \le
    \lipconst^2\bigg\}.
  \end{equation*}
\end{definition}

If we choose the trignometric basis as our orthonormal basis,
membership in the class $\densclass[\lipconst]$ corresponds to
smoothness constraints on the derivatives of $f$.  More precisely, for
$j \in \N$, consider the orthonormal basis for $L^2([0, 1])$ of
trigonometric functions:
\begin{equation}
  \label{eqn:trig-basis}
  \basisfunc_0(t) = 1,
  ~~~ \basisfunc_{2j}(t) = \sqrt{2} \cos(2\pi j t),
  ~~~ \basisfunc_{2j + 1}(t) = \sqrt{2} \sin(2\pi j t).
\end{equation}
Let $f$ be a $\numderiv$-times almost everywhere differentiable
function for which $|f^{(\numderiv)}(x)| \le \lipconst$ for almost
every $x \in [0, 1]$ satisfying $f^{(k)}(0) = f^{(k)}(1)$ for $k \le
\numderiv - 1$. Then, uniformly over all such $f$, there is a
universal constant $c \le 2$ such that that $f \in \densclass[c \lipconst]$
(see, for instance,~\cite[Lemma A.3]{Tsybakov09}).  

Suppose our goal is to estimate a density function $f \in
\densclass[C]$ and that quality is measured in terms of the
squared error (squared $L^2[0,1]$-norm)
\begin{align*}
  \|\what{f} - f\|_2^2 \defeq \int_0^1 (\what{f}(x) - f(x))^2 dx.
\end{align*}
The well-known~\cite{Yu97,YangBa99,Tsybakov09} (non-private)
minimax squared risk scales as
\begin{equation}
  \label{eqn:classical-density-estimation-rate}
  \minimax_n \left( \densclass, \ltwo{\cdot}^2, \infty\right) \asymp
  n^{-\frac{2 \numderiv}{2 \numderiv + 1}}.
\end{equation}
The goal of this section is to understand how this minimax rate
changes when we add an $\diffp$-privacy constraint to the problem.
Our main result is to demonstrate that the classical
rate~\eqref{eqn:classical-density-estimation-rate} is no longer
attainable when we require $\diffp$-local differential privacy.

\subsubsection{Lower bounds on density estimation}
\label{sec:density-lower-bounds}

We begin by giving our main lower bound on the minimax rate of
estimation of densities when observations from the density are
differentially private.  We provide the proof of the following
proposition in Section~\ref{sec:proof-density-estimation}.
\begin{proposition}
  \label{proposition:density-estimation}
  Consider the class of densities $\densclass$ defined using the trigonometric
  basis~\eqref{eqn:trig-basis}.  
  There exists a constant $c_\numderiv > 0$
  such that for any $\diffp$-locally differentially
  private channel~\eqref{eqn:local-privacy} with $\diffp \in [0, 1]$,
  the private minimax risk has lower bound
  \begin{align}
    \label{eqn:private-density-estimation-rate}
    \minimax_n \left( \densclass[1], \ltwo{\cdot}^2, \diffp \right)
    & \geq c_\numderiv \left (n \diffp^2 \right)^{-\frac{2
        \numderiv}{2 \numderiv + 2}}.
  \end{align}
\end{proposition}
\noindent 
The most important feature of the lower
bound~\eqref{eqn:private-density-estimation-rate} is that it involves
a \emph{different polynomial exponent} than the classical minimax
rate~\eqref{eqn:classical-density-estimation-rate}.  Whereas
the exponent in classical
case~\eqref{eqn:classical-density-estimation-rate} is $2 \numderiv /
(2 \numderiv + 1)$, it reduces to $2 \numderiv / (2 \numderiv +
2)$ in the locally private setting.  For example, when we estimate
Lipschitz densities ($\numderiv = 1$), the rate degrades from
$n^{-2/3}$ to $n^{-1/2}$. 


Interestingly, no estimator based on Laplace (or exponential)
perturbation of the observations $\statrv_i$ themselves can attain the
rate of convergence~\eqref{eqn:private-density-estimation-rate}.  This
fact follows from results of \citet{CarrollHa88} on nonparametric
deconvolution. They show that if observations $\statrv_i$ are
perturbed by additive noise $W$, where the characteristic function
$\phi_W$ of the additive noise has tails behaving as $|\phi_W(t)| =
\order(|t|^{-a})$ for some $a > 0$, then no estimator can deconvolve
$\statrv + W$ and attain a rate of convergence better than
$n^{-2 \numderiv / (2 \numderiv + 2a + 1)}$.
Since the characteristic function of the Laplace distribution has tails
decaying as $t^{-2}$, no estimator based on the Laplace mechanism
(applied directly to the observations) can attain rate of convergence
better than $n^{-2\numderiv / (2 \numderiv + 5)}$.  In order to attain
the lower bound~\eqref{eqn:private-density-estimation-rate}, we must
thus study alternative privacy mechanisms.


\subsubsection{Achievability by histogram estimators}
\label{sec:histogram-estimators}

We now turn to the mean-squared errors achieved by specific practical
schemes, beginning with the special case of Lipschitz density
functions ($\numderiv = 1$).  In this special case, it suffices to
consider a private version of a classical histogram estimate.  For a
fixed positive integer $\numbin \in \N$, let
$\{\statdomain_j\}_{j=1}^\numbin$ denote the partition of $\statdomain
= [0, 1]$ into the intervals
\begin{align*}
  \statdomain_j = \openright{(j - 1) / \numbin}{j/\numbin} \quad
  \mbox{for $j = 1, 2, \ldots, \numbin-1$,~~and} ~~
  \statdomain_\numbin = [(\numbin-1)/\numbin, 1].
\end{align*}
Any histogram estimate of the density based on these $\numbin$ bins
can be specified by a vector $\optvar \in \numbin \simplex_\numbin$,
where we recall $\simplex_\numbin \subset \R^\numbin_+$ is the
probability simplex. Letting $\characteristic{E}$ denote the
characteristic (indicator) function of the set $E$,
any such vector $\optvar \in \real^\numbin$
defines a density estimate via the sum
\begin{equation*}
  f_\optvar \defeq \sum_{j=1}^\numbin \optvar_j
  \characteristic{\statdomain_j}.
\end{equation*}

Let us now describe a mechanism that guarantees $\diffp$-local
differential privacy.  Given a sample $\{\statrv_1, \ldots,
\statrv_n\}$ from the distribution $f$, consider
vectors
\begin{align}
  \channelrv_i & \defeq \histelement_\numbin(\statrv_i) + W_i, \quad
  \mbox{for $i = 1, 2, \ldots, \numobs$},
\end{align}
where $\histelement_\numbin(\statrv_i) \in \simplex_\numbin$ is a
$\numbin$-vector with $j^{th}$ entry equal to one if $\statrv_i
\in \statdomain_j$ and zeroes in all other entries, and $W_i$ is a
random vector with i.i.d.\ $\laplace(\diffp/2)$ entries.
The variables $\{\channelrv_i\}_{i=1}^\numobs$ defined in
this way are $\diffp$-locally differentially private for
$\{\statrv_i\}_{i=1}^\numobs$.
Using these private variables, we form the density estimate $\what{f}
\defeq f_{\what{\optvar}} = \sum_{j=1}^\numbin \what{\optvar}_j
\characteristic{\statdomain_j}$ based on the vector $\what{\optvar}
\defeq \Pi_\numbin \left(\frac{\numbin}{n} \sum_{i=1}^n
\channelrv_i\right)$, where $\Pi_\numbin$ denotes the Euclidean
projection operator onto the set $\numbin \simplex_\numbin$.  By
construction, we have $\what{f} \ge 0$ and $\int_0^1
\what{f}(\statsample) d\statsample = 1$, so $\what{f}$ is a valid
density estimate.  The following result characterizes its mean-squared
estimation error:
\begin{proposition}
  \label{proposition:histogram-estimator}
  Consider the estimate $\what{f}$ based on $\numbin = (n
  \diffp^2)^{1/4}$ bins in the histogram.  For any $1$-Lipschitz
  density $f : [0, 1] \rightarrow \R_+$, the MSE is upper bounded as
  \begin{align}
    \label{EqnHistoAchieve}
    \E_f\left [ \ltwobig{\what{f} - f}^2 \right ] & \le 5 (\diffp^2
    n)^{-\half} + \sqrt{\diffp} n^{-3/4}.
  \end{align}
\end{proposition}
\noindent For any fixed $\diffp > 0$, the first term in the
bound~\eqref{EqnHistoAchieve} dominates, and the $\order ((\diffp^2
\numobs)^{-\half})$ rate matches the minimax lower
bound~\eqref{eqn:private-density-estimation-rate} in the case $\beta =
1$.  Consequently, the privatized histogram estimator is
minimax-optimal for Lipschitz densities, providing a
private analog of the classical result that histogram
estimators are minimax-optimal for Lipshitz densities.  See
Section~\ref{sec:proof-histogram} for a proof of
Proposition~\ref{proposition:histogram-estimator}. We remark
that a randomized response scheme parallel to that of
Section~\ref{sec:private-multinomial-estimation} achieves the same
rate of convergence, showing that this classical mechanism is again an
optimal scheme.


\subsubsection{Achievability by orthogonal projection estimators}
\label{sec:orthogonal-series}

For higher degrees of smoothness ($\numderiv > 1$), standard histogram
estimators no longer achieve optimal rates in the classical
setting~\cite{Scott79}.  Accordingly, we now
turn to developing estimators based on orthogonal series expansion,
and show that even in the setting of local privacy, they can achieve
the lower bound~\eqref{eqn:private-density-estimation-rate} for all
orders of smoothness $\numderiv \ge 1$.

Recall the elliptical Sobolev space
(Definition~\ref{definition:sobolev-densities}), in which a function
$f$ is represented in terms of its basis expansion $f = \sum_{j =
  1}^\infty \optvar_j \basisfunc_j$.  This representation underlies
the orthonormal series estimator as follows. Given
a sample $\statrv_{1:n}$ drawn
i.i.d.\ according to a density $f \in L^2([0, 1])$, compute the
empirical basis coefficients
\begin{equation}
  \what{\optvar}_j = \frac{1}{n} \sum_{i=1}^n \basisfunc_j(\statrv_i)
  ~~~ \mbox{for~} j \in \{1, \ldots, \numbin\},
  \label{eqn:projection-estimator}
\end{equation}
where the value $\numbin \in \N$ is chosen either a priori based on
known properties of the estimation problem or adaptively, for example,
using cross-validation~\cite{Efromovich99,Tsybakov09}.  Using these
empirical coefficients, the density estimate is $\what{f} =
\sum_{j=1}^\numbin \what{\optvar}_j \basisfunc_j$.

In our local privacy setting, we consider a mechanism that, instead of
releasing the vector of coefficients $\big(\basisfunc_1(\statrv_i),
\ldots, \basisfunc_\numbin(\statrv_i) \big)$ for each data point,
employs a random vector $\channelrv_i = (\channelrv_{i,1}, \ldots,
\channelrv_{i, \numbin})$ satisfying $\E [
  \channelrv_{i,j} \mid \statrv_i] = \basisfunc_j(\statrv_i)$ for each
$j \in [\numbin]$.  We assume the basis functions are
$\orthbound$-uniformly bounded, that is, $\sup_j \sup_\statsample
|\basisfunc_j(\statsample)| \leq \orthbound < \infty$.  This
boundedness condition holds for many standard bases, including the
trigonometric basis~\eqref{eqn:trig-basis} that underlies the
classical Sobolev classes and the Walsh basis.  We
generate the random variables from the vector $v \in \R^\numbin$
defined by $v_j = \basisfunc_j(\statrv)$ in the hypercube-based
sampling scheme~\eqref{eqn:linf-sampling}, where we assume that the
outer bound $\sbound > \orthbound$.  With this sampling strategy,
iteration of expectation yields
\begin{equation}
  \E[[\channelrv]_j \mid \statrv = \statsample] = c_\numbin
  \frac{\sbound}{\orthbound \sqrt{\numbin}}
  \left(\frac{e^\diffp}{e^\diffp + 1} - \frac{1}{e^\diffp + 1}\right)
  \basisfunc_j(\statsample),
  \label{eqn:size-infinity-channel}
\end{equation}
where $c_k > 0$ is a constant (which is bounded independently of $k$).
Consequently, it suffices to take $\sbound = \order(\orthbound \sqrt{\numbin}
/ \diffp)$  to guarantee the unbiasedness condition
$\E[[\channelrv_i]_j \mid \statrv_i] = \basisfunc_j(\statrv_i)$.

Overall, the privacy mechanism and estimator perform the
following steps:
\begin{itemize}
\item given a data point $\statrv_i$, 
  set the vector $v = [\basisfunc_j(\statrv_i)]_{j=1}^\numbin$;
\item sample $\channelrv_i$ according
  to the strategy~\eqref{eqn:linf-sampling}, starting from the
  vector $v$ and using the
  bound $\sbound = \orthbound \sqrt{\numbin} (e^\diffp + 1) /
  c_\numbin (e^\diffp - 1)$;
\item compute the density estimate
  \begin{equation}
    \what{f} \defeq \frac{1}{\numobs} \sum_{i=1}^\numobs
    \sum_{j=1}^\numbin \channelrv_{i,j} \basisfunc_j.
    \label{eqn:orthogonal-density-estimator}
  \end{equation}
\end{itemize}
The resulting estimate enjoys the following guarantee, which (along with
Proposition~\ref{proposition:histogram-estimator}) makes clear that the
private minimax lower bound~\eqref{eqn:private-density-estimation-rate} is
sharp, providing a variant of the classical rates with a polynomially worse
sample complexity.
(See Section~\ref{sec:proof-density-upper-bound} for a proof.)
\begin{proposition}
  \label{proposition:density-upper-bound}
  Let $\{\basisfunc_j\}$ be a $\orthbound$-uniformly bounded
  orthonormal basis for $L^2([0, 1])$.  There exists a constant $c$
  (depending only on $\lipconst$ and $\orthbound$) such that, for any
  $f$ in the Sobolev space $\densclass[\lipconst]$, the
  estimator~\eqref{eqn:orthogonal-density-estimator} with
  \mbox{$\numbin = (n \diffp^2)^{1 / (2 \numderiv + 2)}$} has an MSE
  that is upper bounded as follows:
  \begin{equation}
    \E_f \left [\ltwos{f - \what{f}}^2 \right] \leq c \left (n \diffp^2
    \right)^{-\frac{2 \numderiv}{2 \numderiv + 2}}.
  \end{equation}
\end{proposition}

Before concluding our exposition, we make a few remarks on other
potential density estimators.  Our orthogonal series
estimator~\eqref{eqn:orthogonal-density-estimator} and sampling
scheme~\eqref{eqn:size-infinity-channel}, while similar in spirit to
that proposed by~\citet[Sec.~6]{WassermanZh10}, is different in that
it is locally private and requires a different noise strategy to
obtain both $\diffp$-local privacy and the optimal convergence rate.
Lastly, similarly to our remarks on the insufficiency of standard
Laplace noise addition for mean estimation, it is worth noting that
density estimators that are based on orthogonal series and Laplace
perturbation are sub-optimal: they can achieve (at best) rates of $(n
\diffp^2)^{-\frac{2 \numderiv}{2 \numderiv + 3}}$.  This rate
is polynomially worse than the sharp result provided by
Proposition~\ref{proposition:density-upper-bound}.  Again, we see that
appropriately chosen noise mechanisms are crucial for obtaining
optimal results.

\section{Comparison to related work}
\label{sec:related-work}

\newcommand{\dham}{d_{\rm ham}}
\newcommand{\minimaxsample}{\minimax^{\mathsf{cond}}}

There has been a substantial amount of work in developing
differentially private mechanisms, both in local and non-local
settings, and a number of authors have attempted to characterize
optimal mechanisms.  For example, \citet{KasiviswanathanLeNiRaSm11},
working within a local differential privacy setting, study
Probably-Approximately-Correct (PAC) learning problems and show that
the statistical query model~\cite{Kearns98} and local learning are
equivalent up to polynomial changes in the sample size. In our work,
we are concerned with a finer-grained assessment of inferential
procedures---that of rates of convergence of procedures and their
optimality.  In the remainder of this section, we discuss further
connections of our work to previous research on optimality, global 
(non-local) differential privacy, as well as error-in-variables models.

\subsection{Sample versus population estimation}
\label{sec:sample-vs-population}

The standard definition of differential privacy, due to
\citet{DworkMcNiSm06}, is somewhat less restrictive than the local
privacy formulation considered here.  In particular, a conditional 
distribution $\channel$ with output space $\channeldomain$ is
$\diffp$-differentially private if
\begin{equation}
  \label{eqn:differential-privacy}
  \sup\left\{\frac{\channel(S \mid \statsample_{1:\numobs})}{
    \channel(S \mid \statsample_{1:\numobs}')}
  \mid \statsample_i, \statsample_i' \in \statdomain,
  S \in \sigma(\channeldomain),
  \dham(\statsample_{1:\numobs}, \statsample'_{1:\numobs}) \le 1
  \right\} \le \exp(\diffp),
\end{equation}
where $\dham$ denotes the Hamming distance between sets.  
Several researchers have considered quantities similar to our minimax
criteria under local~\eqref{eqn:local-privacy-simple} or
non-local~\eqref{eqn:differential-privacy} differential
privacy~\cite{BeimelNiOm08,HardtTa10,HallRiWa11,De12}.  However, the
objective has often been quite different from ours: instead of
bounding errors based on population-based quantities, they provide
bounds in which the data are assumed to be held fixed.  More
precisely, let $\optvar : \statdomain^n \to \optdomain$ denote an
estimator, and let $\optvar(\statsample_{1:n})$ be a sample quantity
based on $\statsample_{1:n}$.  Prior work is based
on \emph{conditional minimax} risks of the form
\begin{equation}
  \label{eqn:sample-minimax-risk}
  \minimaxsample_\numobs(\optvar(\statdomain), \Phi \circ \metric,
  \diffp)
  \defeq
  \inf_\channel \sup_{\statsample_{1:n} \in \statdomain^n}
  \E_\channel\left[\Phi\big(\metric\big(\theta(\statsample_{1:n}),
    \what{\optvar}\big)\big) \mid \statrv_{1:n}
    = \statsample_{1:n} \right],
\end{equation}
where $\what{\optvar}$ is drawn according to $\channel(\cdot \mid
\statsample_{1:\numobs})$, the infimum is taken over all
$\diffp$-differentially private channels $\channel$, and the supremum
is taken over all possible samples of size $n$. The only
randomness in this conditional minimax risk is provided by the
channel; the data are held fixed, so there is no randomness from an
underlying population distribution.  A partial list of papers
that use definitions of this type include~\citet[Section
  2.4]{BeimelNiOm08}, \citet[Definition 2.4]{HardtTa10},
\citet[Section 3]{HallRiWa11}, and \citet{De12}.

The conditional~\eqref{eqn:sample-minimax-risk} and
population minimax risk~\eqref{eqn:minimax-risk-optimal} can differ
substantially, and such differences are critical to address within
a statistical approach to privacy-constrained inference.  The goal of 
inference is to draw conclusions about the \emph{population-based quantity}
$\optvar(\statprob)$ based on the sample.  Moreover, lower bounds on
the conditional minimax risk~\eqref{eqn:sample-minimax-risk} do not
imply bounds on the rate of estimation for the population
$\optvar(\statprob)$.  In fact, the conditional minimax
risk~\eqref{eqn:sample-minimax-risk} involves a supremum over
\emph{all possible samples} $\statsample \in \statdomain$, so the
opposite is usually true: population risks provide lower bounds on the
conditional minimax risk, as we show presently.

An illustrative example is useful to understand the differences.
Consider estimation of the mean of a normal distribution with known
standard deviation $\stddev^2$, in which the mean $\optvar =
\E[\statrv] \in [-1, 1]$ is assumed to belong to the unit interval.  
As our Proposition~\ref{proposition:location-family-bound} shows, 
it is possible to estimate the mean of a normally-distributed random
variable even under $\diffp$-local differential
privacy~\eqref{eqn:local-privacy}. In sharp contrast, the following
result shows that the conditional minimax risk is infinite for this problem:

\begin{lemma}
  \label{LemInfinity}
  Consider the normal location family $\{\normal(\theta, \stddev^2) \mid
  \theta \in [-1, 1]\}$ under $\diffp$-differential
  privacy~\eqref{eqn:differential-privacy}.  The conditional
  minimax risk of the mean statistic is
  $\minimaxsample_\numobs(\optvar(\R), (\cdot)^2, \diffp) = \infty$.
\end{lemma}
\begin{proof}
  Assume for sake of contradiction that $\delta > 0$ satisfies
  \begin{equation*}
    \channel(|\what{\optvar} - \optvar(\statsample_{1:n})| > \delta \mid
    \statsample_{1:n}) \le \half \quad \mbox{for all samples }
    \statsample_{1:n} \in \R^n.
  \end{equation*}
  Fix $N(\delta) \in \N$ and choose points $2\delta$-separated
  points $\theta_\packval$, $\packval
  \in [N(\delta)]$,
  that is, $|\theta_\packval - \theta_{\altpackval}| \ge 2\delta$ for
  $\packval \neq \altpackval$.  Then the sets $\{\optvar \in \R
  \mid |\optvar - \optvar_\packval| \le \delta\}$ are all disjoint, so
  for any pair of samples $\statsample_{1:\numobs}$ and
  $\statsample_{1:\numobs}^\packval$ with
  $\dham(\statsample_{1:\numobs}, \statsample_{1:\numobs}^\packval) \le
  1$,
  \begin{align*}
    \channel(\exists \packval \in \packset ~\mbox{s.t.}~
    |\what{\optvar} - \theta_\packval| \le \delta \mid
    \statsample_{1:n}) & = \sum_{\packval = 1}^{N(\delta)}
    \channel(|\what{\optvar} - \theta_\packval| \le \delta \mid
    \statsample_{1:n}) \\
    & \ge e^{-\diffp}\sum_{\packval = 1}^{N(\delta)}
    \channel(|\what{\optvar} - \theta_\packval| \le \delta \mid
    \statsample_{1:n}^\packval).
  \end{align*}
  We may take each sample $\statsample_{1:n}^\packval$ such that
  $\optvar(\statsample_{1:n}^\packval) = \frac{1}{n} \sum_{i=1}^n
  \statsample_i^\packval = \theta_\packval$ (for example, for each
  $\packval \in [N(\delta)]$ set $\statsample_1^\packval = n
  \theta_\packval - \sum_{i = 2}^n x_i$) and by assumption,
  \begin{equation*}
    1 \ge \channel(\exists \packval \in \packset ~\mbox{s.t.}~
    |\what{\optvar} - \theta_\packval| \le \delta \mid
    \statsample_{1:n}) \ge e^{-\diffp} N(\delta) \half.
  \end{equation*}
  Taking $N(\delta) > 2 e^\diffp$ yields a contradiction. Our
  argument applies to an arbitrary $\delta > 0$, so the claim follows.
\end{proof}
\noindent
There are variations on this result. For instance, even if the
output of the mean estimator is restricted to $[-1, 1]$, the conditional
minimax risk remains constant. Similar arguments
apply to weakenings of differential privacy
(e.g., $\delta$-approximate $\diffp$-differential
privacy~\cite{DworkKeMcMiNa06}).  Conditional
and population risks are very different quantities.

More generally, the population minimax risk usually lower bounds the
conditional minimax risk. Suppose we measure minimax 
risks in some given metric $\metric$ (so the loss $\Phi(t) = t$).  Let 
$\wt{\theta}$ be any estimator based on the original sample $\statrv_{1:n}$, 
and let $\what{\optvar}$ be any estimator based on the privatized sample.
We then have the following series of inequalities:
\begin{align}
    \E_{\channelprob,\statprob}[\metric(\optvar(\statprob),
      \what{\optvar})] & \le
    \E_{\channelprob,\statprob}[\metric(\optvar(\statprob),
      \wt{\optvar})] + \E_{\channelprob,
      \statprob}[\metric(\wt{\optvar},\what{\optvar})] \nonumber \\
    \label{eqn:population-to-minimax}
    & \le
    \E_\statprob[\metric(\optvar(\statprob), \wt{\optvar})] +
    \sup_{\statsample_{1:\numobs} \in \statdomain^n} \E_{\channelprob,
      \statprob}[\metric(\wt{\optvar}(\statsample_{1:n}),
      \what{\optvar}) \mid \statrv_{1:n} = \statsample_{1:n}].
\end{align}
The population minimax risk~\eqref{eqn:minimax-risk-optimal} thus
lower bounds the conditional minimax risk~\eqref{eqn:sample-minimax-risk}
via
$\minimaxsample_n(\wt{\optvar}(\statdomain), \metric, \diffp) \geq
\minimax_n(\optvar(\mc{\statprob}), \metric, \diffp) -
\E_\statprob[\metric(\optvar(\statprob), \wt{\optvar})]$.
In particular, if there exists an estimator $\wt{\optvar}$ based on
the original (non-private data) such that
$\E_\statprob[\metric(\optvar(\statprob), \wt{\optvar})] \leq
\frac{1}{2} \minimax_n(\optvar(\mc{\statprob}), \metric, \diffp)$
we are guaranteed that
\begin{align*}
  \minimaxsample_n(\wt{\optvar}(\statdomain), \metric, \diffp) & \geq
  \frac{1}{2} \, \minimax_n(\optvar(\mc{\statprob}), \metric, \diffp),
\end{align*}
so the conditional minimax risk is lower bounded by a constant
multiple of the population minimax risk.  This lower bound holds for each of
the examples in Sections~\ref{sec:simple-master-lemmas}--\ref{sec:assouad};
lower bounds on the $\diffp$-private population
minimax risk~\eqref{eqn:minimax-risk-optimal} are stronger than lower
bounds on the conditional minimax risk.

To illustrate one application of the lower
bound~\eqref{eqn:population-to-minimax}, consider the estimation of the sample
mean of a data set $\statsample_{1:n} \in \{0, 1\}^n$ under $\diffp$-local
privacy.  This problem has been considered before; for instance,
\citet{BeimelNiOm08} study distributed protocols for this
problem.  In Theorem 2 of their work, they show that if a protocol has $\ell$
rounds of communication, the squared error in estimating the
sample mean $(1/n)\sum_{i=1}^n \statsample_i$ is $\Omega(1 / (n\diffp^2
\ell^2))$.  The standard mean estimator $\wt{\optvar}(\statsample_{1:n}) =
(1/n) \sum_{i=1}^n \statsample_i$ has error
$\E[|\wt{\optvar}(\statsample_{1:n}) - \optvar|] \le n^{-\half}$.
Consequently, the lower bound~\eqref{eqn:population-to-minimax} with
combined with
Proposition~\ref{proposition:location-family-bound} implies
\begin{equation*}
  c \frac{1}{\sqrt{n \diffp^2}} - \frac{1}{\sqrt{n}}
  \ifdefined\useaosstyle
  \else
  \le \minimax_n(\optvar(\mc{\statprob}), |\cdot|, \diffp) - \sup_{\optvar
    \in [-1, 1]}\E[|\wt{\optvar}(\statsample_{1:n}) - \optvar|]
  \fi
  \le
  \minimaxsample_n(\optvar(\{-1, 1\}), |\cdot|, \diffp),
\end{equation*}
for some numerical constant $c > 0$.  A corollary of our results
is such an $\Omega(1 / (n \diffp^2))$ lower bound on the conditional
minimax risk for mean estimation, allowing for sequential
interactivity but not multiple ``rounds.''  An inspection of
\citeauthor{BeimelNiOm08}'s proof technique~\cite[Section
  4.2]{BeimelNiOm08} shows that their lower bound also implies a lower
bound of $1 / n \diffp^2$ for estimation of the population mean
$\E[\statrv]$ in one dimension in
\emph{non-interactive}~\eqref{eqn:local-privacy-simple} settings; it is,
however, unclear how to extend their technique to other settings.

\subsection{Local versus non-local privacy}
\label{sec:local-v-non-local}

It is also worthwhile to make some comparisons to work on non-local
forms of differential privacy, mainly to understand the differences
between local and global forms of privacy.  \citet{ChaudhuriHs12}
provide lower bounds for estimation of certain one-dimensional
statistics based on a two-point family of problems.  Their techniques
differ from those of the current paper, and do not appear to provide bounds
on the statistic being estimated, but rather one that is near to it.
\citet{BeimelKaNi10} provide some bounds on sample complexity in the
``probably approximate correct'' (PAC) framework of learning theory,
though extensions to other inferential tasks are unclear.
Other work on non-local privacy~\cite[e.g.,][]{HallRiWa11,ChaudhuriMoSa11,
Smith11} shows that for various types of estimation problems, adding 
Laplacian noise leads to degraded convergence rates in at most lower-order 
terms.  In contrast, our work shows that the 
Laplace mechanism may be highly sub-optimal in local privacy.

To understand convergence rates for non-local privacy, let us
return to estimation of a multinomial distribution in $\simplex_d$,
based on observations $\statrv_i \in \{e_j\}_{j=1}^d$. In this case,
adding a noise vector $W \in \R^d$ with i.i.d.\ entries distributed as
$\laplace(\diffp n)$ provides differential
privacy~\cite{DworkKeMcMiNa06}; the associated mean-squared error
is at most
\begin{equation*}
  \E_\optvar\bigg[\ltwobigg{\frac{1}{n} \sum_{i=1}^n \statrv_i + W -
      \optvar}^2\bigg] = \E\bigg[\ltwobigg{\frac{1}{n} \sum_{i=1}^n
      \statrv_i - \optvar}^2\bigg] + \E[\ltwo{W}^2] \le \frac{1}{n} +
  \frac{d}{n^2 \diffp^2}.
\end{equation*}
In particular, in the asymptotic regime $n \gg d$, there is no penalty
from providing differential privacy except in higher-order terms.
Similar results hold for histogram estimation~\cite{HallRiWa11},
classification problems~\cite{ChaudhuriMoSa11}, and classical point
estimation problems~\cite{Smith11}; in this sense, local and global
forms of differential privacy can be rather different.

\subsection{Error-in-variables models}

As a final remark on related work, we touch briefly on
errors-in-variables models~\cite{CarrollRuStCr06,Gleser81}, which have
been the subject of extensive study.  In such problems, one observes a
corrupted version $\channelrv_i$ of the true covariate $\statrv_i$.
Privacy analysis is one of the few settings in which it is possible to
precisely know the conditional distribution $\channelprob(\cdot \mid
\statrv_i)$.  However, the mechanisms that are optimal from our
analysis---in particular, those in
strategies~\eqref{eqn:ltwo-sampling}
and~\eqref{eqn:linf-sampling}---are more complicated than adding noise
directly to the covariates, which leads to complications. Known
(statistically) efficient error-in-variables estimation procedures
often require either solving certain integrals or estimating
equations, or solving non-convex optimization
problems~\cite[e.g.,][]{LiangLi09,MaLi10}.  Some recent
work~\cite{LohWa12} shows that certain types of non-convex programs
arising from errors-in-variables can be solved efficiently. In
density estimation (as noted in Section~\ref{sec:density-lower-bounds}),
corrupted observations lead to
nonparametric deconvolution problems that appear harder than
estimation under privacy constraints.
Further investigation of computationally
efficient procedures for nonlinear error-in-variables models for
privacy-preservation is an interesting direction for
future research.

\ifdefined\useaosstyle
\else
\section{Proof of Theorem~\ref{theorem:master} and related results}
\label{SecProofTheoremOne}

We now turn to the proofs of our results, beginning with
Theorem~\ref{theorem:master} and related results.  In all
cases, we defer the proofs of more technical lemmas to the appendices.

\subsection{Proof of Theorem~\ref{theorem:master}}
\label{sec:proof-diffp-to-tv}

Observe that $\marginprob_1$ and $\marginprob_2$ are absolutely
continuous with respect to one another, and there is a measure
$\basemeasure$ with respect to which they have densities
$\margindens_1$ and $\margindens_2$, respectively. The channel
probabilities $\channelprob(\cdot \mid \statsample)$ and
$\channelprob(\cdot \mid \statsample')$ are likewise absolutely
continuous, so that we may assume they have densities
$\channeldensity(\cdot \mid \statsample)$ and write
$\margindens_i(\channelval) = \int \channeldensity(\channelval \mid
\statsample) d \statprob_i(\statsample)$.  In terms of these
densities, we have
\ifdefined\useaosstyle
\begin{align*}
  \dklsym{\marginprob_1}{\marginprob_2} +
  \dkl{\marginprob_2}{\marginprob_1}
  & = \int
  \big(\margindens_1(\channelval) - \margindens_2(\channelval)\big)
  \log \frac{\margindens_1(\channelval)}{\margindens_2(\channelval)}
  d\basemeasure(\channelval).
\end{align*}
\else
\begin{align*}
  \dkl{\marginprob_1}{\marginprob_2} +
  \dkl{\marginprob_2}{\marginprob_1}
  & = \int
  \margindens_1(\channelval) \log \frac{\margindens_1(\channelval)}{
    \margindens_2(\channelval)} d\basemeasure(\channelval) + \int
  \margindens_2(\channelval) \log \frac{\margindens_2(\channelval)}{
    \margindens_1(\channelval)} d\basemeasure(\channelval) \\
  & = \int
  \big(\margindens_1(\channelval) - \margindens_2(\channelval)\big)
  \log \frac{\margindens_1(\channelval)}{\margindens_2(\channelval)}
  d\basemeasure(\channelval).
\end{align*}
\fi
Consequently, we must bound both the difference $\margindens_1 -
\margindens_2$ and the log ratio of the marginal densities.  The
following two auxiliary lemmas are useful:
\begin{lemma}
\label{lemma:channel-diff-tv}
For any $\diffp$-locally differentially private conditional, we have
\begin{align}
    \label{eqn:channel-diff-tv}
    \left|\margindens_1(\channelval) -
    \margindens_2(\channelval)\right| & \leq c_\diffp \inf_\statsample
    \channeldensity(\channelval \mid \statsample) \left(e^\diffp -
    1\right) \tvnorm{\statprob_1 - \statprob_2},
\end{align}
where $c_\diffp = \min\{2, e^\diffp\}$.
\end{lemma}

\begin{lemma}
  \label{lemma:log-ratio-inequality}
  Let $a, b \in \R_+$. Then $\left|\log\frac{a}{b}\right| \le \frac{|a
    - b|}{\min\{a, b\}}$.
\end{lemma}
\noindent We prove these two results at the end of this section.

With the lemmas in hand, let us now complete the proof of the
theorem.  From Lemma~\ref{lemma:log-ratio-inequality}, the log ratio
is bounded as
\begin{equation*}
\left|\log
\frac{\margindens_1(\channelval)}{\margindens_2(\channelval)} \right|
\le \frac{|\margindens_1(\channelval) - \margindens_2(\channelval)|}{
  \min\left\{\margindens_1(\channelval), \margindens_2(\channelval)
  \right\}}.
\end{equation*}
Applying Lemma~\ref{lemma:channel-diff-tv} to the numerator yields
\begin{align*}
  \left|\log \frac{\margindens_1(\channelval)}{
    \margindens_2(\channelval)}\right|
  & \leq \frac{c_\diffp
    \left(e^\diffp - 1\right) \tvnorm{\statprob_1 - \statprob_2} \;
    \inf_\statsample \channeldensity(\channelval \mid \statsample)}{
    \min\{\margindens_1(\channelval), \margindens_2(\channelval)\}} \\
  & \le
  \frac{c_\diffp \left(e^\diffp - 1\right) \tvnorm{\statprob_1 -
      \statprob_2} \; \inf_\statsample \channeldensity(\channelval \mid
    \statsample)}{ \inf_\statsample \channeldensity(\channelval \mid
    \statsample)},
\end{align*}
where the final step uses the inequality
$\min\{\margindens_1(\channelval), \margindens_2(\channelval)\} \ge
\inf_\statsample \channeldensity(\channelval \mid \statsample)$.
Putting together the pieces leads to the bound
\begin{equation*}
  \left|\log \frac{\margindens_1(\channelval)}{
    \margindens_2(\channelval)}\right|
  \le c_\diffp (e^\diffp - 1) \tvnorm{\statprob_1 - \statprob_2}.
\end{equation*}
Combining with inequality~\eqref{eqn:channel-diff-tv} yields
\begin{equation*}
  \dkl{\marginprob_1}{\marginprob_2} +
  \dkl{\marginprob_2}{\marginprob_1} \le c_\diffp^2 \left(e^\diffp -
  1\right)^2 \tvnorm{\statprob_1 - \statprob_2}^2 \int
  \inf_{\statsample} \channeldensity(\channelval \mid \statsample) d
  \basemeasure(\channelval).
\end{equation*}
The final integral is at most one, which completes the proof of the
theorem.\\

It remains to prove Lemmas~\ref{lemma:channel-diff-tv}
and~\ref{lemma:log-ratio-inequality}. We begin
with the former.  For any $\channelval \in
\channeldomain$, we have
\begin{align*}
  \margindens_1(\channelval) - \margindens_2(\channelval) & =
  \int_\statdomain \channeldensity(\channelval \mid \statsample)
  \left[d\statprob_1(\statsample) - d\statprob_2(\statsample)\right]
  \\
& = \int_\statdomain \channeldensity(\channelval \mid \statsample)
  \hinge{d\statprob_1(\statsample) - d\statprob_2(\statsample)} +
  \int_\statdomain \channeldensity(\channelval \mid \statsample)
  \neghinge{d\statprob_1(\statsample) - d\statprob_2(\statsample)}
  \\ 
& \le \sup_{\statsample \in \statdomain}
  \channeldensity(\channelval \mid \statsample) \int_\statdomain
  \hinge{d\statprob_1(\statsample) - d\statprob_2(\statsample)} +
  \inf_{\statsample \in \statdomain} \channeldensity(\channelval \mid
  \statsample) \int_\statdomain \neghinge{d\statprob_1(\statsample) -
    d\statprob_2(\statsample)} \\ & = \left(\sup_{\statsample \in
    \statdomain} \channeldensity( \channelval \mid \statsample) -
  \inf_{\statsample \in \statdomain} \channeldensity(\channelval \mid
  \statsample)\right) \int_\statdomain \hinge{d
    \statprob_1(\statsample) - d\statprob_2(\statsample)}.
\end{align*}
By definition of the total variation norm, we have $\int
\hinge{d\statprob_1 - d\statprob_2} = \tvnorm{\statprob_1 -
  \statprob_2}$, and hence
\begin{equation}
\label{EqnAuxTwo}
  |\margindens_1(\channelval) - \margindens_2(\channelval)|
  \le \sup_{\statsample, \statsample'} \left|\channeldensity(\channelval \mid
  \statsample) - \channeldensity(\channelval \mid \statsample')\right|
  \tvnorm{\statprob_1 - \statprob_2}.
\end{equation}
For any $\hat{\statsample} \in \statdomain$, we may add and subtract
$\channeldensity(\channelval \mid \hat{\statsample})$ from the quantity inside the
supremum, which implies that
\begin{align*}
  \sup_{\statsample, \statsample'}
  \left|\channeldensity(\channelval \mid \statsample)
  - \channeldensity(\channelval \mid \statsample')\right|
  & = \inf_{\hat{\statsample}}
  \sup_{\statsample, \statsample'}
  \left|\channeldensity(\channelval \mid \statsample)
  - \channeldensity(\channelval \mid \hat{\statsample})
  + \channeldensity(\channelval \mid \hat{\statsample})
  - \channeldensity(\channelval \mid \statsample')\right| \\
  & \le 2 \inf_{\hat{\statsample}}
  \sup_\statsample \left|\channeldensity(\channelval \mid \statsample)
  - \channeldensity(\channelval \mid \hat{\statsample})\right| \\
  & = 2 \inf_{\hat{\statsample}}
  \channeldensity(\channelval \mid \hat{\statsample})
  \sup_\statsample \left|\frac{
    \channeldensity(\channelval \mid \statsample)}{
    \channeldensity(\channelval \mid \hat{\statsample})}
  - 1\right|.
\end{align*}
Similarly, we have for any $\statsample, \statsample'$
\begin{equation*}
  |\channeldensity(\channelval \mid \statsample)
  - \channeldensity(\channelval \mid \statsample')|
  = \channeldensity(\channelval \mid \statsample')
  \left|\frac{\channeldensity(\channelval \mid \statsample)}{
    \channeldensity(\channelval \mid \statsample')} - 1\right|
  \le e^\diffp \inf_{\what{\statsample}}
  \channeldensity(\channelval \mid \what{\statsample})
  \left|\frac{\channeldensity(\channelval \mid \statsample)}{
    \channeldensity(\channelval \mid \statsample')} - 1\right|.
\end{equation*}
Since for any choice of $\statsample, \hat{\statsample}$, we have
$\channeldensity(\channelval \mid \statsample)/ \channeldensity(\channelval
\mid \hat{\statsample}) \in [e^{-\diffp}, e^\diffp]$, we
find that (since $e^\diffp - 1 \ge 1 - e^{-\diffp}$)
\begin{equation*}
  \sup_{\statsample, \statsample'} \left|\channeldensity(\channelval \mid
  \statsample) - \channeldensity(\channelval \mid \statsample')\right|
  \le \min\{2, e^\diffp\}
  \inf_\statsample \channeldensity(\channelval \mid \statsample) \left(e^\diffp
  - 1\right).
\end{equation*}
Combining with the earlier inequality~\eqref{EqnAuxTwo} yields the
claim~\eqref{eqn:channel-diff-tv}.

To see Lemma~\ref{lemma:log-ratio-inequality}, note that for any $x > 0$, the
concavity of the logarithm implies that
\begin{equation*}
    \log(x) \le x - 1.
\end{equation*} 
Setting alternatively $x = a/b$ and $x = b/a$, we obtain the inequalities
\begin{equation*}
  \log\frac{a}{b}
  \le \frac{a}{b} - 1
  = \frac{a - b}{b}
  ~~~ \mbox{and} ~~~
  \log \frac{b}{a}
  \le \frac{b}{a} - 1
  = \frac{b - a}{a}.
\end{equation*}
Using the first inequality for $a \ge b$ and the second for $a < b$ completes
the proof.

\subsection{Proof of Corollary~\ref{corollary:idiot-fano}}
\label{sec:proof-idiot-fano}

Let us recall the definition of the induced marginal
distribution~\eqref{eqn:marginal-channel}, given by
\begin{align*}
  \marginprob_\packval(S) = \int_{\Xspace} \channelprob(S \mid
  \statsample_{1:n}) d \statprob_\packval^n(\statsample_{1:n}) \quad
  \mbox{for $S \in \sigma(\channeldomain^n)$.}
\end{align*}
For each $i = 2, \ldots, \numobs$, we let $\marginprob_{\packval,
  i}(\cdot \mid \channelrv_1 = \channelval_1, \ldots, \channelrv_{i-1}
= \channelval_{i-1}) = \marginprob_{\packval, i}(\cdot \mid
\channelval_{1:i-1})$ denote the (marginal over
$\statrv_i$) distribution of the variable $\channelrv_i$ conditioned
on \mbox{$\channelrv_1 = \channelval_1, \ldots, \channelrv_{i-1} =
  \channelval_{i-1}$}. In addition, use the shorthand notation
\begin{align*}
 \dkl{\marginprob_{\packval, i} }{ \marginprob_{\altpackval, i}} &
 \defeq \int_{\channeldomain^{i-1}} \dkl{\marginprob_{\packval,
     i}(\cdot \mid \channelval_{1:i-1})}{
   \marginprob_{\altpackval, i}(\cdot \mid \channelval_{1:i-1})}
 d \marginprob_\packval^{i-1}(\channelval_1, \ldots, \channelval_{i-1})
\end{align*}
to denote the integrated KL divergence of the conditional distributions on
the $\channelrv_i$. By the chain-rule for KL divergences~\cite[Chapter
  5.3]{Gray90}, we obtain
\begin{equation*}
  \dkl{\marginprob_\packval^n}{\marginprob_{\altpackval}^n} =
  \sum_{i=1}^n \dkl{\marginprob_{\packval,
      i}}{\marginprob_{\altpackval, i}}.
\end{equation*}

By assumption~\eqref{eqn:local-privacy}, the distribution
$\channelprob_i(\cdot \mid \statrv_i, \channelrv_{1:i-1})$ on
$\channelrv_i$ is $\diffp$-differentially private for the sample
$\statrv_i$.  As a consequence, if we let \mbox{$\statprob_{\packval,
    i}(\cdot \mid \channelrv_1 = \channelval_1, \ldots,
  \channelrv_{i-1} = \channelval_{i-1})$} denote the conditional
distribution of $\statrv_i$ given the first $i-1$ values
$\channelrv_1, \ldots, \channelrv_{i-1}$ and the packing index
$\packrv = \packval$, then from the chain rule and
Theorem~\ref{theorem:master} we obtain
\begin{align*}
  \dkl{\marginprob_\packval^n}{\marginprob_{\altpackval}^n}
  & = \sum_{i = 1}^n \int_{\channeldomain^{i-1}}
  \dkl{\marginprob_{\packval,i}(\cdot \mid \channelval_{1:i-1})}{
    \marginprob_{\altpackval,i}(\cdot \mid \channelval_{1:i-1})}
  d \marginprob_{\packval}^{i-1}(\channelval_{1:i-1}) \\
  & \le \sum_{i=1}^n 4(e^\diffp - 1)^2 \int_{\channeldomain^{i-1}}
  \tvnorm{\statprob_{\packval, i}(\cdot \mid \channelval_{1:i-1})
    - \statprob_{\altpackval, i}(\cdot \mid \channelval_{1:i-1})}^2
  d \marginprob_\packval^{i-1}(\channelval_1, \ldots, \channelval_{i-1}).
\end{align*}
By the construction of our sampling scheme, the random variables $\statrv_i$
are conditionally independent given $\packrv = \packval$; thus the
distribution $\statprob_{\packval, i}(\cdot \mid \channelval_{1:i-1}) =
\statprob_{\packval, i}$, where $\statprob_{\packval, i}$ denotes the
distribution of $\statrv_i$ conditioned on $\packrv = \packval$.
Consequently, we have
\begin{equation*}
  \tvnorm{\statprob_{\packval, i}(\cdot \mid \channelval_{1:i-1})
    - \statprob_{\altpackval, i}(\cdot \mid \channelval_{1:i-1})}
  = \tvnorm{\statprob_{\packval, i} - \statprob_{\altpackval, i}},
\end{equation*}
which gives the claimed result.


\subsection{Proof of Proposition~\ref{proposition:location-family-bound}}
\label{sec:proof-location-family}

The minimax rate characterized by
equation~\eqref{eqn:location-family-bound} involves both a lower and
an upper bound, and we divide our proof accordingly.  We provide
the proof for $\diffp \in (0,1]$, but note that a similar result
(modulo different constants) holds for any finite value of $\diffp$.

\paragraphc{Lower bound} We use Le Cam's method to prove the lower
bound in equation~\eqref{eqn:location-family-bound}.  Fix a given
constant $\delta \in (0,1]$, with a precise value to be specified
  later.  For $\packval \in \packset \in \{-1, 1\}$, define the
  distribution $\statprob_\packval$ with support $\{-\delta^{-1/k}, 0,
  \delta^{1/k}\}$ by
\begin{equation*}
  \statprob_\packval(\statrv = \delta^{-1/k}) = \frac{\delta(1 +
    \packval)}{2}, ~~~ \statprob_\packval(\statrv = 0) = 1 - \delta,
  \quad \mbox{and} \quad \statprob_\packval(\statrv = -\delta^{-1/k})
  = \frac{\delta(1 - \packval)}{2}.
\end{equation*}
By construction, we have $\E[|\statrv|^k] = \delta (\delta^{-1/k})^k =
1$ and $\optvar_\packval = \E_\packval[\statrv] =
\delta^{\frac{k-1}{k}} \packval$, whence the mean difference is given
by \mbox{$\optvar_1 - \optvar_{-1} = 2 \delta^{\frac{k-1}{k}}$.}
Applying Le Cam's method~\eqref{eqn:le-cam} and the minimax
bound~\eqref{eqn:estimation-to-testing} yields
\begin{equation*}
  \minimax_n(\optdomain, (\cdot)^2, \channelprob) \ge
  \left(\delta^{\frac{k-1}{k}}\right)^2
  \left(\half - \half \tvnorm{\marginprob_1^n -
    \marginprob_{-1}^n}\right),
\end{equation*}
where $\marginprob_\packval^n$ denotes the marginal distribution of
the samples $\channelrv_1, \ldots, \channelrv_n$ conditioned on
$\optvar = \optvar_\packval$.  Now Pinsker's inequality implies that
$\tvnorm{\marginprob_1^n - \marginprob_{-1}^n}^2 \le \half
\dkl{\marginprob_1^n}{\marginprob_{-1}^n}$, and
Corollary~\ref{corollary:idiot-fano} yields
\begin{equation*}
  \dkl{\marginprob_1^n}{\marginprob_{-1}^n} \le 4 (e^\diffp - 1)^2 n
  \tvnorm{\statprob_1 - \statprob_{-1}}^2 = 4 (e^\diffp - 1)^2 n
  \delta^2.
\end{equation*}
Putting together the pieces yields $\tvnorm{\marginprob_1^n -
  \marginprob_{-1}^n} \le (e^\diffp - 1) \delta \sqrt{2 n}$.  For
$\diffp \in (0,1]$, we have $e^\diffp - 1 \le 2 \diffp$, and thus our
  earlier application of Le Cam's method implies
\begin{equation*}
  \minimax_n(\optdomain, (\cdot)^2, \diffp) \ge
  \left(\delta^{\frac{k-1}{k}}\right)^2 \left(\half - \diffp \delta
  \sqrt{2n} \right).
\end{equation*}
Substituting $\delta = \min\{1, 1 / \sqrt{32 n \diffp^2}\}$ yields the
claim~\eqref{eqn:location-family-bound}.


\paragraphc{Upper bound} We must demonstrate an $\diffp$-locally
private conditional distribution $\channelprob$ and an estimator that
achieves the upper bound in
equation~\eqref{eqn:location-family-bound}.  We do so via a
combination of truncation and addition of Laplacian noise.  Define the
truncation function $\truncate{\cdot}{T} : \R \rightarrow [-T, T]$ by
\begin{equation*}
  \truncate{\statsample}{T} \defeq \max\{-T, \min\{\statsample, T\}\},
\end{equation*}
where the truncation level $T$ is to be chosen.  Let $W_i$ be
independent $\laplace(\diffp / (2T))$ random variables, and for each
index $i = 1, \ldots, \numobs$, define
$\channelrv_i \defeq \truncate{\statrv_i}{T} + W_i$.  By construction,
the random variable $\channelrv_i$ is $\diffp$-differentially private
for $\statrv_i$.  For the mean estimator $\what{\optvar} \defeq
\frac{1}{n} \sum_{i=1}^n \channelrv_i$, we have
\begin{equation}
  \label{eqn:location-mse}
  \E\left[(\what{\optvar} - \optvar)^2\right] = \var(\what{\optvar}) +
  \left(\E[\what{\optvar}] - \optvar\right)^2 = \frac{4T^2}{n
    \diffp^2} + \frac{1}{n} \var(\truncate{\statrv_1}{T}) +
  \left(\E[\channelrv_1] - \optvar\right)^2.
\end{equation}
We claim that
\begin{equation}
  \label{eqn:truncated-mean}
  \E[\channelrv] = \E\left[\truncate{\statrv}{T}\right] \in
  \left[\E[\statrv] - \frac{1}{(k - 1)T^{k - 1}}, \E[\statrv] +
    \frac{1}{(k - 1)T^{k - 1}}\right].
\end{equation}
Indeed, by the assumption that $\E[|\statrv|^k] \le 1$, we have by a
change of variables that
\begin{equation*}
\int_T^\infty \statsample d\statprob(\statsample) = \int_T^\infty
\statprob(\statrv \ge \statsample) d \statsample \le \int_T^\infty
\frac{1}{\statsample^k} d\statsample = \frac{1}{(k - 1) T^{k - 1}}.
\end{equation*}
Thus
\begin{align*}
\E[ \truncate{\statrv}{T}] \ge \E[\min\{\statrv, T\}] & = \E[
  \min\{\statrv, T\} + \hinge{\statrv - T} - \hinge{\statrv - T}] \\
& = \E[\statrv] - \int_T^\infty (\statsample - T) d
\statprob(\statsample) \ge \E[\statrv] - \frac{1}{(k - 1) T^{k - 1}}.
\end{align*}
A similar argument yields the upper bound in
equation~\eqref{eqn:truncated-mean}.

From the bound~\eqref{eqn:location-mse} and the inequalities that
since $\truncate{\statrv}{T} \in [-T, T]$ and $\diffp^2 \le 1$, we
have
\begin{equation*}
  \E \left[(\what{\optvar} - \optvar)^2\right] \le \frac{5 T^2}{n
  \diffp^2} + \frac{1}{(k - 1)^2 T^{2k - 2}} \quad \mbox{valid for any $T > 0$.}
\end{equation*}
Choosing $T = (5(k - 1))^{-\frac{1}{2k}} (n \diffp^2)^{1/(2k)}$ yields
\begin{align*}
\E \left[(\what{\optvar} - \optvar)^2\right] & \le \frac{5 (5(k -
  1))^{\frac{-1}{k}} (n \diffp^2)^{\frac{1}{k}}}{n \diffp^2} +
\frac{1}{(k - 1)^2(5 (k - 1))^{-1 + 1/k} (n \diffp^2)^{1 - 1/k}} \\ 
& = 5^{1 - 1/k}\left(1 + \frac{1}{k - 1}\right) \frac{1}{(k -
  1)^{\frac{1}{k}} (n \diffp^2)^{1 - \frac{1}{k}}}.
\end{align*}
Since $(1 + (k - 1)^{-1})(k - 1)^{-\frac{1}{k}} < (k - 1)^{-1} + (k -
1)^{-2}$ for $k \in (1, 2)$ and is bounded by $1 + (k - 1)^{-1} \le 2$
for $k \in [2, \infty]$, the upper
bound~\eqref{eqn:location-family-bound} follows.


\subsection{Proof of Proposition~\ref{proposition:fixed-design}}
\label{sec:proof-fixed-design}

We now turn to the proof of minimax rates for fixed design linear
regression.

\simplefixed{
  \renewcommand{\sphere}{\mathbb{S}}

  \paragraphc{Lower bound}
  
  We use a slight generalization of the $\diffp$-private
  form~\eqref{eqn:local-fano-private} of the local Fano inequality
  from Corollary~\ref{CorPrivateFano}.  For concreteness, we assume
  throughout that $\diffp \in [0, \frac{23}{35}]$, but analogous
  arguments hold for any bounded $\diffp$ with changes only in the
  constant pre-factors. Consider an instance of the linear regression
  model~\eqref{eqn:regression-model} in which the noise variables
  $\{\varepsilon_i\}_{i=1}^\numobs$ are drawn i.i.d.\ from the uniform
  distribution on $[-\sigma, + \sigma]$.  Our first step is to
  construct a suitable packing of the unit sphere $\sphere^{d-1} = \{u
  \in \R^d : \ltwo{u} = 1\}$ in $\ell_2$-norm:
  \begin{lemma}
    \label{lemma:sphere-packing}
    There exists a $1$-packing $\packset = \{\packval^1, \ldots,
    \packval^N\}$ of the unit sphere $\sphere^{d-1}$ with
    $N \ge \exp(d / 8)$.
  \end{lemma}
  \noindent   
  See Appendix~\ref{AppLemSpherePack} for the proof of this claim. \\

  For a fixed $\delta \in \openleft{0}{1}$ to be chosen shortly, define
  the family of vectors $\{\theta_{\packval}, \packval \in \packset\}$
  with $\theta_\packval \defn \delta \packval$.  Since $\ltwo{\packval}
  \le 1$, we have $\ltwo{\optvar_\packval - \optvar_{\altpackval}} \le 2
  \delta$.  Let $\statprob_{\packval, i}$ denote the distribution of
  $Y_i$ conditioned on $\optvar^* = \optvar_\packval$.  By the form of
  the linear regression model~\eqref{eqn:regression-model} and our
  assumption on the noise variable $\varepsilon_i$,
  $\statprob_{\packval,i}$ is uniform on the interval
  $[\<\optvar_\packval, x_i\> - \stddev, \<\optvar_\packval, x_i\> +
    \stddev]$.  Consequently, for $\packval \neq \altpackval \in
  \packset$, we have
  \begin{align*}
    \tvnorm{\statprob_{\packval, i} - \statprob_{\altpackval, i}} & =
    \half \int |\statdensity_{\packval, i}(y) -
    \statdensity_{\altpackval, i}(y)| dy \\ & \le \half\left[
      \frac{1}{2 \stddev} |\<\optvar_\packval, x_i\> -
      \<\optvar_{\altpackval}, x_i\>| + \frac{1}{2 \stddev}
      |\<\optvar_\packval, x_i\> - \<\optvar_{\altpackval}, x_i\>|
      \right] = \frac{1}{2 \stddev} \left|\<\optvar_\packval -
    \optvar_{\altpackval}, x_i\>\right|.
  \end{align*}
  Letting $\packrv$ denote a random sample from the uniform distribution
  on $\packset$, Corollary~\ref{corollary:idiot-fano} implies that the
  mutual information is upper bounded as
  \begin{align*}
    \information(\channelrv_1, \ldots, \channelrv_n; \packrv) & \le
    4(e^\diffp - 1)^2 \sum_{i=1}^n \frac{1}{|\packset|^2}
    \sum_{\packval, \altpackval \in \packset}
    \tvnorm{\statprob_{\packval, i} - \statprob_{\altpackval, i}}^2 \\
    & \le \frac{(e^\diffp - 1)^2}{\stddev^2} \sum_{i=1}^n
    \frac{1}{|\packset|^2} \sum_{\packval, \altpackval \in \packset}
    \left(\<\optvar_\packval - \optvar_{\altpackval}, x_i\>\right)^2
    \\
    & = \frac{(e^\diffp - 1)^2}{\stddev^2} \frac{1}{|\packset|^2}
    \sum_{\packval, \altpackval \in \packset} (\optvar_\packval -
    \optvar_{\altpackval})^\top X^\top X (\optvar_\packval -
    \optvar_{\altpackval}).
  \end{align*}
  Since $\optvar_\packval = \delta \packval$, we have by definition of
  the maximum singular value that
  \begin{equation*}
    (\optvar_\packval - \optvar_{\altpackval})^\top X^\top X
    (\optvar_\packval - \optvar_\altpackval) \le \delta^2 \ltwo{\packval
      - \altpackval}^2 \singval_{\max}(X^\top X) \le 4 \delta^2
    \singval^2_{\max}(X) = 4 n \delta^2 \singval^2_{\max}(X / \sqrt{n}).
  \end{equation*}
  Putting together the pieces, we find that
  \begin{equation*}
    \information(\channelrv_1, \ldots, \channelrv_n; \packrv)
    \le \frac{4 n \delta^2 (e^\diffp - 1)^2}{\stddev^2}
    \singval^2_{\max}(X / \sqrt{n})
    \le \frac{8 n \diffp^2 \delta^2}{\stddev^2} \singval^2_{\max}(X
    / \sqrt{n}),
  \end{equation*}
  where the second inequality is valid for $\diffp \in [0,
    \frac{23}{35}]$.  Consequently, Fano's inequality combined with the
  packing set $\packset$ from Lemma~\ref{lemma:sphere-packing} implies
  that
  \begin{equation*}
    \minimax_n\left(\Theta, \ltwo{\cdot}^2, \diffp\right) \ge
    \frac{\delta^2}{4} \left(1 - \frac{8 n \delta^2 \diffp^2
      \singval_{\max}^2(X / \sqrt{n})
      / \stddev^2 + \log 2}{d / 8}\right).
  \end{equation*}
  We split the remainder of the analysis into cases.

  \noindent \emph{Case 1:} First suppose that $d \geq 16$. Then setting
  $\delta^2 = \min \{1, \frac{d \stddev^2}{128 n \singval_{\max}^2(X /
    \sqrt{n})}\}$ implies that
  \begin{equation*}
    \frac{8 n \delta^2 \diffp^2 \singval_{\max}^2(X / \sqrt{n}) /
      \stddev^2 + \log 2}{d / 8} \le 8 \left[\frac{\log 2}{d} +
      \frac{64}{128}\right] < \frac{7}{8}.
  \end{equation*}
  As a consequence, we have the lower bound
  \begin{equation*}
    \minimax_n\left(\Theta, \ltwo{\cdot}^2, \diffp\right) \ge \frac{1}{4}
    \min\left\{1, \frac{d \stddev^2}{128 n \singval_{\max}^2(X /
      \sqrt{n})} \right\} \cdot \frac{1}{8},
  \end{equation*}
  which yields the claim for $d \geq 16$. \\

  \noindent \emph{Case 2:} Otherwise, we may assume that $d < 16$.  In
  this case, e a lower bound for the case $d = 1$ is sufficient, since
  apart from constant factors, the same bound holds for all $d < 16$.
  We use the Le Cam method based on a two point comparison.  Indeed, let
  $\theta_1 = \delta$ and $\theta_2 = -\delta$ so that the total
  variation distance is at upper bounded $\tvnorm{\statprob_{1,i} -
    \statprob_{2, i}} \le \frac{\delta}{\stddev} |x_i|$. By
  Corollary~\ref{CorPrivateLeCam}, we have
  \begin{equation*}
    \minimax_n\left(\Theta, (\cdot)^2, \diffp\right) \ge \delta^2
    \left(\half - \delta \frac{(e^\diffp - 1)}{\stddev}
    \bigg(\sum_{i=1}^n x_i^2\bigg)^\half \right).
  \end{equation*}
  Letting $x = (x_1, \ldots, x_n)$ and setting $\delta^2 = \min\{1,
  \stddev^2 / (16 (e^\diffp - 1)^2 \ltwo{x}^2)\}$ gives the desired
  result.
}{
} 

\paragraphc{Upper bound} We now turn to the upper bound, for which we
need to specify a private conditional $\channelprob$ and an estimator
$\thetahat$ that achieves the stated upper bound on the mean-squared
error.  Let $W_i$ be independent $\laplace(\diffp / (2\stddev))$
random variables. Then the additively perturbed random variable
$\channelrv_i = Y_i + W_i$ is $\diffp$-differentially private for
$Y_i$, since by assumption the response $Y_i \in [\<\optvar,
  \statsample_i\> - \stddev, \<\optvar, \statsample_i\> +
  \stddev]$. We now claim that the standard least-squares estimator of
$\optvar^*$ achieves the stated upper bound.  Indeed, the
least-squares estimate is given by
\begin{equation*}
  \what{\optvar} = (X^\top X)^{-1} X^\top Y = (X^\top X)^{-1} X^\top(X
  \optvar^* + \varepsilon + W).
\end{equation*}
Moreover, from the independence of $W$ and $\varepsilon$, we have
\begin{equation*}
\E \left [\ltwos{\what{\optvar} - \optvar^*}^2\right] = \E \left
   [\ltwos{(X^\top X)^{-1} X^\top (\varepsilon + W)}^2\right] = \E
   \left [\ltwos{(X^\top X)^{-1} X^\top \varepsilon}^2\right] + \E
   \left [\ltwos{(X^\top X)^{-1} X^\top W)}^2\right].
\end{equation*}
\simplefixed{ Since $\varepsilon \in [-\stddev, \stddev]^n$, we know
  that $\E[\varepsilon \varepsilon^\top] \preceq \stddev^2 I_{n \times
    n}$, and for the given choice of $W$, we have $\E[WW^\top] = (4
  \stddev^2 / \diffp^2) I_{n \times n}$. Since $\diffp \le 1$, we thus
  find
  \begin{equation*}
    \E\left[\ltwos{\what{\optvar} - \optvar^*}^2\right]
    \le \frac{5\stddev^2 }{\diffp^2}
    \tr\left(X(X^\top X)^{-2} X^\top\right)
    = \frac{5\stddev^2}{\diffp^2}
    \tr\left((X^\top X)^{-1}\right).
  \end{equation*}
  Noting that $\tr((X^\top X)^{-1}) \le d / \singval_{\min}^2(X)
  = d / n \singval_{\min}^2(X / \sqrt{n})$ gives the
  claimed upper bound.
}{
  Since $\varepsilon \in [-\stddev,
    \stddev]^n$, we know that $\E[\varepsilon \varepsilon^\top] \preceq
  \stddev^2 I_{n \times n}$, and similarly $\E[WW^\top] = (4 \stddev^2 /
  \diffp^2) I_{n \times n}$. Since $\diffp \le 1$, we thus find
  \begin{equation*}
    \E\left[\ltwos{\what{\optvar} - \optvar^*}^2\right]
    \le \frac{5\stddev^2 }{\diffp^2}
    \tr\left(X(X^\top X)^{-2} X^\top\right)
    = \frac{5\stddev^2}{\diffp^2}
    \tr\left((X^\top X)^{-1}\right),
  \end{equation*}
  which corresponds to the claimed upper bound with $c_u = 5$.
}

\section{Proof of Theorem~\ref{theorem:super-master} and related results}

In this section, we collect together the proof of
Theorem~\ref{theorem:super-master} and related corollaries.

\subsection{Proof of Theorem~\ref{theorem:super-master}}
\label{sec:proof-super-master}

Let $\channeldomain$ denote the domain of the random variable
$\channelrv$.  We begin by reducing the problem to the case when
$\channeldomain = \{1, 2, \ldots, k \}$ for an arbitrary positive
integer $k$.  Indeed, in the general setting, we let $\mc{K} =
\{K_i\}_{i=1}^k$ be any (measurable) finite partition of
$\channeldomain$, where for $z \in \channeldomain$ we let $[z]_\mc{K}
= K_i$ for the $K_i$ such that $z \in K_i$.  The KL divergence
$\dkl{\marginprob_\packval}{\meanmarginprob}$ can be defined as the
supremum of the (discrete) KL divergences between the random variables
$[\channelrv]_{\mc{K}}$ sampled according to $\marginprob_\packval$
and $\meanmarginprob$ over all partitions $\mc{K}$ of
$\channeldomain$; for instance, see~\citet[Chapter~5]{Gray90}.
Consequently, we can prove the claim for $\channeldomain = \{1, 2,
\ldots, k\}$, and then take the supremum over $k$ to recover the
general case.  Accordingly, we can work with the probability mass
functions $\margindens(\channelval \mid \packval) =
\marginprob_\packval(\channelrv = \channelval)$ and
$\meanmargindensity(\channelval) = \meanmarginprob(\channelrv =
\channelval)$, and we may write
\begin{equation}
  \label{eqn:kl-to-finite}
  \dkl{\marginprob_\packval}{\meanmarginprob} +
  \dkl{\meanmarginprob}{\marginprob_\packval} = \sum_{\channelval =
    1}^k \left(\margindens(\channelval \mid \packval) -
  \meanmargindensity(\channelval)\right) \log
  \frac{\margindens(\channelval \mid \packval)}{
    \meanmargindensity(\channelval)}.
\end{equation}
Throughout, we will also use (without loss of generality) the probability mass
functions $\channeldensity(\channelval \mid \statsample) =
\channelprob(\channelrv = \channelval \mid \statrv = \statsample)$, where we
note that $\margindens(\channelval \mid \packval) = \int
\channeldensity(\channelval \mid \statsample)
d\statprob_\packval(\statsample)$.

Now we use Lemma~\ref{lemma:log-ratio-inequality} from the proof of
Theorem~\ref{theorem:master} to complete the proof of
Theorem~\ref{theorem:super-master}.
Starting with equality~\eqref{eqn:kl-to-finite}, we have
\begin{align*}
  \frac{1}{|\packset|}
  \sum_{\packval \in \packset}
  \left[\dkl{\marginprob_\packval}{\meanmarginprob}
    + \dkl{\meanmarginprob}{\marginprob_\packval}\right]
  & \le
  \sum_{\packval \in \packset}
  \frac{1}{|\packset|}
  \sum_{\channelval = 1}^k
  \left|\margindens(\channelval \mid \packval)
  - \meanmargindensity(\channelval)\right|
  \left|\log \frac{\margindens(\channelval \mid \packval)}{
    \meanmargindensity(\channelval)}\right| \\
  & \le
  \sum_{\packval \in \packset}
  \frac{1}{|\packset|}
  \sum_{\channelval = 1}^k
  \left|\margindens(\channelval \mid \packval)
  - \meanmargindensity(\channelval)\right|
  \frac{\left|\margindens(\channelval \mid \packval)
    - \meanmargindensity(\channelval)\right|}{
    \min\left\{\meanmargindensity(\channelval),
    \margindens(\channelval \mid \packval)\right\}}.
\end{align*}
Now, we define the measure $\margincenter$ on $\channeldomain = \{1, \ldots,
k\}$ by $\margincenter(\channelval) \defeq \inf_{\statval \in \statdomain}
\channeldensity(\channelval \mid \statval)$.  It is clear that
$\min\left\{\meanmargindensity(\channelval), \margindens(\channelval \mid
\packval)\right\} \ge \margincenter(\channelval)$, whence we find
\begin{align*}
  \frac{1}{|\packset|} \sum_{\packval \in \packset}
  \left[\dkl{\marginprob_\packval}{\meanmarginprob} +
    \dkl{\meanmarginprob}{\marginprob_\packval}\right]
  & \le \sum_{\packval \in
    \packset} \frac{1}{|\packset|} \sum_{\channelval = 1}^k
  \frac{\left(\margindens(\channelval \mid \packval) -
    \meanmargindensity(\channelval)\right)^2}{
    \margincenter(\channelval)}.
\end{align*}
It remains to bound the final sum.  For any constant $c \in \R$, we
have
\begin{equation*}
  \margindens(\channelval \mid \packval) - \meanmargindensity(\channelval)
  = \int_\statdomain \left( \channeldensity(\channelval \mid \statsample) -
  c\right) \left(d\statprob_\packval(\statsample) -
  d\meanstatprob(\statsample)\right).
\end{equation*}
We define a set of functions $f : \channeldomain \times \statdomain
\rightarrow \R$ (depending implicitly on $\channeldensity$) by
\begin{equation*}
  \mc{F}_\diffp \defeq \left\{
  f \, \mid f(\channelval, \statsample) \in [1, e^\diffp]
  \margincenter(\channelval)
  ~ \mbox{for~all~} \channelval
  \in \channeldomain ~\mbox{and}~ \statsample \in \statdomain \right\}.
\end{equation*}
By the definition of differential privacy, when viewed as a joint mapping from
$\channeldomain \times \statdomain \rightarrow \R$, the conditional
p.m.f.\ $\channeldensity$ satisfies $\{(\channelval, \statsample) \mapsto
\channeldensity(\channelval \mid \statsample)\} \in \mc{F}_\diffp$.
Since constant (with respect to $\statsample$) shifts do not change the above
integral, we can modify the range of functions in $\mc{F}_\diffp$ by
subtracting $\margincenter(\channelval)$
from each, yielding the set
\begin{equation*}
  \mc{F}'_\diffp
  \defeq \left\{
  f \, \mid f(\channelval, \statsample) \in \left[0, e^\diffp - 1\right]
  \margincenter(\channelval)
  ~ \mbox{for~all~} \channelval
  \in \channeldomain ~\mbox{and}~ \statsample \in \statdomain \right\}.
\end{equation*}
As a consequence, we find that
\begin{align*}
  \sum_{\packval \in \packset} 
  \left(\margindens(\channelval \mid \packval) -
  \meanmargindensity(\channelval)
  \right)^2
  & \le \sup_{f \in \mc{F}_\diffp}
  \left\{ \sum_{\packval\in \packset} 
  \left(\int_\statdomain f(\channelval, \statsample)
  \left(d\statprob_\packval(\statsample) -
  d\meanstatprob(\statsample)\right) \right)^2\right\} \\
  & = \sup_{f \in \mc{F}'_\diffp}
  \left\{ \sum_{\packval\in \packset} 
  \left(\int_\statdomain \left( f(\channelval, \statsample) -
  \margincenter(\channelval) \right)
  \left(d\statprob_\packval(\statsample) -
  d\meanstatprob(\statsample)\right) \right)^2\right\}.
\end{align*}
By inspection, when we divide by $\margincenter(\channelval)$ and recall the
definition of the set $\linfset \subset L^\infty(\statdomain)$ in the
statement of Theorem~\ref{theorem:super-master}, we obtain
\begin{equation*}
  \sum_{\packval \in \packset}
  \left(\margindens(\channelval \mid \packval) 
  - \meanmargindensity(\channelval)\right)^2
  \le \left(\margincenter(\channelval)\right)^2
  (e^\diffp - 1)^2 \sup_{\optdens \in \linfset}
  \sum_{\packval \in \packset} \left(\int_\statdomain \optdens
  (\statsample) \left( d\statprob_\packval(\statsample) -
  d\meanstatprob(\statsample)\right) \right)^2.
\end{equation*}
Putting together our bounds, we have
\begin{align*}
  \lefteqn{\frac{1}{|\packset|} \sum_{\packval \in \packset}
    \left[\dkl{\marginprob_\packval}{\meanmarginprob} +
      \dkl{\meanmarginprob}{\marginprob_\packval}\right]} \\
  & \quad \le (e^\diffp - 1)^2
  \sum_{\channelval = 1}^k \frac{1}{|\packset|}
  \frac{\left(\margincenter(\channelval)\right)^2}{
    \margincenter(\channelval)} \sup_{\optdens \in \linfset}
  \sum_{\packval \in \packset} \left(\int_\statdomain
  \optdens(\statsample) \left( d\statprob_\packval(\statsample) -
  d\meanstatprob(\statsample)\right) \right)^2 \\
  & \quad \le (e^\diffp - 1)^2
  \frac{1}{|\packset|} \sup_{\optdens \in \linfset} \sum_{\packval \in
    \packset} \left(\int_\statdomain \optdens(\statsample) \left(
  d\statprob_\packval(\statsample) - d\meanstatprob(\statsample)\right)
  \right)^2,
\end{align*}
since $\sum_\channelval \margincenter(\channelval) \le 1$,
which is the statement of the theorem. \\


\subsection{Proof of Corollary~\ref{corollary:super-fano}}
\label{sec:proof-super-fano}

In the non-interactive setting~\eqref{eqn:local-privacy-simple}, the
marginal distribution $\marginprob^n_\packval$ is a product measure
and $\channelrv_i$ is conditionally independent of
$\channelrv_{1:i-1}$ given $\packrv$. Thus by the chain rule for
mutual information~\cite[Chapter 5]{Gray90} and the fact (as in the
proof of Theorem~\ref{theorem:super-master}) that we may assume
w.l.o.g.\ that $\channelrv$ has finite range
\begin{align*}
  \information(\channelrv_1, \ldots, \channelrv_\numobs;
  \packrv)
  & = \sum_{i = 1}^\numobs \information(\channelrv_i; \packrv
  \mid \channelrv_{1:i-1})
  = \sum_{i = 1}^\numobs \left[H(\channelrv_i \mid \channelrv_{1:i-1})
  - H(\channelrv_i \mid \packrv, \channelrv_{1:i-1})\right].
\end{align*}
Since conditioning reduces entropy and $\channelrv_{1:i-1}$ is conditionally
independent of $\channelrv_i$ given $\packrv$, we have $H(\channelrv_i \mid
\channelrv_{1:i-1}) \le H(\channelrv_i)$ and $H(\channelrv_i \mid \packrv,
\channelrv_{1:i-1}) = H(\channelrv_i \mid \packrv)$.
In particular, we have
\begin{equation*}
  \information(\channelrv_1, \ldots, \channelrv_\numobs; \packrv)
  \le \sum_{i = 1}^\numobs \information(\channelrv_i; \packrv)
  = \sum_{i = 1}^\numobs \frac{1}{|\packset|}
  \sum_{\packval \in \packset} \dkl{\marginprob_{\packval,i}}{
    \meanmarginprob_i}.
\end{equation*}
Applying Theorem~\ref{theorem:super-master} completes the proof.


\section{Proof of Theorem~\ref{theorem:sequential-interactive}}
\label{sec:proof-sequential-interactive}

The proof of this theorem combines the techniques we used in the
proofs of Theorems~\ref{theorem:master}
and~\ref{theorem:super-master}; the first handles interactivity, while
the techniques to derive the variational bounds are reminiscent of
those used in Theorem~\ref{theorem:super-master}.  Our first step is
to note a consequence of the independence structure in
Fig.~\ref{fig:interactive-channel} essential to our tensorization
steps. More precisely, we claim that for any set $S \in
\sigma(\channeldomain)$,
\begin{equation}
  \label{eqn:help-tensorize}
  \marginprob_{\pm j}(Z_i \in S \mid \channelval_{1:i-1}) = \int
  \channel(Z_i \in S \mid \channelrv_{1:i-1} = \channelval_{1:i-1},
  \statrv_i = \statval) d \statprob_{\pm j,i}(\statval).
\end{equation}
We postpone the proof of this intermediate claim to the end of this
section.

Now consider the summed KL-divergences. Let $\marginprob_{\pm j,
  i}(\cdot \mid \channelval_{1:i-1})$ denote the conditional
distribution of $\channelrv_i$ under $\statprob_{\pm j}$, conditional
on $\channelrv_{1:i-1} = \channelval_{1:i-1}$. As in the proof of
Corollary~\ref{corollary:idiot-fano}, the chain-rule for
KL-divergences~\cite[e.g.][Chapter 5]{Gray90} implies
\begin{align*}
  \dkl{\marginprob^n_{+j}}{\marginprob^n_{-j}} & = \sum_{i=1}^n
  \int_{\channeldomain^{i-1}} \dkl{\marginprob_{+j}(\cdot \mid
    \channelval_{1:i-1})}{ \marginprob_{-j}(\cdot \mid
    \channelval_{1:i-1})} d
  \marginprob^{i-1}_{+j}(\channelval_{1:i-1}).
\end{align*}
For notational convenience in the remainder of the proof, let us
define the symmetrized KL divergence between measures $M$ and $M'$ as
$\dklsym{M}{M'} = \dkl{M}{M'} + \dkl{M'}{M}$.

Defining $\meanstatprob \defn 2^{-d} \sum_{\packval \in \packset}
\statprob_\packval^n$, we have $2 \meanstatprob = \statprob_{+j} +
\statprob_{-j}$ for each $j$ simultaneously, We also introduce
$\meanmarginprob(S) = \int \channel(S \mid \statval_{1:n})
d\meanmarginprob(\statval_{1:n})$, and let $\E_{\pm j}$ denote the
expectation taken under the marginals $\marginprob_{\pm j}$.  We then
have
\begin{align*}
  \lefteqn{\dkl{\marginprob^n_{+j}}{\marginprob^n_{-j}} +
    \dkl{\marginprob^n_{-j}}{\marginprob^n_{+j}}} \\ 
& = \sum_{i=1}^n \Big(\E_{+j}[ \dkl{\marginprob_{+j,i}(\cdot \mid
      \channelrv_{1:i-1})}{ \marginprob_{-j,i}(\cdot \mid
      \channelrv_{1:i-1})}] + \E_{-j}[ \dkl{\marginprob_{-j,i}(\cdot
      \mid \channelrv_{1:i-1})}{ \marginprob_{+j,i}(\cdot \mid
      \channelrv_{1:i-1})}]\Big) \\ 
& \le \sum_{i=1}^n \Big(\E_{+j}[\dklsym{\marginprob_{+j,i}(\cdot \mid
      \channelrv_{1:i-1})}{\marginprob_{-j,i}( \cdot \mid
      \channelrv_{1:i-1})}] + \E_{-j}[\dklsym{\marginprob_{+j,i}(\cdot
      \mid \channelrv_{1:i-1})}{\marginprob_{-j,i}( \cdot \mid
      \channelrv_{1:i-1})}]\Big) \\ 
& = 2 \sum_{i=1}^n \int_{\channeldomain^{i-1}}
  \dklsym{\marginprob_{+j,i}(\cdot \mid \channelval_{1:i-1})}{
    \marginprob_{-j,i}(\cdot \mid \channelval_{1:i-1})} d
  \meanmarginprob^{i-1}(\channelval_{1:i-1}),
\end{align*}
where we have used the definition of $\meanmarginprob$ and that
$2\meanstatprob = \statprob_{+j} + \statprob_{-j}$ for all $j$.
Summing over $j \in [d]$ yields
\begin{align}
\sum_{j=1}^d \dklsym{\marginprob_{+j}^n}{\marginprob_{-j}^n} & \le 2
\sum_{i=1}^n \int_{\channeldomain^{i-1}} \sum_{j=1}^d \underbrace{
  \dklsym{\marginprob_{+j,i}(\cdot \mid \channelval_{1:i-1})}{
    \marginprob_{-j,i}(\cdot \mid \channelval_{1:i-1})} }_{\eqdef
  \term_{j,i}} d \meanmarginprob^{i-1}(\channelval_{1:i-1}).
  \label{eqn:initial-sum-d}
\end{align}
We bound the underlined expression in inequality~\eqref{eqn:initial-sum-d},
whose elements we denote by $\term_{j,i}$.

Without loss of generality (as in the proof of
Theorem~\ref{theorem:super-master}), we may assume $\channeldomain$ is
finite, and that $\channeldomain = \{1, 2, \ldots, k\}$ for some
positive integer $k$. Using the probability mass functions
$\margindens_{\pm j,i}$ and omitting the index $i$ when it is clear
from context, Lemma~\ref{lemma:log-ratio-inequality} implies
\begin{align*}
  \term_{j,i}
  & = \sum_{\channelval = 1}^k \left(\margindens_{+j}(\channelval
  \mid \channelval_{1:i-1})
  - \margindens_{+j}(\channelval
  \mid \channelval_{1:i-1})\right)\log \frac{\margindens_{+j}(\channelval
  \mid \channelval_{1:i-1})}{\margindens_{-j}(\channelval
  \mid \channelval_{1:i-1})} \\
  & \le \sum_{\channelval = 1}^k \left(\margindens_{+j}(\channelval
  \mid \channelval_{1:i-1})
  - \margindens_{+j}(\channelval
  \mid \channelval_{1:i-1})\right)^2
  \frac{1}{\min\{\margindens_{+j}(\channelval
  \mid \channelval_{1:i-1}), \margindens_{-j}(\channelval
  \mid \channelval_{1:i-1})\}}.
\end{align*}
For each fixed $\channelval_{1:i-1}$, define the infimal measure
$\margincenter(\channelval \mid \channelval_{1:i-1}) \defeq \inf
\limits_{\statval \in \statdomain} \channeldensity(\channelval \mid
\statrv_i = \statval, \channelval_{1:i-1})$.  By construction, we have
$\min\{\margindens_{+j}(\channelval \mid \channelval_{1:i-1}),
\margindens_{-j}(\channelval \mid \channelval_{1:i-1})\} \ge
\margincenter(\channelval \mid \channelval_{1:i-1})$, and hence
\begin{equation*}
\term_{j,i}
  \le \sum_{\channelval = 1}^k \left(\margindens_{+j}(\channelval
  \mid \channelval_{1:i-1})
  - \margindens_{+j}(\channelval
  \mid \channelval_{1:i-1})\right)^2
  \frac{1}{\margindens^0(\channelval \mid \channelval_{1:i-1})}.
\end{equation*}
Recalling equality~\eqref{eqn:help-tensorize}, we have
\begin{align*}
  \margindens_{+j}(\channelval
  \mid \channelval_{1:i-1})
  - \margindens_{+j}(\channelval
  \mid \channelval_{1:i-1})
  & = \int_{\statdomain}
  \channeldensity(\channelval \mid \statval, \channelval_{1:i-1})
  (d \statprob_{+j,i}(\statval) - d \statprob_{-j,i}(\statval)) \\
  & = \margindens^0(\channelval \mid \channelval_{1:i-1})
  \int_{\statdomain}
  \left(\frac{\channeldensity(\channelval \mid \statval, \channelval_{1:i-1})}{
    \margindens^0(\channelval \mid \channelval_{1:i-1})}
  - 1 \right)
  (d \statprob_{+j,i}(\statval) - d \statprob_{-j,i}(\statval)).
\end{align*}

From this point, the proof is similar to that of
Theorem~\ref{theorem:super-master}.  Define the collection of
functions
\begin{equation*}
  \mc{F}_\diffp \defeq \{f : \statdomain \times \channeldomain^i \to
     [0, e^\diffp - 1]\}.
\end{equation*}
Using the definition of differential privacy, we have
$\frac{\channeldensity(\channelval \mid \statval,
  \channelval_{1:i-1})}{ \margindens^0(\channelval \mid
  \channelval_{1:i-1})} \in [1, e^\diffp]$, so there exists $f \in
\mc{F}_\diffp$ such that
\begin{align*}
  \sum_{j=1}^d \term_{j,i}
  & \le \sum_{j=1}^d \sum_{\channelval = 1}^k
  \frac{\left(\margindens^0(\channelval \mid \channelval_{1:i-1})\right)^2}{
    \margindens^0(\channelval \mid \channelval_{1:i-1})}
  \bigg(\int_{\statdomain} f(\statsample, \channelval, \channelval_{1:i-1})
  (d \statprob_{+j,i}(\statval) - d \statprob_{-j,i}(\statval))\bigg)^2 \\
  & = \sum_{\channelval = 1}^k
  \margindens^0(\channelval \mid \channelval_{1:i-1})
  \sum_{j=1}^d
  \bigg(\int_{\statdomain} f(\statsample, \channelval, \channelval_{1:i-1})
  (d \statprob_{+j,i}(\statval) - d \statprob_{-j,i}(\statval))\bigg)^2.
\end{align*}
Taking a supremum over $\mc{F}_\diffp$, we find the further upper bound
\begin{align*}
  \sum_{j=1}^d \term_{j,i} \le \sum_{\channelval = 1}^k
  \margindens^0(\channelval \mid \channelval_{1:i-1})
  \sup_{f \in \mc{F}_\diffp} \sum_{j=1}^d \bigg(
  \int_{\statdomain} f(\statsample, \channelval, \channelval_{1:i-1})
  (d \statprob_{+j,i}(\statval) - d \statprob_{-j,i}(\statval))\bigg)^2.
\end{align*}
The inner supremum may be taken independently
of $\channelval$ and $\channelval_{1:i-1}$, so we
rescale by $(e^\diffp - 1)$ to obtain
our penultimate inequality
\begin{align*}
  \lefteqn{\sum_{j=1}^d  \dklsym{\marginprob_{+j,i}(\cdot
      \mid \channelval_{1:i-1})}{
      \marginprob_{-j,i}(\cdot \mid \channelval_{1:i-1})}} \\
  & \qquad ~ \le (e^\diffp - 1)^2 \sum_{\channelval = 1}^k
  \margindens^0(\channelval \mid \channelval_{1:i-1})
  \sup_{\optdens \in \linfset(\statdomain)}
  \sum_{j=1}^d \bigg( \int_{\statdomain} \optdens(\statval)
  (d \statprob_{+j,i}(\statval) - d \statprob_{-j,i}(\statval))\bigg)^2.
\end{align*}
Noting that $\margindens^0$ sums to a quantity $\le 1$ and substituting
the preceding expression in inequality~\eqref{eqn:initial-sum-d} completes the
proof.\\

\noindent Finally, we return to prove our intermediate marginalization
claim~\eqref{eqn:help-tensorize}.  We have that
\begin{align*}
  \marginprob_{\pm j}(Z_i \in S \mid \channelval_{1:i-1})
  & = \int \channel(Z_i \in S \mid \channelval_{1:i-1}, \statval_{1:n})
  d \statprob_{\pm j}(\statval_{1:n} \mid \channelval_{1:i-1}) \\
  & \stackrel{(i)}{=}
  \int \channel(Z_i \in S \mid \channelval_{1:i-1}, \statval_i)
  d \statprob_{\pm j}(\statval_{1:n} \mid \channelval_{1:i-1}) \\
  & \stackrel{(ii)}{=} \int \channel(Z_i \in S \mid
  \channelrv_{1:i-1} = \channelval_{1:i-1}, \statrv_i = \statval)
  d \statprob_{\pm j,i}(\statval),
\end{align*}
where equality~(i) follows by the assumed conditional independence structure
of $\channel$ (recall Figure~\ref{fig:interactive-channel}) and equality~(ii)
is a consequence of the independence of $\statrv_i$ and $\channelrv_{1:i-1}$
under $\statprob_{\pm j}$. That is, we have $\statprob_{+j}(\statrv_i \in S
\mid \channelrv_{1:i-1} = \channelval_{1:i-1}) = \statprob_{+j,i}(S)$ by the
definition of $\statprob_\packval^n$ as a product and that $\statprob_{\pm j}$
are a mixture of the products $\statprob_\packval^n$.

\fi

\section{Conclusions}

We have linked minimax analysis from statistical decision theory with
differential privacy, bringing some of their respective foundational
principles into close contact. Our main technique, in the form of the
divergence inequalities in Theorems~\ref{theorem:master}
and~\ref{theorem:super-master}, and their
Corollaries~\ref{corollary:idiot-fano}--\ref{corollary:super-fano},
shows that applying differentially private sampling schemes
essentially acts as a contraction on distributions. These contractive
inequalities allow us to give sharp minimax rates for estimation in
locally private settings, and we think such results may be more
generally applicable.  With our examples in
Sections~\ref{sec:mean-estimation}, \ref{sec:multinomial-estimation},
and~\ref{sec:density-estimation}, we have developed a framework
that shows that roughly, if one can construct a family of
distributions $\{\statprob_\packval\}$ on the sample space
$\statdomain$ that is not well ``correlated'' with any member of $f
\in L^\infty(\statdomain)$ for which $f(\statsample) \in \{-1, 1\}$,
then providing privacy is costly: the contraction
Theorems~\ref{theorem:super-master}
and~\ref{theorem:sequential-interactive} provide is strong.

By providing sharp convergence rates for many standard statistical
estimation procedures under local differential privacy, we have
developed and explored some tools that may be used to better
understand privacy-preserving statistical inference. We have
identified a fundamental continuum along which privacy may be traded
for utility in the form of accurate statistical estimates, providing a
way to adjust statistical procedures to meet the privacy or utility
needs of the statistician and the population being sampled.

There are a number of open questions raised by our work. It is natural
to wonder whether it is possible to obtain tensorized inequalities of
the form of Corollary~\ref{corollary:super-fano} even for interactive
mechanisms. Another important question is whether the results we have
provided can be extended to settings in which standard (non-local)
differential privacy holds. Such extensions could yield insights into
optimal mechanisms for differentially private procedures.


\subsection*{Acknowledgments}

\ifdefined\useaosstyle

We are very thankful to Shuheng Zhou for pointing out errors in
Corollaries~\ref{corollary:idiot-fano} and~\ref{corollary:super-fano} in an
earlier version of this manuscript.  We also thank Guy Rothblum for helpful
discussions.  We are also thankful to the associate editor and reviewers for
their constructive feedback.

\begin{supplement}[id=suppA]
  \stitle{Proofs of Results}
  \slink[doi]{COMPLETED BY THE TYPESETTER}
  \sdatatype{.pdf}
  \sdescription{The supplementary material contains proofs of our results}
\end{supplement}


\else 

We are very thankful to Shuheng Zhou for pointing out errors in
Corollaries~\ref{corollary:idiot-fano} and~\ref{corollary:super-fano} in an
earlier version of this manuscript.  We also thank Guy Rothblum for helpful
discussions.  JCD was partially supported by a Facebook Graduate Fellowship
and an NDSEG fellowship.  Our work was supported in part by the U.S.\ Army
Research Office under grant number W911NF-11-1-0391, and Office of Naval
Research MURI grant N00014-11-1-0688.






\appendix

\section{Proofs of multi-dimensional mean-estimation results}
\label{sec:proofs-big-mean-estimation}

At a high level, our proofs of these results
consist of three steps, the first of
which is relatively standard, while the second two exploit specific
aspects of the local privacy setting. We outline them here:

\begin{enumerate}[(1)]
\item \label{step:standard}
  The first step is a standard reduction, based on
  inequalities~\eqref{eqn:estimation-to-testing}--\eqref{eqn:fano}
  in Section~\ref{SecEstimationToTest}, from an
  estimation problem to a multi-way testing problem that involves
  discriminating between indices $\packval$ contained within some subset
  $\packset$ of $\R^d$.
\item \label{item:construct-packing} The second step is an appropriate
  construction of a maximal $\delta$-packing, meaning a set $\packset
  \subset \R^d$ such that each pair is \mbox{$\delta$-separated} and
  the resulting set is as large as possible.  In addition, our
  arguments require that, for a random variable $\packrv$ uniformly
  distributed over $\packset$, the covariance $\cov(\packrv)$ has
  relatively small operator norm.
\item The final step is to apply Theorem~\ref{theorem:super-master} in
  order to control the mutual information associated with the testing
  problem. Doing so requires bounding the supremum in
  Corollary~\ref{corollary:super-fano} via the the operator norm of
  $\cov(\packrv)$.
\end{enumerate}

\noindent The estimation to testing reduction of
Step~\ref{step:standard} was previously described in
Section~\ref{SecEstimationToTest}.  Accordingly, the proofs to follow
are devoted to the second and third steps in each case.


\subsection{Proof of Proposition~\ref{proposition:d-dimensional-mean}}
\label{sec:proof-d-dimensional-mean}

We provide a proof of the lower bound, as we provided the
argument for the
upper bound in Section~\ref{sec:attainability-means}.

\paragraphc{Constructing a good packing}

Let $k$ be an arbitrary integer in $\{1, 2, \ldots, d\}$. The
following auxiliary result provides a building block for the packing
set underlying our proof:
\begin{lemma}
  \label{lemma:hypercube-packing}
  For each integer $k$, there exists a packing $\packset_k$ of the
  $k$-dimensional hypercube $\{-1, 1\}^k$ with $\lone{\packval -
    \altpackval} \ge k/2$ for each $\packval, \altpackval \in
  \packset_k$ with $\packval \neq \altpackval$ such that $|\packset_k|
  \geq \ceil{\exp(k/16)}$, and
  \begin{equation*}
    \frac{1}{|\packset_k|} \sum_{\packval \in \packset_k} \packval
    \packval^\top \preceq 25 I_{k \times k}.
  \end{equation*}
\end{lemma}
\noindent See Appendix~\ref{sec:proof-hypercube-packing} for the
proof.

For a given $k \le d$, we extend the set $\packset_k \subseteq \real^k$ to a
subset of $\R^d$ by setting $\packset = \packset_{k} \times \{0\}^{d -k}$.
For a parameter $\delta \in \openleft{0}{1/2}$ to be chosen, we define a
family of probability distributions $\{\statprob_\packval\}_{\packval \in
  \packset}$ constructively.  In particular, the random vector $X \sim
\statprob_\packval$ (a single observation) is formed by the following
procedure:
\begin{equation}
  \label{eqn:linf-sampling-scheme}
  \mbox{Choose~index~} j \in \{1, \ldots, k\} ~
  \mbox{uniformly~at~random~and~set} ~ \statrv =
  \begin{cases}
    \radius e_j & \mbox{w.p.}~ \frac{1 + \delta \packval_j}{2}
    \\ -\radius e_j & \mbox{w.p.}~ \frac{1 - \delta \packval_j}{2}.
  \end{cases}
\end{equation}
By construction, these distributions have mean vectors
\begin{align*}
  \optvar_\packval \defeq
  \E_{\statprob_\packval}[\statrv]
  = \frac{\delta \radius}{k} \packval.
\end{align*}
Consequently, given the properties of the packing $\packset$, we have
$\statrv \in \ball_1(\radius)$ with probability 1,
and $\ltwo{\optvar_\packval -
  \optvar_{\altpackval}}^2 \ge \radius^2 \delta^2 / k$. Thus we see that
the mean vectors $\{\optvar_\packval\}_{\packval \in \packset}$ provide
us with an $\radius \delta / \sqrt{k}$-packing of the ball.


\paragraphc{Upper bounding the mutual information}

Our next step is to bound the mutual information
$\information(\channelrv_1, \ldots, \channelrv_n; \packrv)$ when the
observations $\statrv$ come from the
distribution~\eqref{eqn:linf-sampling-scheme} and $\packrv$ is uniform
in the set $\packset$.  We have the following lemma, which applies so
long as the channel $\channelprob$ is non-interactive and
$\diffp$-locally private~\eqref{eqn:local-privacy-simple}.  See
Appendix~\ref{sec:linf-info-bound} for the proof.
\begin{lemma}
  \label{lemma:linf-info-bound}
  Fix $k \in \{1, \ldots, d\}$.
  Let $\channelrv_i$ be $\diffp$-locally differentially
  private for $\statrv_i$, and let $\statrv$
  be sampled according to the distribution~\eqref{eqn:linf-sampling-scheme}
  conditional on $\packrv = \packval$. Then
  \begin{equation*}
    \information(\channelrv_1, \ldots, \channelrv_n; \packrv) \le n
    \frac{25 e^\diffp}{16} \frac{\delta^2}{k} (e^\diffp -
    e^{-\diffp})^2.
  \end{equation*}
\end{lemma}

\paragraphc{Applying testing inequalities}

We now show how a combination of the hypercube packing specified by
Lemma~\ref{lemma:hypercube-packing} and the sampling
scheme~\eqref{eqn:linf-sampling-scheme} give us our desired lower bound. Fix
$k \le d$ and let $\packset = \packset_k \times \{0\}^{d-k}$ be the packing of
$\{-1, 1\}^k \times \{0\}^{d-k}$ defined following
Lemma~\ref{lemma:hypercube-packing}.  Combining
Lemma~\ref{lemma:linf-info-bound} and the fact that the vectors
$\optvar_\packval$ provide an $\radius \delta / \sqrt{k}$ packing of the ball
of cardinality at least $\exp(k / 16)$, Fano's inequality implies
that for any $k \in \{1, \ldots, d\}$,
\begin{equation*}
  \minimax_n(\optvar(\mc{\statprob}), \ltwo{\cdot}^2, \diffp)
  \ge \frac{\radius^2 \delta^2}{4 k}
  \left(1 - \frac{25 n e^\diffp \delta^2 (e^\diffp - e^{-\diffp})^2 / (16 k)
    + \log 2}{k / 16}\right)
\end{equation*}
Because of the $1$-dimensional
mean-estimation lower bounds provided in Section~\ref{sec:location-family}, we
may assume w.l.o.g.\ that $k \ge 32$. Setting $\delta_{n,\diffp,k}^2
= \min\{1, k^2 / (50 n e^\diffp (e^\diffp - e^{-\diffp})^2)\}$, we obtain
\begin{equation*}
  \minimax_n(\optvar(\mc{\statprob}), \ltwo{\cdot}^2, \diffp)
  \ge \frac{\radius^2 \delta_{n, \diffp, k}^2}{4 k}
  \left(1 - \half - \frac{\log 2}{2}\right)
  \ge c \radius^2 \min\left\{\frac{1}{k},
  \frac{k}{n e^\diffp (e^\diffp - e^{-\diffp})^2}\right\}
\end{equation*}
for a universal (numerical) constant $c$. Since $e^\diffp (e^\diffp -
e^{-\diffp})^2 < 16 \diffp^2$ for $\diffp \in [0, 1]$,
we obtain the lower bound
\begin{equation*}
  \minimax_n(\optvar(\mc{\statprob}), \ltwo{\cdot}^2, \diffp)
  \ge c \radius^2 \max_{k \in [d]} \left\{ \min\left\{ \frac{1}{k},
  \frac{k}{n \diffp^2}\right\}\right\}
\end{equation*}
for $\diffp \in [0,1]$ and a universal constant $c > 0$.  Setting $k$ in the
preceding display to be the integer in $\{1, \ldots, d\}$ nearest $\sqrt{n
  \diffp^2}$ gives the result of the proposition.


\subsection{Proof of Proposition~\ref{proposition:minimax-mean-linf}}
\label{sec:proof-minimax-mean-linf}

Since the upper bound was established in
Section~\ref{sec:attainability-means}, we focus on the lower bound.

\paragraphc{Constructing a good packing} In this case, the packing set
is very simple: set $\packset = \{\pm e_j\}_{j=1}^d$ so that
$|\packset| = 2d$.  Fix some $\delta \in [0, 1]$, and for $\packval
\in \packset$, define a distribution $\statprob_\packval$ supported on
$\statdomain = \{-\radius, \radius\}^d$ via
\begin{align*}
\statprob_\packval(\statrv = \statsample) = (1 +
\delta \packval^\top \statsample / \radius) / 2^d.
\end{align*}
In words, for $\packval = e_j$, the coordinates of $\statrv$ are
independent uniform on $\{-\radius, \radius\}$ except for the
coordinate $j$, for which $\statrv_j = \radius$ with probability $1/2
+ \delta \packval_j$ and $\statrv_j = -\radius$ with probability $1/2
- \delta \packval_j$.  With this scheme, we have
$\optvar(\statprob_\packval) = \radius \delta \packval$, and since
$\linf{\delta \radius \packval - \delta \radius \altpackval} \ge
\delta \radius$, we have constructed a $\delta \radius$ packing in
$\ell_\infty$-norm.

\paragraphc{Upper bounding the mutual information} Let $\packrv$ be
drawn uniformly from the packing set $\packset = \{\pm
e_j\}_{j=1}^d$. With the sampling scheme in the previous paragraph, we
may provide the following upper bound on the mutual information
$\information(\channelrv_1, \ldots, \channelrv_n; \packrv)$ for any
non-interactive private distribution~\eqref{eqn:local-privacy-simple}:
\begin{lemma}
  \label{lemma:l1-information-bound}
For any non-interactive $\diffp$-differentially private distribution
$\channel$, we have
  \begin{equation*}
    \information(\channelrv_1, \ldots, \channelrv_n; \packrv) \le n
    \frac{e^\diffp}{4d} \left(e^{\diffp} - e^{-\diffp}\right)^2
    \delta^2.
  \end{equation*}
\end{lemma}
\noindent See Appendix~\ref{appendix:l1-information-bound} for a proof.

\paragraphc{Applying testing inequalities}

Finally, we turn to application of the testing
inequalities. Lemma~\ref{lemma:l1-information-bound}, in conjunction
with the standard testing reduction and Fano's
inequality~\eqref{eqn:fano}, implies that
\begin{equation*}
  \minimax_n(\theta(\mc{\statprob}), \linf{\cdot}, \diffp)
  \ge \frac{\radius \delta}{2} \left(1 - \frac{e^\diffp \delta^2 n
    (e^\diffp - e^{-\diffp})^2 / (4d)
    + \log 2}{\log(2d)}\right).
\end{equation*}
There is no loss of generality in assuming that $d \ge 2$, in which
case the choice
\begin{align*}
\delta^2 = \min \bigg \{1, \frac{d \log (2d)}{e^\diffp (e^\diffp -
  e^{-\diffp})^2 n} \bigg\}
\end{align*}
yields the proposition.

\subsection{Proof of Proposition~\ref{proposition:minimax-mean-high-dim}}
\label{sec:proof-mean-high-dim}

For this proposition, the construction of the packing and lower bound
used in the proof of Proposition~\ref{proposition:minimax-mean-linf}
also apply. Under these packing and sampling procedures, note that the
separation of points $\optvar(\statprob_\packval) = \radius \delta
\packval$ in $\ell_2$-norm is $\radius \delta$.  It thus remains to
provide the upper bound.  In this case, we use the sampling
strategy~\eqref{eqn:linf-sampling}, as in
Proposition~\ref{proposition:minimax-mean-linf} and
Section~\ref{sec:attainability-means}, noting that we may take the
bound $\sbound$ on $\linf{\channelrv}$ to be $\sbound = c \sqrt{d}
\radius / \diffp$ for a constant $c$.  Let $\optvar^*$ denote the true
mean, assumed to be $\nnzs$-sparse.  Now consider estimating
$\optvar^*$ by the $\ell_1$-regularized optimization problem
\begin{equation*}
  \what{\optvar} \defeq \argmin_{\optvar \in \R^d} \left\{ \frac{1}{2
    n} \ltwobigg{ \sum_{i = 1}^n (\channelrv_i - \optvar)}^2 + \lambda
  \lone{\optvar}\right\},
\end{equation*}
Defining the error vector $W = \optvar^* - \frac{1}{n} \sum_{i=1}^n
\channelrv_i$, we claim that
\begin{equation}
  \label{eqn:lone-solution}
  \lambda \ge 2 \linf{W}
  ~~~ \mbox{implies that} ~~~
  \ltwos{\what{\optvar} - \optvar}
  \le 3 \lambda \sqrt{\nnzs}.
\end{equation}
This result is a consequence of standard results on sparse estimation
(e.g., \citet[Theorem 1 and Corollary 1]{NegahbanRaWaYu12}).

Now we note if $W_i = \optvar^* - \channelrv_i$, then $W = \frac{1}{n}
\sum_{i=1}^n W_i$, and by construction of the sampling
mechanism~\eqref{eqn:linf-sampling} we have
$\linf{W_i} \le c \sqrt{d} \radius / \diffp$ for a constant $c$.
By Hoeffding's inequality and a union bound, we thus have for some
(different) universal constant $c$ that
\begin{equation*}
  \P(\linf{W} \ge t)
  \le 2d \exp\left(-c \frac{n \diffp^2 t^2}{\radius^2 d}\right)
  ~~ \mbox{for~}t \ge 0.
\end{equation*}
By taking $t^2 = \radius^2 d (\log(2d) + \epsilon^2) / (c n \diffp^2)$, we
find that $\linf{W}^2 \le \radius^2 d (\log(2d) + \epsilon^2) / (c n
\diffp^2)$ with probability at least $1 - \exp(-\epsilon^2)$, which gives the
  claimed minimax upper bound by appropriate choice of $\lambda = c \sqrt{d
    \log d / n \diffp^2}$ in inequality~\eqref{eqn:lone-solution}.

\subsection{Proof of inequality~\eqref{eqn:mean-laplace-sucks}}
\label{sec:mean-laplace-sucks}

We prove the bound by an argument using the private form of Fano's
inequality from Corollary~\ref{CorPrivateFano}.  The proof makes use
of the classical Varshamov-Gilbert bound (e.g.~\cite[Lemma 4]{Yu97}):

\begin{lemma}[Varshamov-Gilbert]
  There is a packing $\packset$ of the $d$-dimensional hypercube $\{-1,
  1\}^d$ of size $|\packset| \ge \exp(d/8)$ such that
  \begin{align*}
    \lone{\packval - \altpackval} \ge d/2 \quad \mbox{for all distinct
      pairs $\packval, \altpackval \in \packset$.}
  \end{align*}
\end{lemma}
Now, let $\delta \in [0, 1]$ and the distribution $\statprob_\packval$
be a point mass at $\delta \packval / \sqrt{d}$. Then
$\theta(\statprob_\packval) = \delta \packval / \sqrt{d}$ and
$\ltwo{\theta(\statprob_\packval) - \theta(\statprob_\altpackval)}^2
\ge \delta^2$. In addition, a calculation implies that if
$\marginprob_1$ and $\marginprob_2$ are $d$-dimensional
$\laplace(\kappa)$ distributions with means $\theta_1$ and $\theta_2$,
respectively, then
\begin{equation*}
  \dkl{\marginprob_1}{\marginprob_2}
  = \sum_{j = 1}^d \left(\exp(-\kappa |\theta_{1,j} - \theta_{2,j}|)
  + \kappa |\theta_{1,j} - \theta_{2,j}| - 1\right)
  \le \frac{\kappa^2}{2} \ltwo{\theta_1 - \theta_2}^2.
\end{equation*}
As a consequence, we have that under our Laplacian sampling scheme for the
$\channelrv$ and with $\packrv$ chosen uniformly from $\packset$,
\begin{equation*}
  \information(\channelrv_1, \ldots, \channelrv_n; \packrv)
  \le \frac{1}{|\packset|^2}
  n \sum_{\packval, \altpackval \in \packset} \dkl{\marginprob_\packval}{
    \marginprob_{\altpackval}}
  \le \frac{n \diffp^2}{2 d |\packset|^2}
  \sum_{\packval, \altpackval \in \packset}
  \ltwo{(\delta / \sqrt{d})(\packval - \altpackval)}^2
  \le \frac{2n \diffp^2 \delta^2}{d}.
\end{equation*}
Now, applying Fano's inequality~\eqref{eqn:fano} in the context of the testing
inequality~\eqref{eqn:estimation-to-testing}, we find that
\begin{equation*}
  \inf_{\what{\optvar}} \sup_{\packval \in \packset}
  \E_{\statprob_\packval}\left[\ltwos{\what{\optvar}(\channelrv_1,
      \ldots, \channelrv_n) - \optvar(\statprob_\packval)}^2\right]
  \ge \frac{\delta^2}{4}
  \left(1 - \frac{2 n \diffp^2 \delta^2 / d + \log 2}{d/8}\right).
\end{equation*}
We may assume (based on our one-dimensional results in
Proposition~\ref{proposition:location-family-bound}) w.l.o.g.\ that $d \ge
10$.  Taking $\delta^2 = d^2 / (48 n \diffp^2)$ then implies the
result~\eqref{eqn:mean-laplace-sucks}.

\section{Proofs of multinomial estimation results}
\label{sec:proof-multinomial-rate}

In this section, we prove the lower bounds
in Proposition~\ref{proposition:multinomial-rate}.
Before proving the bounds, however,
we outline our technique, which borrows from that in
Section~\ref{sec:proofs-big-mean-estimation}, and which we also
use to prove the lower bounds on density estimation.
The outline is as follows:
\begin{enumerate}[(1)]
\item As in step~\eqref{step:standard} of
  Section~\ref{sec:proofs-big-mean-estimation}, our first step is a standard
  reduction using the sharper version of Assouad's method
  (Lemma~\ref{lemma:sharp-assouad}) from estimation to a multiple binary
  hypothesis testing problem. Specifically, we perform a (essentially
  standard) reduction of the form~\eqref{eqn:risk-separation}.
\item Having constructed appropriately separated binary hypothesis tests, we
  use apply Theorem~\ref{theorem:sequential-interactive} via
  inequality~\eqref{eqn:sharp-assouad-kled} to control the testing error in
  the binary testing problem. Applying the theorem requires bounding certain
  suprema related to the covariance structure of randomly selected elements of
  $\packset = \{-1, 1\}^d$, as in the arguments in
  Section~\ref{sec:proofs-big-mean-estimation}. In this case, though, the
  symmetry of the binary hypothesis testing problems eliminates the need for
  carefully constructed packings of
  step~\ref{sec:proofs-big-mean-estimation}\eqref{item:construct-packing}.
\end{enumerate}

With this outline in mind, we turn to the proofs of
inequalities~\eqref{eqn:multinomial-limits}
and~\eqref{eqn:multinomial-rate-l1}.  As we proved the upper bounds in
Section~\ref{sec:private-multinomial-estimation}, this section focuses on the
argument for the lower bound. We provide the full proof for the mean-squared
Euclidean error, after which we show how the result for the $\ell_1$-error
follows.

Our first step is to provide a lower bound of the
form~\eqref{eqn:risk-separation}, giving a Hamming separation for the squared
error. To that end, fix $\delta \in [0, 1]$, and for simplicity, let us assume
that $d$ is even.  In this case, we set $\packset = \{-1, 1\}^{d/2}$, and for
$\packval \in \packset$ let $\statprob_\packval$ be the multinomial
distribution with parameter
\begin{equation*}
  \optvar_\packval \defeq \frac{1}{d} \onevec
  + \delta \frac{1}{d} \left[\begin{matrix} \packval \\ -\packval \end{matrix}
    \right] \in \simplex_d.
\end{equation*}
For any estimator $\what{\optvar}$, by defining $\what{\packval}_j =
\sign(\what{\optvar}_j - 1/d)$ for $j \in [d/2]$ we have the lower bound
\begin{equation*}
  \ltwos{\what{\optvar} - \optvar_\packval}^2
  \ge \frac{\delta^2}{d^2} \sum_{j=1}^{d/2} \indic{\what{\packval}_j \neq
    \packval_j},
\end{equation*}
so that by the sharper variant~\eqref{eqn:sharp-assouad-kled}
of Assouad's Lemma, we obtain
\begin{align}
  \max_{\packval \in \packset} \E_{\statprob_\packval}[\ltwos{\what{\optvar}
      - \optvar_\packval}^2]
  & \ge \frac{\delta^2}{4 d}
  \left[1 - \bigg(\frac{1}{2 d} \sum_{j=1}^{d/2}
    \dkl{\marginprob_{+j}^n}{\marginprob_{-j}^n}
    + \dkl{\marginprob_{-j}^n}{\marginprob_{+j}^n}\bigg)^\half\right].
  \label{eqn:assouad-multinomial}
\end{align}
Now we apply Theorem~\ref{theorem:sequential-interactive}, which requires
bounding sums of integrals $\int \optdens (d \statprob_{+j} -
d\statprob_{-j})$, where $\statprob_{+j}$ is defined in
expression~\eqref{eqn:paired-mixtures}. We claim the following inequality:
\begin{equation}
  \sup_{\optdens \in \linfset(\statdomain)}
  \sum_{j=1}^{d/2} \left(\int_{\statdomain} \optdens(\statval)
  d\statprob_{+j}(\statval) - d\statprob_{-j}(\statval)\right)^2
  \le \frac{8 \delta^2}{d}.
  \label{eqn:multinomial-sup}
\end{equation}
Indeed, by construction $\statprob_{+j}$ is
the multinomial with parameter
$(1/d) \onevec + (\delta / d) [e_j^\top ~ -e_j^\top]^\top \in \simplex_d$
and similarly for $\statprob_{-j}$, where $e_j \in \{0,1\}^{d/2}$ denotes the
$j$th standard basis vector. Abusing notation and identifying $\optdens$ with
vectors $\optdens \in [-1,1]^d$, we have
\begin{equation*}
  \int_{\statdomain} \optdens(\statval)
  d\statprob_{+j}(\statval) - d\statprob_{-j}(\statval)
  = \frac{2 \delta}{d} \optdens^\top \left[\begin{matrix} e_j \\ -e_j
      \end{matrix}\right],
\end{equation*}
whence we find
\begin{equation*}
  \sum_{j=1}^{d/2} \left(\int_{\statdomain} \optdens(\statval)
  d\statprob_{+j}(\statval) - d\statprob_{-j}(\statval)\right)^2
  = \frac{4 \delta^2}{d^2}
  \optdens^\top \sum_{j=1}^{d/2} \left[\begin{matrix}
      e_j \\ -e_j \end{matrix}\right]
  \left[\begin{matrix}
      e_j \\ -e_j \end{matrix}\right]^\top \optdens
  = \frac{4 \delta^2}{d^2} \optdens^\top \left[\begin{matrix}
      I & -I \\ -I & I \end{matrix}\right] \optdens
  \le \frac{8 \delta^2}{d},
\end{equation*}
because the operator norm of the matrix is bounded by 2. This gives the
claim~\eqref{eqn:multinomial-sup}.

Substituting the bound~\eqref{eqn:multinomial-sup} into the
bound~\eqref{eqn:assouad-multinomial} via
Theorem~\ref{theorem:sequential-interactive}, we obtain
\begin{equation*}
  \max_{\packval \in \packset} \E_{\statprob_\packval}[\ltwos{\what{\optvar}
      - \optvar_\packval}^2]
  \ge \frac{\delta^2}{4d}
  \left[1 - \left(4
    n (e^\diffp - 1)^2 \delta^2 / d^2\right)^\half \right].
\end{equation*}
Choosing $\delta^2 = \min\{1, d^2 / (16 n (e^\diffp  - 1)^2)\}$
gives the lower bound
\begin{equation*}
  \minimax_n(\simplex_d, \ltwo{\cdot}^2, \diffp)
  \ge \min\left\{\frac{1}{4d}, \frac{d}{64 n (e^\diffp - 1)^2}
  \right\}.
\end{equation*}
To complete the proof, we note that we can prove the preceding upper bound
for any even $d_0 \in \{2, \ldots, d\}$; this requires choosing
$\packval \in \packset = \{-1, 1\}^{d_0/2}$ and constructing the multinomial
vectors
\begin{equation*}
  \optvar_\packval = \frac{1}{d_0} \left[\begin{matrix}
      \onevec_{d_0} \\ 0_{d - d_0} \end{matrix}\right]
  + \frac{\delta}{d_0}
  \left[\begin{matrix} \packval \\ -\packval \\ 0_{d - d_0} \end{matrix}
    \right]
  \in \simplex_d,
  ~~~ \mbox{where} ~~
  \onevec_{d_0} = [1 ~ 1 ~ \cdots ~ 1]^\top \in \R^{d_0}.
\end{equation*}
Repeating the proof \emph{mutatis mutandis} gives the bound
\begin{equation*}
  \minimax_n(\simplex_d, \ltwo{\cdot}^2, \diffp)
  \ge \max_{d_0 \in \{2, 4, \ldots,
    2 \floor{d/2}\}}
  \min\left\{\frac{1}{4d_0}, \frac{d_0}{64 n (e^\diffp - 1)^2}
  \right\}.
\end{equation*}
Choosing $d_0$ to be the even integer closest to $\sqrt{n \diffp^2}$ in $\{1,
\ldots, d\}$ and noting that $(e^\diffp - 1)^2 \le 3 \diffp^2$ for
$\diffp \in [0, 1]$
gives the claimed result~\eqref{eqn:multinomial-limits}.

In the case of measuring error in the $\ell_1$-norm, we provide a
completely identical proof, except that we have the separation
$\lones{\what{\optvar} - \optvar_\packval} \ge (\delta / d)
\sum_{j=1}^{d/2} \indic{\what{\packval}_j \neq \packval_j}$, and
thus inequality~\eqref{eqn:assouad-multinomial} holds with
the initial multiplier $\delta^2 / (4d)$ replaced by $\delta / (4d)$.
Parallel reasoning to the $\ell_2^2$ case then gives the minimax lower bound
\begin{equation*}
  \minimax_n(\simplex_d, \lone{\cdot}, \diffp)
  \ge \frac{\delta}{4d_0} \left[1 - (4 n (e^\diffp - 1)^2 \delta^2
    / d_0^2)^\half\right]
\end{equation*}
for any even $d_0 \in \{2, \ldots, d\}$. Choosing
$\delta = \min\{1, d_0^2 / (16 n (e^\diffp - 1)^2)\}$ gives
the claim~\eqref{eqn:multinomial-rate-l1}.

\section{Proofs of density estimation results}

In this section, we provide the proofs of the results stated in
Section~\ref{sec:density-estimation} on density estimation. We defer the
proofs of more technical results to later appendices.  Throughout all proofs,
we use $c$ to denote a universal constant whose value may change from line to
line.

\subsection{Proof of Proposition~\ref{proposition:density-estimation}}
\label{sec:proof-density-estimation}

As with our proof for multinomial estimation, the argument follows the general
outline described at the beginning of
Section~\ref{sec:proof-multinomial-rate}.  We remark that our proof is based
on an explicit construction of densities identified with corners of the
hypercube, a more classical approach than the global metric entropy approach
of~\citet{YangBa99} (cf.~\cite{Yu97}).  We use the local packing approach
since it is better suited to the privacy constraints and information
contractions that we have developed.  In comparison with our proofs of
previous propositions, the construction of a suitable packing of $\densclass$
is somewhat more challenging: the identification of densities with
finite-dimensional vectors, which we require for our application of
Theorem~\ref{theorem:sequential-interactive}, is not immediately obvious.  In
all cases, we guarantee that our density functions $f$ belong to the
trigonometric Sobolev space, so we may work directly with smooth density
functions $f$.

\begin{figure}
  \begin{center}
    \begin{tabular}{cc}
      \psfrag{g}{$g_1$}
      \includegraphics[height=.32\columnwidth]{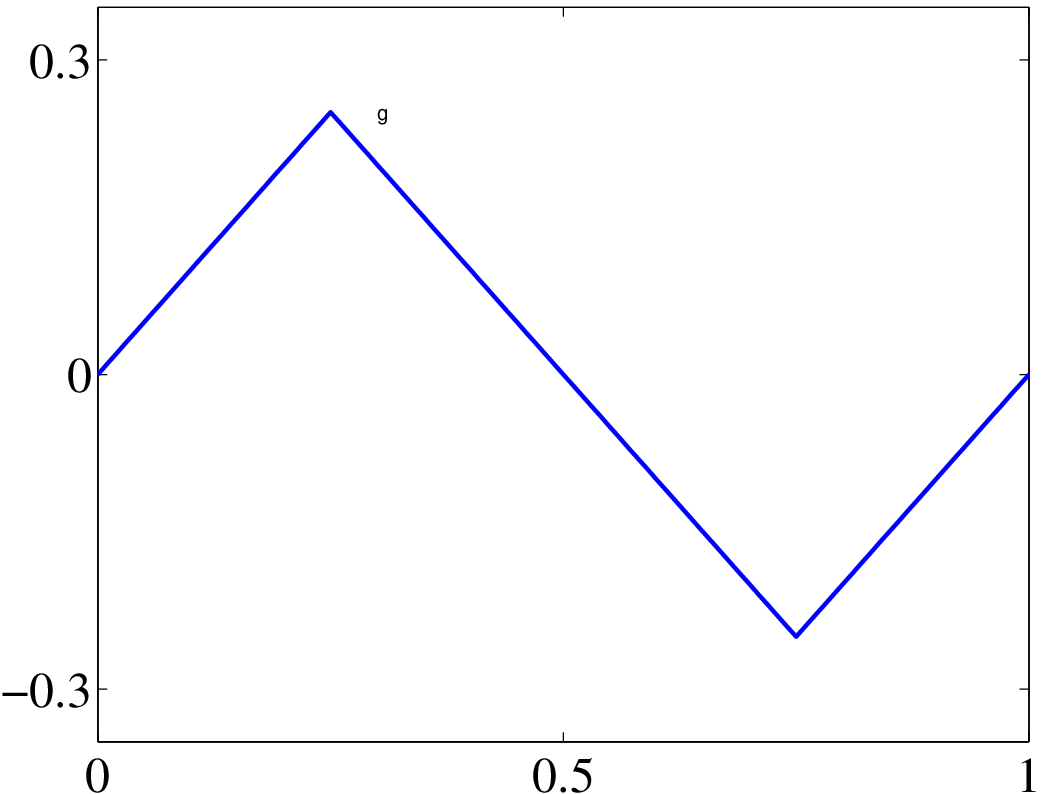} &
      \psfrag{g}{$g_2$}
      \includegraphics[height=.32\columnwidth]{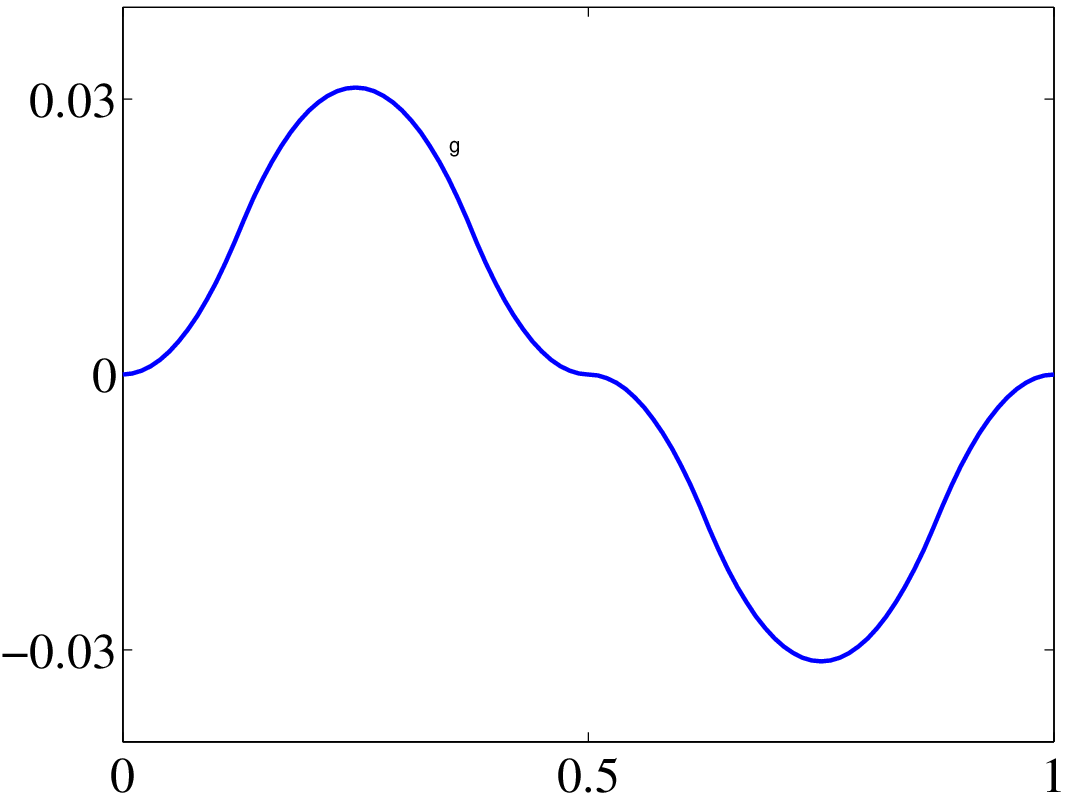} \\
        (a) & (b)
    \end{tabular}
    \caption{\label{fig:bump-func} Panel (a): illustration of
      $1$-Lipschitz continuous bump function $g_1$ used to pack
      $\densclass$ when $\numderiv = 1$.  Panel (b): bump function
      $g_2$ with $|g_2''(\statsample)| \le 1$ used to pack
      $\densclass$ when $\numderiv = 2$.
    }
  \end{center}
\end{figure}

\paragraphc{Constructing well-separated densities} 

We begin by describing a standard framework for defining local
packings of density functions.  Let $g_\numderiv: [0,1] \rightarrow
\real$ be a function satisfying the following properties:

\begin{enumerate}[(a)]
\item The function $g_\numderiv$ is $\numderiv$-times differentiable
  with
  \begin{equation*}
    0 = g_\numderiv^{(i)}(0) = g_\numderiv^{(i)}(1/2)
    = g_\numderiv^{(i)}(1)
    ~~ \mbox{for~all~} i < \numderiv.
  \end{equation*}
\item The function $g_\numderiv$ is centered with $\int_0^1
  g_\numderiv(x) dx = 0$, and there exist constants $c, c_{1/2} > 0$
  such that
\begin{equation*}
  \int_0^{1/2} g_\numderiv(x) dx = -\int_{1/2}^1 g_\numderiv(x) dx =
  c_{1/2} ~~~ \mbox{and} ~~~ \int_0^1
  \left(g^{(i)}_\numderiv(x)\right)^2 dx \ge c ~~ \mbox{for~all~} i <
  \numderiv.
\end{equation*}
\item The function $g_\numderiv$ is non-negative on $[0, 1/2]$ and
  non-positive on $[1/2, 1]$, and Lebesgue measure is absolutely
  continuous with respect to the measures $G_j, j = 1, 2$, given by
  \begin{equation}
    \label{eqn:absolute-continuity-bumps}
    G_1(A) = \int_{A \cap [0, 1/2]} g_\numderiv(x) dx ~~~ \mbox{and} ~~~
    G_2(A) = -\int_{A \cap [1/2, 1]} g_\numderiv(x) dx.
  \end{equation}
\item Lastly, for almost every $x \in [0, 1]$, we have
  $|g^{(\numderiv)}_\numderiv(x)| \le 1$ and $|g_\numderiv(x)| \le 1$.
\end{enumerate}
\noindent As illustrated in Figure~\ref{fig:bump-func}, the functions
$g_\numderiv$ are smooth ``bump'' functions.

Fix a positive integer $\numbin$ (to be specified in the sequel).  Our first
step is to construct a family of ``well-separated'' densities for which we can
reduce the density estimation problem to one of identifying corners of a
hypercube, which allows application of
Lemma~\ref{lemma:sharp-assouad}. Specifically, we must exhibit a condition
similar to the separation condition~\eqref{eqn:risk-separation}. For each $j
\in \{1, \ldots, \numbin$ define the function
\begin{align*}
  g_{\numderiv, j}(x) & \defeq \frac{1}{\numbin^\numderiv} \, g_\beta
  \left(\numbin \Big( x - \frac{j - 1}{\numbin}\Big)\right) \indic{ x
    \in \left[\textstyle{\frac{j-1}{\numbin}, \frac{j}{\numbin}}\right]}.
\end{align*}
Based on this definition, we define the family of
densities
\begin{equation}
  \label{eqn:density-packer}
  \bigg\{ f_\packval \defeq 1 + \sum_{j=1}^\numbin \packval_j
  g_{\numderiv, j} ~~~ \mbox{for}~ \packval \in \packset \bigg\}
  \; \subseteq \densclass.
\end{equation}
It is a standard fact~\cite{Yu97,Tsybakov09} that for any $\packval
\in \packset$, the function $f_\packval$ is $\numderiv$-times
differentiable, satisfies $|f^{(\numderiv)}(x)| \le 1$ for all $x$.
Now, based on some density $f \in \densclass$, let us define the
sign vector $\maptocube(f) \in \{-1, 1\}^\numbin$ to have entries
\begin{equation*}
  \maptocube_j(f)
  \defeq \argmin_{s \in \{-1, 1\}}
  \int_{[\frac{j-1}{\numbin}, \frac{j}{\numbin}]}
  \left(f(\statval) - s g_{\numderiv,j}(\statval)\right)^2 d\statval.
\end{equation*}
Then by construction of the $g_\numderiv$ and $\maptocube$, we have for a
numerical constant $c$ (whose value may depend on $\numderiv$) that
\begin{equation*}
  \ltwo{f - f_\packval}^2
  \ge c \sum_{j = 1}^\numbin \indic{\maptocube_j(f) \neq \packval_j}
  \int_{[\frac{j-1}{\numbin},
      \frac{j}{\numbin}]} (g_{\numderiv,j}(x))^2 dx
  = \frac{c}{\numbin^{2 \numderiv + 1}}
  \sum_{j = 1}^\numbin \indic{\maptocube_j(f) \neq \packval_j}.
\end{equation*}
By inspection, this is the Hamming separation required in
inequality~\eqref{eqn:risk-separation},
whence the sharper version~\eqref{eqn:sharp-assouad-kled}
of Assouad's Lemma~\ref{lemma:sharp-assouad} gives the result
\begin{equation}
  \label{eqn:sharp-assouad-density}
  \minimax_n \left(\densclass[1], \ltwo{\cdot}^2, \diffp \right)
  \ge \frac{c}{\numbin^{2 \numderiv}}
  \left[1 - \bigg(\frac{1}{4 \numbin}
    \sum_{j = 1}^\numbin (\dkl{\marginprob_{+j}^n}{\marginprob_{-j}^n}
    + \dkl{\marginprob_{-j}^n}{\marginprob_{+j}^n})\bigg)^\half\right],
\end{equation}
where we have defined $\statprob_{\pm j}$ to be the probability distribution
associated with the averaged densities $f_{\pm j} = 2^{1-\numbin}
\sum_{\packval : \packval_j = \pm 1} f_\packval$.


\paragraphc{Applying divergence inequalities}

Now we must control the summed KL-divergences. To do so,
we note that by the construction~\eqref{eqn:density-packer},
symmetry implies that
\begin{equation}
  f_{+j} = 1 + g_{\numderiv,j} ~~~ \mbox{and} ~~~
  f_{-j} = 1 - g_{\numderiv,j}
  ~~~ \mbox{for each ~} j \in [\numbin].
  \label{eqn:f-plus-j}
\end{equation}
We then obtain the following result, which bounds the averaged
KL-divergences.
\begin{lemma}
  \label{lemma:density-kls}
  For any $\diffp$-locally private conditional distribution
  $\channel$, the summed KL-divergences are bounded as
  \begin{equation*}
  \sum_{j=1}^\numbin \left(\dkl{\marginprob_{+j}^n}{\marginprob_{-j}^n}
  + \dkl{\marginprob_{+j}^n}{\marginprob_{-j}^n}\right)
  \le
  4 c_{1/2}^2 \, n
  \frac{(e^\diffp - 1)^2}{\numbin^{2 \numderiv + 1}}.
  \end{equation*}
\end{lemma}
\noindent
The proof of this lemma is fairly involved, so we defer it to
Appendix~\ref{appendix:density-kls}.  We note that, for $\diffp \le 1$, we
have $(e^\diffp - 1)^2 \le 3 \diffp^2$, so we may replace the bound in
Lemma~\ref{lemma:density-kls} with the quantity $c n \diffp^2 / \numbin^{2
  \numderiv + 1}$ for a constant $c$.  We remark that standard
divergence bounds using Assouad's lemma~\cite{Yu97,Tsybakov09} provide a bound
of roughly $n / \numbin^{2 \numderiv}$; our bound is thus essentially a factor
of the ``dimension'' $\numbin$ tighter.

The remainder of the proof is an application of
inequality~\eqref{eqn:sharp-assouad-density}. In particular,
by applying Lemma~\ref{lemma:density-kls}, we find that for
any $\diffp$-locally private channel
$\channelprob$, there are constants $c_0, c_1$ (whose values
may depend on $\numderiv$) such that
\begin{equation*}
  \minimax_n\left(\densclass, \ltwo{\cdot}^2, \channelprob\right) \ge
  \frac{c_0}{\numbin^{2 \numderiv}}
  \left[
    1 - \left(\frac{c_1 n \diffp^2}{\numbin^{2 \numderiv + 2}}\right)^\half
  \right].
\end{equation*}
Choosing $\numbin_{n, \diffp, \numderiv} = \left(4 c_1 n \diffp^2
\right )^{\frac{1}{2 \numderiv + 2}}$ ensures that the quantity inside
the parentheses is at least $1/2$. Substituting for
$\numbin$ in the preceding display proves the proposition.

\subsection{Proof of Proposition~\ref{proposition:histogram-estimator}}
\label{sec:proof-histogram}

Note that the operator $\Pi_\numbin$ performs a Euclidean projection
of the vector $(\numbin/n) \sum_{i=1}^n \channelrv_i$ onto the scaled
probability simplex, thus projecting $\what{f}$ onto the set of
probability densities.  Given the non-expansivity of Euclidean
projection, this operation can only decrease the error
$\ltwos{\what{f} - f}^2$. Consequently, it suffices to bound the error
of the unprojected estimator; to reduce notational overhead
we retain our previous notation of $\what{\optvar}$ for the unprojected
version.  Using this notation, we have
\begin{align*}
  \E \left [\ltwobig{\what{f} - f}^2\right] & \leq \sum_{j=1}^\numbin
  \E_f\left [\int_{\frac{j-1}{\numbin}}^{\frac{j}{\numbin}}
    (f(\statsample) - \what{\theta}_j)^2 d \statsample \right].
\end{align*}
By expanding this expression and noting that the independent noise
variables $W_{ij} \sim \laplace(\diffp/2)$ have zero mean, we obtain
\begin{align}
  \E \left[\ltwobig{\what{f} - f}^2\right] & \le \sum_{j=1}^\numbin \E_f
  \left[\int_{\frac{j-1}{\numbin}}^{\frac{j}{\numbin}} \bigg
    (f(\statsample) - \frac{\numbin}{n} \sum_{i=1}^n
    [\histelement_\numbin(\statrv_i)]_j\bigg)^2 d\statsample\right] +
  \sum_{j=1}^\numbin \int_{\frac{j-1}{\numbin}}^{\frac{j}{\numbin}}
  \E\bigg[\bigg(\frac{\numbin}{n} \sum_{i = 1}^n W_{ij}\bigg)^2 \bigg]
  \nonumber \\ 
  \label{eqn:private-histogram-risk}
  & = \sum_{j=1}^\numbin \int_{\frac{j-1}{\numbin}}^{\frac{j}{\numbin}}
  \E_f\left[ \bigg(f(\statsample) - \frac{\numbin}{n} \sum_{i=1}^n
    [\histelement_\numbin(\statrv_i)]_j\bigg)^2\right] d\statsample +
  \numbin \, \frac{1}{\numbin} \, \frac{4\numbin^2}{n \diffp^2}.
\end{align}

Next we bound the error term inside the
expectation~\eqref{eqn:private-histogram-risk}.  Defining $p_j \defeq
\P_f(\statrv \in \statdomain_j) = \int_{\statdomain_j} f(\statsample)
d\statsample$, we have
\begin{equation*}
\numbin\E_f\left[[\histelement_\numbin(\statrv)]_j\right] = \numbin
p_j = \numbin \int_{\statdomain_j} f(\statsample) d\statsample \in
\left[f\left(\statsample\right) - \frac{1}{\numbin},
  f\left(\statsample\right) + \frac{1}{\numbin}\right] ~~
\mbox{for~any~} \statsample \in \statdomain_j,
\end{equation*}
by the Lipschitz continuity of $f$.
Thus, expanding the bias and variance of the integrated expectation above,
we find that
\begin{align*}
  \E_f\left[
    \bigg(f(\statsample) - \frac{\numbin}{n} \sum_{i=1}^n
         [\histelement_\numbin(\statrv_i)]_j\bigg)^2\right]
  & \le \frac{1}{\numbin^2} + \var\left(\frac{\numbin}{n} \sum_{i=1}^n
  \left[\histelement_\numbin(\statrv_i)\right]_j\right) \\
  & = \frac{1}{\numbin^2} + \frac{\numbin^2}{n}
  \var([\histelement_\numbin(\statrv)]_j)
  = \frac{1}{\numbin^2} + \frac{\numbin^2}{n} p_j(1 - p_j).
\end{align*}
Recalling the inequality~\eqref{eqn:private-histogram-risk}, we obtain
\begin{equation*}
  \E_f\left[\ltwobig{\what{f} - f}^2\right]
  \le \sum_{j=1}^\numbin \int_{\frac{j-1}{\numbin}}^\frac{j}{\numbin}
  \left(\frac{1}{\numbin^2} + \frac{\numbin^2}{n} p_j(1 - p_j)\right)
  d \statsample
  + \frac{4\numbin^2}{n \diffp^2}
  = \frac{1}{\numbin^2} + \frac{4\numbin^2}{n\diffp^2}
  + \frac{\numbin}{n} \sum_{j=1}^\numbin p_j(1 - p_j).
\end{equation*}
Since $\sum_{j=1}^\numbin p_j = 1$, we find that
\begin{equation*}
  \E_f\left[\ltwobig{\what{f} - f}^2\right] \le \frac{1}{\numbin^2} +
  \frac{4 \numbin^2}{n \diffp^2} + \frac{\numbin}{n},
\end{equation*}
and choosing $\numbin = (n \diffp^2)^{\frac{1}{4}}$ yields the claim.


\subsection{Proof of Proposition~\ref{proposition:density-upper-bound}}
\label{sec:proof-density-upper-bound}

We begin by fixing $\numbin \in \N$; we will optimize the choice of $\numbin$
shortly.  Recall that, since $f \in \densclass[\lipconst]$, we have $f
= \sum_{j=1}^\infty \optvar_j \basisfunc_j$ for $\optvar_j = \int f
\basisfunc_j$. Thus we may define $\overline{\channelrv}_j = \frac{1}{n}
\sum_{i=1}^n \channelrv_{i,j}$ for each $j \in \{1, \ldots, \numbin\}$, and we have
\begin{equation*}
  \ltwos{\what{f} - f}^2
  = \sum_{j=1}^\numbin (\optvar_j - \overline{\channelrv}_j)^2
  + \sum_{j = \numbin + 1}^\infty \optvar_j^2.
\end{equation*}
Since $f \in \densclass[\lipconst]$, we are guaranteed that
$\sum_{j=1}^\infty j^{2 \numderiv} \optvar_j^2 \leq \lipconst^2$, and
hence
\begin{equation*}
\sum_{j > \numbin} \optvar_j^2 = \sum_{j > \numbin} j^{2 \numderiv}
\frac{\optvar_j^2}{j^{2 \numderiv}} \le \frac{1}{\numbin^{2
    \numderiv}} \sum_{j > \numbin} j^{2 \numderiv} \optvar_j^2 \le
\frac{1}{\numbin^{2 \numderiv}} \lipconst^2.
\end{equation*}
For the indices $j \le \numbin$, we note that by assumption,
$\E[\channelrv_{i,j}] = \int \basisfunc_j f = \optvar_j$,
and since $|\channelrv_{i,j}| \le \sbound$, we have 
\begin{equation*}
  \E\left[(\optvar_j - \overline{Z}_j)^2\right] = \frac{1}{n}
  \var(Z_{1,j}) \le \frac{\sbound^2}{n} =
  \frac{\orthbound^2}{c_\numbin} \, \frac{\numbin}{n} \,
  \left(\frac{e^\diffp + 1}{e^\diffp - 1}\right)^2,
\end{equation*}
where $c_\numbin = \Omega(1)$ is the constant in
expression~\eqref{eqn:size-infinity-channel}.  Putting together the pieces,
the mean-squared $L^2$-error is upper bounded as
\begin{equation*}
\E_f\left[\ltwos{\what{f} - f}^2\right] \le c \left(\frac{\numbin^2}{n
  \diffp^2} + \frac{1}{\numbin^{2 \numderiv}}\right),
\end{equation*}
where $c$ is a constant depending on $\orthbound$, $c_\numbin$, and
$\lipconst$.  Choose $\numbin = (n \diffp^2)^{1 / (2 \numderiv +
  2)}$ to complete the proof.

\subsection{Insufficiency of Laplace noise for density estimation}
\label{sec:density-laplace-sucks}

Finally, we consider the insufficiency of standard Laplace noise addition for
estimation in the setting of this section. Consider the vector
\mbox{$[\basisfunc_j(\statrv_i)]_{j=1}^\numbin \in [-\orthbound,
    \orthbound]^\numbin$.} To make this vector $\diffp$-differentially private
by adding an independent Laplace noise vector $W \in \R^\numbin$, we must take
$W_j \sim \laplace(\diffp / (\orthbound k))$. The natural orthogonal series
estimator~\cite[e.g.,][]{WassermanZh10} is to take $\channelrv_i =
[\basisfunc_j(\statrv_i)]_{j=1}^\numbin + W_i$, where $W_i \in \R^\numbin$ are
independent Laplace noise vectors. We then use the density
estimator~\eqref{eqn:orthogonal-density-estimator}, except that we use the
Laplacian perturbed $\channelrv_i$. However, this estimator suffers the
following drawback:
\begin{observation}
  \label{observation:laplace-density-bad}
  Let $\what{f} = \frac{1}{\numobs} \sum_{i=1}^\numobs \sum_{j=1}^\numbin
  \channelrv_{i, j} \basisfunc_j$, where the $\channelrv_i$ are the
  Laplace-perturbed vectors of the previous paragraph. Assume the orthonormal
  basis $\{\basisfunc_j\}$ of $L^2([0, 1])$ contains the constant function.
  There is a constant $c$ such that for any $\numbin \in \N$,
  there is an $f \in \densclass[2]$ such that
  \begin{equation*}
    \E_f\left[\ltwos{f - \what{f}}^2\right] \ge 
    c (n \diffp^2)^{-\frac{2 \numderiv}{2 \numderiv + 3}}.
  \end{equation*}
\end{observation}
\begin{proof}
  We begin by noting that for $f = \sum_j \optvar_j \basisfunc_j$,
  by definition of $\what{f} = \sum_j \what{\optvar}_j \basisfunc_j$ we have
  \begin{equation*}
    \E\left[\ltwos{f - \what{f}}^2\right]
    = \sum_{j=1}^\numbin \E\left[(\optvar_j - \what{\optvar}_j)^2\right]
    + \sum_{j \ge \numbin + 1} \optvar_j^2
    = \sum_{j=1}^\numbin \frac{\orthbound^2 \numbin^2}{n \diffp^2}
    + \sum_{j \ge \numbin + 1} \optvar_j^2
    = \frac{\orthbound^2 \numbin^3}{n \diffp^2}
    + \sum_{j \ge \numbin + 1} \optvar_j^2.
  \end{equation*}
  Without loss of generality, let us assume $\basisfunc_1 = 1$ is the constant
  function. Then $\int \basisfunc_j = 0$ for all $j > 1$, and by defining the
  true function $f = \basisfunc_1 + (\numbin + 1)^{-\numderiv}
  \basisfunc_{\numbin + 1}$, we have $f \in \densclass[2]$ and
  $\int f = 1$, and moreover,
  \begin{equation*}
    \E\left[\ltwos{f - \what{f}}^2\right]
    \ge \frac{\orthbound^2 \numbin^3}{n \diffp^2}
    + \frac{1}{(\numbin + 1)^{-2 \numderiv}}
    \ge
    C_{\numderiv, \orthbound}
    (n \diffp^2)^{-\frac{2 \numderiv}{2 \numderiv + 3}},
  \end{equation*}
  where $C_{\numderiv, \orthbound}$ is a constant depending on
  $\numderiv$ and $\orthbound$.  This final lower bound comes by
  minimizing over all $\numbin$.  (If $(\numbin + 1)^{-\numderiv}
  \orthbound > 1$, we can rescale $\basisfunc_{\numbin + 1}$ by
  $\orthbound$ to achieve the same result and guarantee that $f \ge 0$.)
\end{proof}

This lower bound shows that standard estimators
based on adding Laplace noise to appropriate basis expansions of the
data fail: there is a degradation in rate from
$n^{-\frac{2 \numderiv}{2 \numderiv + 2}}$ to $n^{-\frac{2
    \numderiv}{2 \numderiv + 3}}$. While this is not a formal proof
that no approach based on Laplace perturbation can provide optimal
convergence rates in our setting, it does suggest that finding such an
estimator is non-trivial.

\section{Packing set constructions}
\label{sec:packing-sets}

In this appendix, we collect proofs of the constructions of our
packing sets.

\simplefixed{
  \subsection{Proof of Lemma~\ref{lemma:sphere-packing}}
  \label{AppLemSpherePack}
  By the Varshamov-Gilbert bound~\cite[e.g.,][Lemma 4]{Yu97}, there is a
  packing $\mc{H}_d$ of the $d$-dimensional hypercube $\{-1, 1\}^d$ of size
  $|\mc{H}_d| \ge \exp(d/8)$ satisfying $\lone{u - v} \ge
  d/2$ for all $u, v \in \mc{H}_d$ with $u \neq
  v$. For each $u \in \mc{H}_d$, set $\packval_u = u / \sqrt{d}$, so that
  $\ltwo{\packval_u} = 1$ and $\ltwo{\packval_u - \packval_v}^2
  \ge d / d = 1$ for $u \neq v \in \mc{H}_d$. Setting
  $\packset = \{\packval_u \mid u \in \mc{H}_d\}$ gives the desired result.
}{
\subsection{Proof of Lemma~\ref{lemma:sphere-packing}}
\label{AppLemSpherePack}

The first statement of the lemma follows from an application of the
probabilistic method. Let $\packval^i \in \R^d$ be sampled independently and
uniformly from the $\ell_2$-sphere of radius
$1$, $i = 1, \ldots, N$.  Consider the event $\Event$ that there exists a
collection of vectors $\{\packval^1, \ldots, \packval^N\}$ such that $(1/N)
\sum_{i=1}^N \packval^i (\packval^i)^\top \preceq ((1 + \delta) / d) I_{d
  \times d}$.  Following an argument of~\citet{DuchiJoWaWi12}, the event
$\Event$ holds whenever
\begin{equation*}
  \binom{N}{2} \exp\left(-\frac{49 d}{128}\right)
  + 2 \exp\left(-\frac{N \delta^2}{16}\right) < 1.
\end{equation*}
(See equation (23) in the paper~\cite{DuchiJoWaWi12}.)
For $d > 16$, choosing $\delta = 1$ and $N = \ceil{\exp(49 d /
  256)}$ yields the desired inequality.
For $d \le 16$, this inequality fails, but a simpler argument gives
the result. The choice of $\packset = \{u / \ltwo{u} : u \in \{-1,
1\}^d\}$ yields $|\packset| = \exp(d \log 2)$, and by inspection
\begin{equation*}
  \frac{1}{|\packset|}
  \sum_{\packval \in \packset} \packval \packval^\top =
  (1/d) I_{d \times d},
  ~~~ \mbox{and} ~~~
  \ltwo{\packval - \altpackval}
  = \frac{1}{\sqrt{d}} \ltwo{u - u'}
  \ge \frac{2}{\sqrt{16}} = \half
\end{equation*}
for $u \neq u' \in \{-1, 1\}^d$.  Combining the pieces yields the
claim.
} 


\subsection{Proof of Lemma~\ref{lemma:hypercube-packing}}
\label{sec:proof-hypercube-packing}

We use the probabilistic method~\cite{AlonSp00}, showing that for random
draws from the Boolean hypercube, a collection of vectors as
claimed in the lemma exists with positive probability.  Consider a
set of $N$ vectors $\packval^i \in \{-1, 1\}^k$ sampled
uniformly at random from the Boolean hypercube, and for a fixed $t >
0$, define the two ``bad'' events
\begin{align*}
  \BadEvent_1 \defeq \big \{ \exists ~ i \neq j \mid \lone{\packval^i -
    \packval^j} < k/2 \big\}, \quad \mbox{and} \quad \BadEvent_2(t) \defeq
  \bigg \{ \frac{1}{N} \sum_{i=1}^N \packval^i (\packval^i)^\top \not\preceq
  (t +1) I_{k \times k} \bigg\}.
\end{align*}
We begin by analyzing $\BadEvent_1$. Letting $\{W_\ell\}_{\ell=1}^k$
denote a sequence of i.i.d.\ Bernoulli $\{0,1\}$ variables, for any $i
\neq j$, the event $\{ \lones{\packval^i - \packval^j} < k/2 \}$ is
equivalent to the event $\{ \sum_{\ell=1}^k W_\ell < k/4 \}$.
Consequently, by combining the union bound with the the Hoeffding
bound, we find
\begin{align}
  \label{EqnBadOne}
  \P(\BadEvent_1) & \leq \binom{N}{2} \P \big( \lone{\packval_i -
    \packval_j} < k/2 \big) \; \leq \; \binom{N}{2} \; \exp(-k/8).
\end{align}
Turning to the event $\BadEvent_2(t)$, we have $\frac{1}{N} \sum_{i=1}^N
\packval_i (\packval^i)^\top \not\preceq (t+1) I_{k \times k}$ if and only
if the maximum eigenvalue $\lambda_{\max}(\frac{1}{N}
\sum_{i=1}^N \packval^i (\packval^i)^\top - I_{k \times k})$ is
larger than $t$.  Using sharp versions of the Ahlswede-Winter
inequalities~\cite{AhlswedeWi02} (see Corollary 4.2 in the
paper~\cite{MackeyJoChFaTr12}), we obtain
\begin{align}
  \label{EqnBadTwo}
  \P(\BadEvent_2(t)) & \leq k \exp\left(-\frac{N t^2}{k^2}\right).
\end{align}
Finally, combining the union bound with inequalities~\eqref{EqnBadOne}
and~\eqref{EqnBadTwo}, we find that
\begin{align*}
  \mprob( \BadEvent_1 \cup \BadEvent_2(t) \big) & \leq \frac{N (N - 1)}{2}
  \exp(-k/8) + d \exp\left(-\frac{N t^2}{k^2} \right).
\end{align*}
By inspection, if we choose $t = 24$ and $N = \ceil{\exp(k/16)}$, the
above bound is strictly less than 1, so a packing satisfying the
constraints must exist.

\section{Information bounds}
\label{appendix:lemmas-mutual-information}

In this appendix, we collect the proofs of lemmas providing mutual
information and KL-divergence bounds.


\subsection{Proof of Lemma~\ref{lemma:linf-info-bound}}
\label {sec:linf-info-bound}

Our strategy is to apply Theorem~\ref{theorem:super-master} to bound the
mutual information.  Without loss of generality, we may assume that $\radius =
1$ so the set $\statdomain = \{\pm e_j\}_{j=1}^k$, where $e_j \in \R^d$. Thus,
under the notation of Theorem~\ref{theorem:super-master}, we may identify
vectors $\optdens \in L^\infty(\statdomain)$ by vectors $\optdens \in
\R^{2k}$. If we define $\overline{\packval} = \frac{1}{|\packset|}
\sum_{\packval \in \packset} \packval$ to be the mean element of the packing
set, the linear functional $\varphi_\packval$ defined in
Theorem~\ref{theorem:super-master} is
\begin{align*}
  \varphi_\packval(\optdens)
  & = \frac{1}{2k} \bigg[\sum_{j=1}^k
    \optdens(e_j)
    \frac{1 + \packval_j \delta}{2}
    + \sum_{j=1}^k \optdens(-e_j) \frac{1 - \packval_j \delta}{2}
    \bigg]
  - \frac{1}{2k}
  \bigg[\sum_{j=1}^k
    \optdens(e_j)
    \frac{1 + \overline{\packval}_j \delta}{2}
    + \sum_{j=1}^k \optdens(-e_j) \frac{1 - \overline{\packval}_j \delta}{2}
    \bigg] \\
  & =
  \frac{1}{2k} \sum_{j=1}^k
  \left[\frac{\delta}{2} \optdens(e_j)(\packval_j - \overline{\packval}_j)
    - \frac{\delta}{2} \optdens(-e_j)(\packval_j - \overline{\packval}_j)
    \right]
  = \frac{\delta}{4k}
  \optdens^\top \left[\begin{matrix} I_{k \times k} & 0_{k \times d-k}
      \\ -I_{k \times k} & 0_{k \times d-k} \end{matrix} \right]
  (\packval - \overline{\packval}).
\end{align*}
Define the matrix
\begin{equation*}
  A \defeq \left[\begin{matrix} I_{k \times k} & 0_{k \times d-k}
      \\ -I_{k \times k} & 0_{k \times d-k} \end{matrix} \right]
  \in \{-1, 0, 1\}^{2k \times d}.
\end{equation*}
Then we have that
\begin{align}
  \frac{1}{|\packset|} \sum_{\packval \in \packset}
  \varphi_\packval(\optdens)^2
  & = \frac{\delta^2}{(4k)^2} \optdens^\top A
  \frac{1}{|\packset|}
  \sum_{\packval \in \packset} (\packval - \overline{\packval})
  (\packval -\overline{\packval})^\top
  A^\top \optdens \nonumber \\
  & = \frac{\delta^2}{(4 k)^2}
  \optdens^\top A
  \left(\frac{1}{|\packset|}
  \sum_{\packval \in \packset} \packval \packval^\top
  - \overline{\packval}\overline{\packval}^\top
  \right)
  A^\top \optdens \nonumber
  \le \frac{\delta^2}{(4k)^2}
  \optdens^\top A
  \left(\frac{1}{|\packset|}
  \sum_{\packval \in \packset} \packval \packval^\top
  \right)
  A^\top \optdens \nonumber \\
  & \le \frac{25}{16} \, \frac{\delta^2}{k^2}
  \optdens^\top A
  I_{2d \times 2d} A^\top \optdens
  = \left(\frac{5 \delta}{4k}\right)^2
  \optdens^\top\left[\begin{matrix} I_{k \times k} & -I_{k \times k}
      \\ -I_{k \times k} & I_{k \times k} \end{matrix} \right]
  \optdens.
  \label{eqn:spectral-linf-bound}
\end{align}
Here the final inequality used our assumption on the sum of
outer products in $\packset$.

We complete our proof using the bound~\eqref{eqn:spectral-linf-bound}. The
operator norm of the matrix specified in~\eqref{eqn:spectral-linf-bound} is
2.  As a consequence, since we have the containment
\begin{equation*}
  \linfset = \left\{
  \optdens \in \R^{2k} : \linf{\optdens} \le 1 \right\}
  \subset \left\{\optdens \in \R^{2k}
  : \ltwo{\optdens}^2 \le 2 k \right\}
\end{equation*}
we have the inequality
\begin{equation*}
  \sup_{\optdens \in \linfset}
  \frac{1}{|\packset|}
  \sum_{\packval \in \packset} \varphi_\packval(\optdens)^2
  \le \frac{25 \delta^2}{16 k^2} \cdot 2 \cdot 2k
  = \frac{25}{4} \, \frac{\delta^2}{k}
\end{equation*}
Applying Theorem~\ref{theorem:super-master} completes the proof.


\subsection{Proof of Lemma~\ref{lemma:l1-information-bound}}
\label{appendix:l1-information-bound}

It is no loss of generality to assume the radius $\radius = 1$.
We use the notation of Theorem~\ref{theorem:super-master}, recalling the
linear functionals $\varphi_\packval : L^\infty(\statdomain)
\rightarrow \R$.  Because the set $\statdomain = \{-1, 1\}^d$, we can
identify vectors $\optdens \in L^\infty(\statdomain)$ with vectors $\optdens
\in \R^{2^d}$. Moreover, we have (by construction) that
  \begin{align*}
    \varphi_\packval(\optdens) & = \sum_{\statsample \in \{-1, 1\}^d}
    \optdens(\statsample) \statdensity_\packval(\statsample) -
    \sum_{\statsample \in \{-1, 1\}^d} \optdens(\statsample)
    \meanstatdensity(\statsample) \\ & = \frac{1}{2^d}
    \sum_{\statsample \in \statdomain} \optdens(\statsample)(1 + \delta
    \packval^\top \statsample - 1) = \frac{\delta}{2^d}
    \sum_{\statsample \in \statdomain} \optdens(\statsample) \packval^\top
    \statsample.
  \end{align*}
  For each $\packval \in \packset$, we may construct a vector
  $u_\packval \in \{-1, 1\}^{2^d}$, indexed by $\statsample \in \{-1,
  1\}^d$, with
  \begin{equation*}
    u_\packval(\statsample) = \packval^\top \statsample =
    \begin{cases}
      1 & \mbox{if~} \packval = \pm e_j ~ \mbox{and}~
      \sign(\packval_j) = \sign(\statsample_j) \\ -1 & \mbox{if~}
      \packval = \pm e_j ~ \mbox{and} ~ \sign(\packval_j) \neq
      \sign(\statsample_j).
    \end{cases}
  \end{equation*}
  For $\packval = e_j$, we see that $u_{e_1}, \ldots, u_{e_d}$ are the
  first $d$ columns of the standard Hadamard transform matrix (and
  $u_{-e_j}$ are their negatives).  Then we have that
  $\sum_{\statsample \in \statdomain} \optdens(x) \packval^\top \statsample =
  \optdens^\top u_\packval$, and
  \begin{equation*}
    \varphi_\packval(\optdens) = \optdens^\top u_\packval u_\packval^\top
    \optdens.
  \end{equation*}
  Note also that $u_\packval u_\packval^\top = u_{-\packval}
  u_{-\packval}^\top$, and as a consequence we have
  \begin{equation}
    \label{eqn:l1-packing-quadratic}
    \sum_{\packval \in \packset} \varphi_\packval(\optdens)^2 =
    \frac{\delta^2}{4^d} \optdens^\top \sum_{\packval \in \packset}
    u_\packval u_\packval^\top \optdens = \frac{2\delta^2}{4^d} \optdens^\top
    \sum_{j=1}^d u_{e_j} u_{e_j}^\top \optdens.
  \end{equation}

  But now, studying the quadratic
  form~\eqref{eqn:l1-packing-quadratic}, we note that the vectors
  $u_{e_j}$ are orthogonal. As a consequence, the vectors (up to
  scaling) $u_{e_j}$ are the only eigenvectors corresponding to
  positive eigenvalues of the positive semidefinite matrix
  $\sum_{j=1}^d u_{e_j} u_{e_j}^\top$.  Thus, since the set
  \begin{equation*}
    \linfset = \left\{\optdens \in \R^{2^d} : \linf{\optdens} \le 1 \right\}
    \subset \left\{\optdens \in \R^{2^d} :
    \ltwo{\optdens}^2 \le 2^d \right\},
  \end{equation*}
  we have via an eigenvalue calculation that
  \begin{align*}
    \sup_{\optdens \in \linfset} \sum_{\packval \in \packset}
    \varphi_\packval(\optdens)^2
    & \le \frac{2 \delta^2}{4^d} \sup_{\optdens :
      \ltwo{\optdens}^2 \le 2^d}
    \optdens^\top \sum_{j=1}^d u_{e_j} u_{e_j}^\top \optdens \\
    & = \frac{2 \delta^2}{4^d} \ltwo{u_{e_1}}^4 =
    2 \delta^2
  \end{align*}
  since $\ltwos{u_{e_j}}^2 = 2^d$ for each $j$.  Applying
  Theorem~\ref{theorem:super-master} and
  Corollary~\ref{corollary:super-fano} completes the proof.


\subsection{Proof of Lemma~\ref{lemma:density-kls}}
\label{appendix:density-kls}

This result relies on
Theorem~\ref{theorem:sequential-interactive}, along with a careful argument to
understand the extreme points of $\optdens \in L^\infty([0, 1])$ that we use
when applying the result. First, we take the packing $\packset = \{-1,
1\}^\numderiv$ and densities $f_\packval$ for $\packval \in \packset$ as in
the construction~\eqref{eqn:density-packer}.  Overall, our first step is to
show for the purposes of applying Theorem~\ref{theorem:sequential-interactive},
it is
no loss of generality to identify $\optdens \in L^\infty([0, 1])$ with vectors
$\optdens \in \R^{2 \numbin}$, where $\optdens$ is constant on intervals of
the form $[i/2\numbin, (i + 1)/2\numbin]$. With this identification complete,
we can then provide a bound on the correlation of any $\optdens \in \linfset$
with the densities $f_{\pm j}$ defined in~\eqref{eqn:f-plus-j}, which
completes the proof.

With this outline in mind, let the sets $\MySet_i$, $i \in \{1, 2,
\ldots, 2\numbin\}$, be defined as $\MySet_i =
\openright{(i-1)/2\numbin}{i/2\numbin}$ except that $\MySet_{2\numbin}
= [(2\numbin - 1)/2\numbin, 1]$, so the collection
$\{\MySet_i\}_{i=1}^{2 \numbin}$ forms a partition of the unit
interval $[0, 1]$. By construction of the densities $f_\packval$, the
sign of $f_\packval - 1$ remains constant on each
$\MySet_i$. Let us define (for shorthand) the
linear functionals $\varphi_j : L^\infty([0, 1]) \to \R$
for each $j \in \{1, \ldots, \numbin\}$ via
\begin{equation*}
  \varphi_j(\optdens)
  \defeq \int \optdens (d\statprob_{+j}
  - d\statprob_{-j})
  = \sum_{i = 1}^{2 \numbin}
  \int_{\MySet_i} \optdens(x) (f_{+j}(x) - f_{-j}(x)) dx
  = 2 \int_{\MySet_{2j - 1} \cup \MySet_{2j}}
  \!\!\!\! \optdens(x) g_{\numderiv,j}(x)
  dx,
\end{equation*}
where we recall the definitions~\eqref{eqn:f-plus-j} of the mixture
densities $f_{\pm j} = 1 \pm g_{\numderiv,j}$.
Since the set $\linfset$ from
Theorem~\ref{theorem:sequential-interactive} is compact, convex,
and Hausdorff, the Krein-Milman theorem~\cite[Proposition
  1.2]{Phelps01} guarantees that it is equal to the convex hull of its
extreme points; moreover, since the functionals $\optdens \mapsto
\varphi^2_j(\optdens)$ are convex, the supremum in
Theorem~\ref{theorem:sequential-interactive} must be attained at
the extreme points of $\linfset([0,1])$.  As a consequence, when applying the
divergence bound
\begin{equation}
  \sum_{j=1}^\numbin \left(\dkl{\marginprob_{+j}^n}{\marginprob_{-j}^n}
  + \dkl{\marginprob_{-j}^n}{\marginprob_{+j}^n}\right)
  \le 2 n (e^\diffp - 1)^2
  \sup_{\optdens \in \linfset}
  \sum_{j = 1}^\numbin \varphi_j^2(\optdens),
  \label{eqn:kl-sup-densities}
\end{equation}
we can restrict our attention to $\optdens \in \linfset$ for
which $\optdens(x) \in \{-1, 1\}$.

Now we argue that it is no loss of generality to assume that $\optdens$, when
restricted to $\MySet_i$, is a constant (apart from a measure zero set). Fix
$i \in [2\numbin]$, and assume for the sake of contradiction that there exist
sets $B_i, C_i \subset \MySet_i$ such that $\optdens(B_i) = \{1\}$ and
$\optdens(C_i) = \{-1\}$, while $\lebesgue(B_i) > 0$ and $\lebesgue(C_i) > 0$
where $\lebesgue$ denotes Lebesgue measure.\footnote{For a function $f$ and
  set $A$, the notation $f(A)$ denotes the image $f(A) = \{f(x) \mid x \in
  A\}$.}  We will construct vectors $\optdens_1$ and $\optdens_2 \in \linfset$
and a value $\lambda \in (0, 1)$ such that
\begin{equation*}
  \int_{\MySet_i} \optdens(x) g_{\numderiv,j}(x)
  d x =
  \lambda \int_{\MySet_i} \optdens_1(x) g_{\numderiv,j}(x) dx
  + (1 - \lambda) \int_{\MySet_i} \optdens_2(x) g_{\numderiv,j}(x) dx
\end{equation*}
simultaneously for all $j \in [\numbin]$, while
on $\MySet_i^c = [0, 1] \setminus \MySet_i$, we will have the equivalence
\begin{equation*}
  \left.\optdens_1\right|_{\MySet_i^c}
  \equiv \left.\optdens_2 \right|_{\MySet_i^c}
  \equiv \left.\optdens\right|_{\MySet_i^c}.
\end{equation*}
Indeed, set $\optdens_1(\MySet_i) = \{1\}$ and $\optdens_2(\MySet_i) =
\{-1\}$, otherwise setting $\optdens_1(x) = \optdens_2(x) = \optdens(x)$ for
$x \not \in \MySet_i$. For the unique index $j \in [\numbin]$ such that
$[(j-1)/\numbin, j/\numbin] \supset D_i$, we define
\begin{equation*}
  \lambda \defeq \frac{\int_{B_i} g_{\numderiv,j}(x) dx}{
    \int_{\MySet_i} g_{\numderiv,j}(x) dx}
  ~~~ \mbox{so} ~~~
  1 - \lambda = \frac{\int_{C_i} g_{\numderiv,j}(x) dx}{
    \int_{\MySet_i} g_{\numderiv,j}(x) dx}.
\end{equation*}
By the construction of the function $g_\numderiv$, the functions
$g_{\numderiv,j}$ do not change signs on $\MySet_i$, and the absolute
continuity conditions on $g_\numderiv$ specified in
equation~\eqref{eqn:absolute-continuity-bumps} guarantee $1 > \lambda
> 0$, since $\lebesgue(B_i) > 0$ and $\lebesgue(C_i) > 0$. 
We thus find that for any $j \in [\numbin]$,
\begin{align*}
  \int_{\MySet_i} \optdens(x)  g_{\numderiv,j}(x) dx
  & = \int_{B_i} \optdens_1(x) g_{\numderiv,j}(x) dx
  + \int_{C_i} \optdens_2(x) g_{\numderiv,j}(x) dx \\
  & = \int_{B_i} g_{\numderiv,j}(x)dx
  - \int_{C_i} g_{\numderiv,j}(x) dx
  = \lambda \int_{\MySet_i} g_{\numderiv,j}(x)dx
  - (1 - \lambda) \int_{\MySet_i} g_{\numderiv,j}(x)dx \\
  & = \lambda \int \optdens_1(x)
  g_{\numderiv,j}(x)dx + (1 - \lambda) \int \optdens_2(x)
  g_{\numderiv,j}(x)dx.
\end{align*}
(Notably, for $j$ such that $g_{\numderiv,j}$ is identically 0 on $\MySet_i$,
this equality is trivial.) By linearity and the strong convexity of the
function $x \mapsto x^2$, then, we find that
for sets $E_j \defeq \MySet_{2j - 1} \cup \MySet_{2j}$,
\begin{align*}
  \sum_{j=1}^\numbin \varphi_j^2(\optdens)
  & = \sum_{j=1}^\numbin
  \left(\int_{E_j}
  \optdens(x) g_{\numderiv,j}(x) dx\right)^2 \\
  & < \lambda \sum_{j = 1}^\numbin
  \left(\int_{E_j} \optdens_1(x)
  g_{\numderiv,j}(x) dx\right)^2
  + (1 - \lambda) \sum_{\packval \in
    \packset} \left(\int_{E_j} \optdens_2(x)
  g_{\numderiv,j}(x)dx \right)^2.
\end{align*}
Thus one of the densities $\optdens_1$ or $\optdens_2$ must have a
larger objective value than $\optdens$. This is our desired
contradiction, which shows that (up to
measure zero sets) any $\optdens$ attaining the supremum in the
information bound~\eqref{eqn:kl-sup-densities} must be
constant on each of the $\MySet_i$.

\newcommand{\linfsetfinite}[1]{\mc{B}_{1,#1}} 

Having shown that $\optdens$ is constant on each of the intervals
$\MySet_i$, we conclude that the
supremum~\eqref{eqn:kl-sup-densities} can be reduced to a
finite-dimensional problem over the subset
\begin{align*}
  \linfsetfinite{2\numbin} \defeq \left\{u \in \R^{2\numbin}
  \mid \linf{u} \le 1 \right\}
\end{align*}
of $\R^{2\numbin}$.
In terms of this subset, the supremum~\eqref{eqn:kl-sup-densities}
can be rewritten as the the upper bound
\begin{align*}
  \sup_{\optdens \in \linfset} \sum_{j=1}^\numbin
  \varphi_j(\optdens)^2
  & \le \sup_{\optdens \in
  \linfsetfinite{2 \numbin}}
  \sum_{j=1}^\numbin
  \bigg(\optdens_{2j-1}
  \int_{\MySet_{2j - 1}} g_{\numderiv,j}(x) dx
  + \optdens_{2j}
  \int_{\MySet_2j} g_{\numderiv,j}(x) dx \bigg)^2
\end{align*}
By construction of the function $g_\numderiv$, we have the
equality
\begin{equation*}
  \int_{\MySet_{2j-1}} g_{\numderiv,j}(x) dx
  = -\int_{\MySet_{2j}} g_{\numderiv,j}(x) dx
  = \int_0^{\frac{1}{2 \numbin}} g_{\numderiv,1}(x) dx
  = \int_0^{\frac{1}{2 \numbin}}
  \frac{1}{\numbin^\numderiv} g_\numderiv(\numbin x) dx
  = \frac{c_{1/2}}{\numbin^{\numderiv + 1}}.
\end{equation*}
This implies that
\begin{align}
  \lefteqn{\frac{1}{2e^\diffp (e^\diffp - 1)^2 n}
    \sum_{j=1}^\numbin \left(\dkl{\marginprob_{+j}^n}{\marginprob_{-j}^n}
    + \dkl{\marginprob_{+j}^n}{\marginprob_{-j}^n}\right)
    \le 
    \sup_{\optdens \in \linfset} \sum_{j=1}^\numbin
    \varphi_j(\optdens)^2 } \nonumber \\
  & \le \sup_{\optdens \in \linfsetfinite{2
      \numbin}} \sum_{j=1}^\numbin
  \left(\frac{c_{1/2}}{\numbin^{\numderiv + 1}} \optdens^\top
  (e_{2j - 1} - e_{2j})\right)^2
  = \frac{c_{1/2}^2}{\numbin^{2
      \numderiv + 2}} \sup_{\optdens \in \linfsetfinite{2 \numbin}}
  \optdens^\top \sum_{j=1}^\numbin
  (e_{2j - 1} - e_{2j})(e_{2j - 1} - e_{2j})^\top
  \optdens,
  \label{eqn:basis-density-bound}
\end{align}
where $e_j \in \R^{2 \numbin}$ denotes the $j$th standard basis vector.
Rewriting this using the Kronecker product $\otimes$, we have
\begin{equation*}
  \sum_{j=1}^\numbin (e_{2j - 1} - e_{2j})(e_{2j - 1} - e_{2j})^\top
  = I_{\numbin \times \numbin} \otimes \left[\begin{matrix} 1 & -1 \\
      -1 & 1 \end{matrix}\right]
  \preceq 2 I_{2\numbin \times 2 \numbin}.
\end{equation*}
Combining this bound with our
inequality~\eqref{eqn:basis-density-bound}, we obtain
\begin{align*}
  \sum_{j=1}^\numbin \left(\dkl{\marginprob_{+j}^n}{\marginprob_{-j}^n}
  + \dkl{\marginprob_{+j}^n}{\marginprob_{-j}^n}\right)
  \le
  4 n (e^\diffp - 1)^2
  \frac{c_{1/2}^2}{\numbin^{2 \numderiv + 2}}
  \sup_{\optdens \in
    \linfsetfinite{2 \numbin}} \ltwo{\optdens}^2
  = 4 c_{1/2}^2
  \frac{n (e^\diffp - 1)^2}{\numbin^{2 \numderiv + 1}}.
\end{align*}

\section{Technical arguments}

In this appendix, we collect proofs of technical lemmas and results
needed for completeness.

\subsection{Proof of Lemma~\ref{lemma:sharp-assouad}}
\label{sec:proof-sharp-assouad}

Fix an (arbitrary) estimator $\what{\optvar}$. By
assumption~\eqref{eqn:risk-separation}, we have
\begin{equation*}
  \Phi(\metric(\optvar, \optvar(\statprob_\packval)))
  \ge 2 \delta \sum_{j=1}^d \indic{[\maptocube(\optvar)]_j
    \neq \packval_j}.
\end{equation*}
Taking expectations, we see that
\begin{align*}
  \sup_{\statprob \in \mc{\statprob}} \E_\statprob 
  \left[\Phi(\metric(\what{\optvar}(\channelrv_1, \ldots, \channelrv_n),
    \optvar(\statprob)))\right]
  & \ge 
  \max_{\packval \in \packset} \E_{\statprob_\packval}
  \left[\Phi(\metric(\what{\optvar}(\channelrv_1, \ldots, \channelrv_n),
    \optvar_\packval))\right] \\
  & \ge \frac{1}{|\packset|} \sum_{\packval \in \packset}
  \E_{\statprob_\packval}
  \left[\Phi(\metric(\what{\optvar}(\channelrv_1, \ldots, \channelrv_n),
    \optvar_\packval))\right] \\
  & \ge \frac{1}{|\packset|} \sum_{\packval \in \packset}
  2 \delta \sum_{j=1}^d \E_{\statprob_\packval}\left[
    \indic{[\psi(\what{\optvar})]_j \neq \packval_j}\right]
\end{align*}
as the average is smaller than the maximum of a set and using the separation
assumption~\eqref{eqn:risk-separation}.  Recalling the
definition~\eqref{eqn:paired-mixtures} of the mixtures $\statprob_{\pm j}$, we
swap the summation orders to see that
\begin{align*}
  \frac{1}{|\packset|} \sum_{\packval \in \packset}
  \statprob_{\packval}\left([\maptocube(\what{\optvar})]_j
  \neq \packval_j\right)
  & = \frac{1}{|\packset|} \sum_{\packval : \packval_j = 1}
  \statprob_{\packval}\left([\maptocube(\what{\optvar})]_j
  \neq \packval_j\right)
  + \frac{1}{|\packset|} \sum_{\packval : \packval_j = -1}
  \statprob_{\packval}\left([\maptocube(\what{\optvar})]_j
  \neq \packval_j\right)
  \\
  & = \half \statprob_{+j}
  \left([\maptocube(\what{\optvar})]_j \neq \packval_j\right)
  + \half \statprob_{-j}
  \left([\maptocube(\what{\optvar})]_j \neq \packval_j\right).
\end{align*}
This gives the statement claimed in the lemma, while taking an infimum over
all testing procedures $\test : \channeldomain^n \to \{-1, +1\}$ gives the
claim~\eqref{eqn:sharp-assouad}.

\subsection{Proof of unbiasedness for
  sampling strategy~\eqref{eqn:ltwo-sampling}}
\label{appendix:ltwo-sampling}

We compute the expectation of a random variable $\channelrv$
sampled according to the strategy~\eqref{eqn:ltwo-sampling}, i.e.\
we compute $\E[\channelrv \mid v]$ for a vector $v \in \R^d$. By
scaling, it is no loss of generality to assume that
$\ltwo{v} = 1$, and using the rotational symmetry of the $\ell_2$-ball, we see
it is no loss of generality to assume that $v = e_1$, the first standard basis
vector.
  
Let the function $s_d$ denote the surface area of the sphere in
$\R^d$, so that
\begin{equation*}
  s_d(r) = \frac{d \pi^{d/2}}{\Gamma(d/2 + 1)} r^{d-1}
\end{equation*}
is the surface area of the sphere of radius $r$. (We use $s_d$ as a
shorthand for $s_d(1)$ when convenient.) Then for a random variable
$W$ sampled uniformly from the half of the $\ell_2$-ball with
first coordinate $W_1 \ge
0$, symmetry implies that by integrating over the radii of the ball,
\begin{equation*}
  \E[W] = e_1 \frac{2}{s_d} \int_0^1 s_{d-1}\left(\sqrt{1 -
    r^2}\right) r dr.
\end{equation*}
Making the change of variables to spherical coordinates (we use $\phi$
as the angle), we have
\begin{equation*}
  \frac{2}{s_d} \int_0^1 s_{d-1}\left(\sqrt{1 - r^2}\right) r dr =
  \frac{2}{s_{d}} \int_0^{\pi/2} s_{d-1}\left(\cos \phi \right) \sin
  \phi\, d\phi = \frac{2 s_{d-1}}{s_d} \int_0^{\pi/2} \cos^{d - 2}(\phi)
  \sin(\phi) \, d\phi.
\end{equation*}
Noting that $\frac{d}{d \phi} \cos^{d - 1}(\phi) = -(d - 1) \cos^{d -
  2}(\phi) \sin(\phi)$, we obtain
\begin{equation*}
  \frac{2 s_{d-1}}{s_d} \int_0^{\pi/2} \cos^{d - 2}(\phi) \sin(\phi)
  \,d\phi = \left.-\frac{\cos^{d - 1}(\phi)}{d - 1}\right|_0^{\pi/2} =
  \frac{1}{d - 1},
\end{equation*}
or that
\begin{equation}
  \E[W] = e_1 \frac{(d - 1) \pi^{\frac{d - 1}{2}} \Gamma(\frac{d}{2} +
    1)}{ d \pi^{\frac{d}{2}} \Gamma(\frac{d - 1}{2} + 1)} \frac{1}{d -
    1} = e_1 \underbrace{ \frac{\Gamma(\frac{d}{2} + 1)}{\sqrt{\pi} d
      \Gamma(\frac{d - 1}{2} + 1)}}_{\eqdef c_d},
  \label{eqn:ltwo-halfspace-expectation}
\end{equation}
where we define the constant $c_d$ to be the final ratio.

Allowing again $\ltwo{v} \le \radius$, with the
expression~\eqref{eqn:ltwo-halfspace-expectation}, we see that for our
sampling strategy for $\channelrv$, we have
\begin{equation*}
  \E[\channelrv \mid v] = v \frac{\sbound}{\radius} c_d \left(
  \frac{e^\diffp}{e^\diffp + 1} - \frac{1}{e^\diffp + 1} \right) =
  \frac{\sbound}{\radius} c_d \frac{e^\diffp - 1}{e^\diffp + 1}.
\end{equation*}
Consequently, the choice
\begin{equation*}
  \sbound = \frac{e^\diffp + 1}{e^\diffp - 1} \frac{\radius}{c_d} =
  \frac{e^\diffp + 1}{e^\diffp - 1} \frac{\radius \sqrt{\pi} d
    \Gamma(\frac{d - 1}{2} + 1)}{ \Gamma(\frac{d}{2} + 1)}
\end{equation*}
yields $\E[\channelrv \mid v] = v$. Moreover, we have
\begin{equation*}
  \ltwo{\channelrv} = \sbound \le \radius \frac{e^\diffp + 1}{e^\diffp - 1}
  \frac{3 \sqrt{\pi} \sqrt{d}}{2}
\end{equation*}
by Stirling's approximation to the $\Gamma$-function. By noting that
$(e^\diffp + 1) / (e^\diffp - 1) \le 3 / \diffp$ for $\diffp \le 1$,
we see that $\ltwo{\channelrv} \le 8 \radius \sqrt{d} / \diffp$.

\section{Effects of differential privacy in non-compact spaces}
\label{AppPathology}

In this appendix, we present a somewhat pathological example that
demonstrates the effects of differential privacy in non-compact
spaces.  Let us assume only that $\optvar \in \R$ and $\diffp <
\infty$, and we denote $\mc{P}_\optvar$ to be the collection of
probability measures with variance $1$ having $\optvar$ as a mean. In
contrast to the non-private case, where the risk of the sample mean
scales as $1/n$, we obtain
\begin{equation}
  \label{eqn:infinite-minimax}
  \minimax_n(\R, (\cdot)^2, \diffp) = \infty
\end{equation}
for all $n \in \N$. To see this, consider the Fano inequality
version~\eqref{eqn:fano}.  Fix
$\delta > 0$ and choose $\{\theta_1 = 0, \theta_2 = 2\delta, \ldots,
\theta_N = 2N \delta\}$ where $N = N(\delta, n) =
\max\{\ceil{\exp(64(e^\diffp - 1)^2 n)}, 2^4\}$. Then by applying
Corollary~\ref{corollary:idiot-fano}, we have for $\packset = [N]$
that
\begin{equation*}
  \minimax_n(\R, (\cdot)^2, \diffp) \ge \delta^2 \left(1 - \frac{4
    (e^\diffp - 1)^2 n \sum_{\packval, \altpackval \in \packset}
    \tvnorm{\statprob_\packval - \statprob_{\altpackval}}^2 /
    |\packset|^2 + \log 2}{\log N(\delta, n)}\right).
\end{equation*}
We have $\tvnorm{\statprob_\packval - \statprob_{\altpackval}} \le 1$
for any two distributions $\statprob_\packval$ and
$\statprob_{\altpackval}$, which implies
\begin{equation*}
  \minimax_n(\R, (\cdot)^2, \diffp) \ge \delta^2 \left(1 - \frac{16
    (e^\diffp - 1)^2 n + \log 2}{\log N(\delta, n)}\right) \ge
  \delta^2 \left(1 - \frac{1}{2}\right) = \half \delta^2.
\end{equation*}
Since $\delta > 0$ was arbitrary, this proves the infinite minimax
risk bound~\eqref{eqn:infinite-minimax}.  The construction to
achieve~\eqref{eqn:infinite-minimax} is somewhat contrived, but it
suggests that care is needed when designing differentially private
inference procedures, and shows that even in cases when it is possible
to attain a parametric rate of convergence, there may be no (locally)
differentially private inference procedure.

\fi


\ifdefined\useaosstyle
\else
\setlength{\bibsep}{3pt}
\fi
\bibliographystyle{abbrvnat}
\bibliography{bib}

\begin{thebibliography}{53}
\providecommand{\natexlab}[1]{#1}
\providecommand{\url}[1]{\texttt{#1}}
\expandafter\ifx\csname urlstyle\endcsname\relax
  \providecommand{\doi}[1]{doi: #1}\else
  \providecommand{\doi}{doi: \begingroup \urlstyle{rm}\Url}\fi

\bibitem[Ahlswede and Winter(2002)]{AhlswedeWi02}
R.~Ahlswede and A.~Winter.
\newblock Strong converse for identification via quantum channels.
\newblock \emph{IEEE Transactions on Information Theory}, 48\penalty0
  (3):\penalty0 569--579, March 2002.

\bibitem[Alon and Spencer(2000)]{AlonSp00}
N.~Alon and J.~H. Spencer.
\newblock \emph{The Probabilistic Method}.
\newblock Wiley-Interscience, second edition, 2000.

\bibitem[Anantharam et~al.(2013)Anantharam, Gohari, Kamath, and
  Nair]{AnantharamGoKaNa13}
V.~Anantharam, A.~Gohari, S.~Kamath, and C.~Nair.
\newblock On maximal correlation, hypercontractivity, and the data processing
  inequality studied by {E}rkip and {C}over.
\newblock \emph{arXiv:1304.6133 [cs.IT]}, 2013.
\newblock URL \url{http://arxiv.org/abs/1304.6133}.

\bibitem[Arias-Castro et~al.(2013)Arias-Castro, Cand\'es, and
  Davenport]{Arias-CastroCaDa13}
E.~Arias-Castro, E.~Cand\'es, and M.~Davenport.
\newblock On the fundamental limits of adaptive sensing.
\newblock \emph{IEEE Transactions on Information Theory}, 59\penalty0
  (1):\penalty0 472--481, 2013.

\bibitem[Assouad(1983)]{Assouad83}
P.~Assouad.
\newblock Deux remarques sur l'estimation.
\newblock \emph{C. R. Academy Scientifique Paris S\'eries I Mathematics},
  296\penalty0 (23):\penalty0 1021--1024, 1983.

\bibitem[Barak et~al.(2007)Barak, Chaudhuri, Dwork, Kale, McSherry, and
  Talwar]{BarakChDwKaMcTa07}
B.~Barak, K.~Chaudhuri, C.~Dwork, S.~Kale, F.~McSherry, and K.~Talwar.
\newblock Privacy, accuracy, and consistency too: A holistic solution to
  contingency table release.
\newblock In \emph{Proceedings of the 26th ACM Symposium on Principles of
  Database Systems}, 2007.

\bibitem[Beimel et~al.(2008)Beimel, Nissim, and Omri]{BeimelNiOm08}
A.~Beimel, K.~Nissim, and E.~Omri.
\newblock Distributed private data analysis: Simultaneously solving how and
  what.
\newblock In \emph{Advances in Cryptology}, volume 5157 of \emph{Lecture Notes
  in Computer Science}, pages 451--468. Springer, 2008.

\bibitem[Beimel et~al.(2010)Beimel, Kasiviswanathan, and Nissim]{BeimelKaNi10}
A.~Beimel, S.~P. Kasiviswanathan, and K.~Nissim.
\newblock Bounds on the sample complexity for private learning and private data
  release.
\newblock In \emph{Proceedings of the 7th Theory of Cryptography Conference},
  pages 437--454, 2010.

\bibitem[Birg\'e(1983)]{Birge83}
L.~Birg\'e.
\newblock Approximation dans les espaces m\'etriques et th\'eorie de
  l'estimation.
\newblock \emph{Zeitschrift f\"ur Wahrscheinlichkeitstheorie und verwebte
  Gebiet}, 65:\penalty0 181--238, 1983.

\bibitem[Blum et~al.(2008)Blum, Ligett, and Roth]{BlumLiRo08}
A.~Blum, K.~Ligett, and A.~Roth.
\newblock A learning theory approach to non-interactive database privacy.
\newblock In \emph{Proceedings of the Fourtieth Annual ACM Symposium on the
  Theory of Computing}, 2008.

\bibitem[Brucker(1984)]{Brucker84}
P.~Brucker.
\newblock An {$O(n)$} algorithm for quadratic knapsack problems.
\newblock \emph{Operations Research Letters}, 3\penalty0 (3):\penalty0
  163--166, 1984.

\bibitem[Buldygin and Kozachenko(2000)]{BuldyginKo00}
V.~Buldygin and Y.~Kozachenko.
\newblock \emph{Metric Characterization of Random Variables and Random
  Processes}, volume 188 of \emph{Translations of Mathematical Monographs}.
\newblock American Mathematical Society, 2000.

\bibitem[Carroll and Hall(1988)]{CarrollHa88}
R.~Carroll and P.~Hall.
\newblock Optimal rates of convergence for deconvolving a density.
\newblock \emph{Journal of the American Statistical Association}, 83\penalty0
  (404):\penalty0 1184--1186, 1988.

\bibitem[Carroll et~al.(2006)Carroll, Ruppert, Stefanski, and
  Crainiceanu]{CarrollRuStCr06}
R.~Carroll, D.~Ruppert, L.~Stefanski, and C.~Crainiceanu.
\newblock \emph{Measurement Error in Nonlinear Models: A Modern Perspective}.
\newblock Chapman and Hall, second edition, 2006.

\bibitem[Chaudhuri and Hsu(2012)]{ChaudhuriHs12}
K.~Chaudhuri and D.~Hsu.
\newblock Convergence rates for differentially private statistical estimation.
\newblock In \emph{Proceedings of the 29th International Conference on Machine
  Learning}, 2012.

\bibitem[Chaudhuri et~al.(2011)Chaudhuri, Monteleoni, and
  Sarwate]{ChaudhuriMoSa11}
K.~Chaudhuri, C.~Monteleoni, and A.~D. Sarwate.
\newblock Differentially private empirical risk minimization.
\newblock \emph{Journal of Machine Learning Research}, 12:\penalty0 1069--1109,
  2011.

\bibitem[Cover and Thomas(2006)]{CoverTh06}
T.~M. Cover and J.~A. Thomas.
\newblock \emph{Elements of Information Theory, Second Edition}.
\newblock Wiley, 2006.

\bibitem[De(2012)]{De12}
A.~De.
\newblock Lower bounds in differential privacy.
\newblock In \emph{Proceedings of the Ninth Theory of Cryptography Conference},
  2012.
\newblock URL \url{http://arxiv.org/abs/1107.2183}.

\bibitem[Duchi et~al.(2012)Duchi, Jordan, and Wainwright]{DuchiJoWa12}
J.~C. Duchi, M.~I. Jordan, and M.~J. Wainwright.
\newblock Privacy aware learning.
\newblock \emph{arXiv:1210.2085 [stat.ML]}, 2012.
\newblock URL \url{http://arxiv.org/abs/1210.2085}.

\bibitem[Duncan and Lambert(1986)]{DuncanLa86}
G.~T. Duncan and D.~Lambert.
\newblock Disclosure-limited data dissemination.
\newblock \emph{Journal of the American Statistical Association}, 81\penalty0
  (393):\penalty0 10--18, 1986.

\bibitem[Duncan and Lambert(1989)]{DuncanLa89}
G.~T. Duncan and D.~Lambert.
\newblock The risk of disclosure for microdata.
\newblock \emph{Journal of Business and Economic Statistics}, 7\penalty0
  (2):\penalty0 207--217, 1989.

\bibitem[Dwork and Lei(2009)]{DworkLe09}
C.~Dwork and J.~Lei.
\newblock Differential privacy and robust statistics.
\newblock In \emph{Proceedings of the Fourty-First Annual ACM Symposium on the
  Theory of Computing}, 2009.

\bibitem[Dwork et~al.(2006{\natexlab{a}})Dwork, Kenthapadi, McSherry, Mironov,
  and Naor]{DworkKeMcMiNa06}
C.~Dwork, K.~Kenthapadi, F.~McSherry, I.~Mironov, and M.~Naor.
\newblock Our data, ourselves: Privacy via distributed noise generation.
\newblock In \emph{Advances in Cryptology (EUROCRYPT 2006)},
  2006{\natexlab{a}}.

\bibitem[Dwork et~al.(2006{\natexlab{b}})Dwork, McSherry, Nissim, and
  Smith]{DworkMcNiSm06}
C.~Dwork, F.~McSherry, K.~Nissim, and A.~Smith.
\newblock Calibrating noise to sensitivity in private data analysis.
\newblock In \emph{Proceedings of the 3rd Theory of Cryptography Conference},
  pages 265--284, 2006{\natexlab{b}}.

\bibitem[Dwork et~al.(2010)Dwork, Rothblum, and Vadhan]{DworkRoVa10}
C.~Dwork, G.~N. Rothblum, and S.~P. Vadhan.
\newblock Boosting and differential privacy.
\newblock In \emph{51st Annual Symposium on Foundations of Computer Science},
  pages 51--60, 2010.

\bibitem[Efromovich(1999)]{Efromovich99}
S.~Efromovich.
\newblock \emph{Nonparametric Curve Estimation: Methods, Theory, and
  Applications}.
\newblock Springer-Verlag, 1999.

\bibitem[Evfimievski et~al.(2003)Evfimievski, Gehrke, and
  Srikant]{EvfimievskiGeSr03}
A.~V. Evfimievski, J.~Gehrke, and R.~Srikant.
\newblock Limiting privacy breaches in privacy preserving data mining.
\newblock In \emph{Proceedings of the Twenty-Second Symposium on Principles of
  Database Systems}, pages 211--222, 2003.

\bibitem[Fienberg et~al.(1998)Fienberg, Makov, and Steele]{FienbergMaSt98}
S.~E. Fienberg, U.~E. Makov, and R.~J. Steele.
\newblock Disclosure limitation using perturbation and related methods for
  categorical data.
\newblock \emph{Journal of Official Statistics}, 14\penalty0 (4):\penalty0
  485--502, 1998.

\bibitem[Fienberg et~al.(2010)Fienberg, Rinaldo, and Yang]{FienbergRiYa10}
S.~E. Fienberg, A.~Rinaldo, and X.~Yang.
\newblock Differential privacy and the risk-utility tradeoff for
  multi-dimensional contingency tables.
\newblock In \emph{The International Conference on Privacy in Statistical
  Databases}, 2010.

\bibitem[Ganta et~al.(2008)Ganta, Kasiviswanathan, and Smith]{GantaKaSm08}
S.~R. Ganta, S.~Kasiviswanathan, and A.~Smith.
\newblock Composition attacks and auxiliary information in data privacy.
\newblock In \emph{Proceedings of the 14th ACM SIGKDD Conference on Knowledge
  and Data Discovery (KDD)}, 2008.

\bibitem[Gleser(1981)]{Gleser81}
L.~J. Gleser.
\newblock Estimation in a multivariate ``errors in variables'' regression
  model: large sample results.
\newblock \emph{Annals of Statistics}, 9\penalty0 (1):\penalty0 24--44, 1981.

\bibitem[Gray(1990)]{Gray90}
R.~M. Gray.
\newblock \emph{Entropy and Information Theory}.
\newblock Springer, 1990.

\bibitem[Hall et~al.(2011)Hall, Rinaldo, and Wasserman]{HallRiWa11}
R.~Hall, A.~Rinaldo, and L.~Wasserman.
\newblock Random differential privacy.
\newblock \emph{arXiv:1112.2680 [stat.ME]}, 2011.
\newblock URL \url{http://arxiv.org/abs/1112.2680}.

\bibitem[Hardt and Rothblum(2010)]{HardtRo10}
M.~Hardt and G.~N. Rothblum.
\newblock A multiplicative weights mechanism for privacy-preserving data
  analysis.
\newblock In \emph{51st Annual Symposium on Foundations of Computer Science},
  2010.

\bibitem[Hardt and Talwar(2010)]{HardtTa10}
M.~Hardt and K.~Talwar.
\newblock On the geometry of differential privacy.
\newblock In \emph{Proceedings of the Fourty-Second Annual ACM Symposium on the
  Theory of Computing}, pages 705--714, 2010.
\newblock URL \url{http://arxiv.org/abs/0907.3754}.

\bibitem[Has'minskii(1978)]{Hasminskii78}
R.~Z. Has'minskii.
\newblock A lower bound on the risks of nonparametric estimates of densities in
  the uniform metric.
\newblock \emph{Theory of Probability and Applications}, 23:\penalty0 794--798,
  1978.

\bibitem[Kasiviswanathan et~al.(2011)Kasiviswanathan, Lee, Nissim,
  Raskhodnikova, and Smith]{KasiviswanathanLeNiRaSm11}
S.~P. Kasiviswanathan, H.~K. Lee, K.~Nissim, S.~Raskhodnikova, and A.~Smith.
\newblock What can we learn privately?
\newblock \emph{SIAM Journal on Computing}, 40\penalty0 (3):\penalty0 793--826,
  2011.

\bibitem[Kearns(1998)]{Kearns98}
M.~Kearns.
\newblock Efficient noise-tolerant learning from statistical queries.
\newblock \emph{Journal of the Association for Computing Machinery},
  45\penalty0 (6):\penalty0 983--1006, 1998.

\bibitem[Ling and Li(2009)]{LiangLi09}
H.~Ling and R.~Li.
\newblock Variable selection for partially linear models with measurement
  errors.
\newblock \emph{Journal of the American Statistical Association}, 104\penalty0
  (485):\penalty0 234--248, 2009.

\bibitem[Loh and Wainwright(2012)]{LohWa12}
P.-L. Loh and M.~J. Wainwright.
\newblock High-dimensional regression with noisy and missing data: provable
  guarantees with nonconvexity.
\newblock \emph{Annals of Statistics}, 40\penalty0 (3):\penalty0 1637--1664,
  2012.

\bibitem[Ma and Li(2010)]{MaLi10}
Y.~Ma and R.~Li.
\newblock Variable selection in measurement error models.
\newblock \emph{Bernoulli}, 16\penalty0 (1):\penalty0 274--300, 2010.

\bibitem[Mackey et~al.(2012)Mackey, Jordan, Chen, Farrell, and
  Tropp]{MackeyJoChFaTr12}
L.~W. Mackey, M.~I. Jordan, R.~Y. Chen, B.~Farrell, and J.~A. Tropp.
\newblock Matrix concentration inequalities via the method of exchangeable
  pairs.
\newblock \emph{arXiv:1201.6002 [math.PR]}, 2012.
\newblock URL \url{http://arxiv.org/abs/1201.6002}.

\bibitem[McGregor et~al.(2010)McGregor, Mironov, Pitassi, Reingold, Talwar, and
  Vadhan]{McGregorMiPiReTaVa10}
A.~McGregor, I.~Mironov, T.~Pitassi, O.~Reingold, K.~Talwar, and S.~Vadhan.
\newblock The limits of two-party differential privacy.
\newblock In \emph{51st Annual Symposium on Foundations of Computer Science},
  2010.

\bibitem[Negahban et~al.(2012)Negahban, Ravikumar, Wainwright, and
  Yu]{NegahbanRaWaYu12}
S.~Negahban, P.~Ravikumar, M.~Wainwright, and B.~Yu.
\newblock A unified framework for high-dimensional analysis of ${M}$-estimators
  with decomposable regularizers.
\newblock \emph{Statistical Science}, 27\penalty0 (4):\penalty0 538--557, 2012.

\bibitem[Phelps(2001)]{Phelps01}
R.~R. Phelps.
\newblock \emph{Lectures on Choquet's Theorem, Second Edition}.
\newblock Springer, 2001.

\bibitem[Rubinstein et~al.(2012)Rubinstein, Bartlett, Huang, and
  Taft]{RubinsteinBaHuTa12}
B.~I.~P. Rubinstein, P.~L. Bartlett, L.~Huang, and N.~Taft.
\newblock Learning in a large function space: privacy-preserving mechanisms for
  {SVM} learning.
\newblock \emph{Journal of Privacy and Confidentiality}, 4\penalty0
  (1):\penalty0 65--100, 2012.

\bibitem[Scott(1979)]{Scott79}
D.~Scott.
\newblock On optimal and data-based histograms.
\newblock \emph{Biometrika}, 66\penalty0 (3):\penalty0 605--610, 1979.

\bibitem[Smith(2011)]{Smith11}
A.~Smith.
\newblock Privacy-preserving statistical estimation with optimal convergence
  rates.
\newblock In \emph{Proceedings of the Fourty-Third Annual ACM Symposium on the
  Theory of Computing}, 2011.

\bibitem[Tsybakov(2009)]{Tsybakov09}
A.~B. Tsybakov.
\newblock \emph{Introduction to Nonparametric Estimation}.
\newblock Springer, 2009.

\bibitem[Warner(1965)]{Warner65}
S.~Warner.
\newblock Randomized response: a survey technique for eliminating evasive
  answer bias.
\newblock \emph{Journal of the American Statistical Association}, 60\penalty0
  (309):\penalty0 63--69, 1965.

\bibitem[Wasserman and Zhou(2010)]{WassermanZh10}
L.~Wasserman and S.~Zhou.
\newblock A statistical framework for differential privacy.
\newblock \emph{Journal of the American Statistical Association}, 105\penalty0
  (489):\penalty0 375--389, 2010.

\bibitem[Yang and Barron(1999)]{YangBa99}
Y.~Yang and A.~Barron.
\newblock Information-theoretic determination of minimax rates of convergence.
\newblock \emph{Annals of Statistics}, 27\penalty0 (5):\penalty0 1564--1599,
  1999.

\bibitem[Yu(1997)]{Yu97}
B.~Yu.
\newblock Assouad, {F}ano, and {L}e {C}am.
\newblock In \emph{Festschrift for Lucien Le Cam}, pages 423--435.
  Springer-Verlag, 1997.

\end{thebibliography}

\end{document}